%% file: paper_LMDG_SISC.tex
\pdfoutput=1
\documentclass[onefignum,onetabnum]{siamart190516}
\renewcommand{\theequation}{\thesection.\arabic{equation}}
\renewcommand{\thefigure}{\thesection.\arabic{figure}}
\renewcommand{\thetable}{\thesection.\arabic{table}}

\usepackage{amsmath} 
\usepackage{amsfonts}
\usepackage{amssymb}
\usepackage{multirow}
\usepackage{stmaryrd}
\usepackage{subcaption}
\usepackage{makecell}
\usepackage{cite}
\usepackage{boldline}
\usepackage{graphics}
\usepackage{epstopdf}
\usepackage{mathtools}

\newcommand{\vphi}{\varphi} 
\newcommand{\veps}{\varepsilon}

\newcommand{\tspan}{\mathrm{span}}

\newcommand{\cC}{\mathcal{C}}
\newcommand{\cS}{\mathcal{S}}
\newcommand{\cD}{\overline{\mathcal{V}}}
\newcommand{\cF}{\mathcal{F}}
\newcommand{\cV}{\mathcal{V}}
\newcommand{\cW}{\mathcal{W}}

\newcommand{\cJ}{\mathcal{J}}

\newcommand{\cO}{\mathcal{O}}

\newcommand{\cT}{\mathcal{T}}
\newcommand{\bA}{\mathbf{A}}
\newcommand{\bB}{\mathbf{B}}
\newcommand{\bC}{\mathbf{C}}
\newcommand{\bD}{\mathbf{D}}

\newcommand{\bI}{\mathbf{I}}
\newcommand{\bK}{\mathbf{K}}
\newcommand{\bL}{\mathbf{L}}
\newcommand{\bM}{\mathbf{M}}
\newcommand{\bP}{\mathbf{P}}
\newcommand{\bQ}{\mathbf{Q}}
\newcommand{\bR}{\mathbf{R}}
\newcommand{\bS}{\mathbf{S}}
\newcommand{\bU}{\mathbf{U}}
\newcommand{\bV}{\mathbf{V}}
\newcommand{\bX}{\mathbf{X}}

\newcommand{\br}{\mathbf{r}}

\newcommand{\bPsi}{\mathbf{\Psi}}
\newcommand{\bPhi}{\mathbf{\Phi}}
\newcommand{\bSigma}{\mathbf{\Sigma}}

\newcommand{\sign}{\text{sign}}
\newcommand{\quand}{\quad \text{and} \quad}

\newcommand{\sig}[1]{\sigma_{\mathrm{#1}}}
\newcommand{\zero}{\mathrm{zero}}
\newcommand{\lm}{\mathrm{lm}}
\newcommand{\rlm}{\mathrm{rlm}}

\newcommand{\lb}{\llbracket}
\newcommand{\rb}{\rrbracket}

\newtheorem{THM}{Theorem}[section]

\newtheorem{LEM}{Lemma}[section]

\input{ex_shared}

\newsiamremark{example}{Example}
\newsiamremark{REM}{Remark}
\usepackage{cleveref}
\begin{document}

\maketitle 

\begin{abstract}
The discrete ordinates discontinuous Galerkin ($S_N$-DG) method is a well-established 
and practical approach for solving the radiative transport equation. In this paper, we study a low-memory variation of the upwind $S_N$-DG method. The proposed method uses a smaller finite element space that is constructed by coupling spatial unknowns across collocation angles, thereby yielding an approximation with fewer degrees of freedom than the standard method.  Like the original $S_N$-DG method, the low memory variation still preserves the asymptotic diffusion limit and maintains the characteristic structure needed for mesh sweeping algorithms. While we observe second-order convergence in scattering dominated, diffusive regime, the low-memory method is in general only first-order accurate. To address this issue, we use upwind reconstruction to recover  second-order accuracy. For both methods, numerical procedures based on upwind sweeps are proposed to reduce the system dimension in the underlying Krylov solver strategy.
\end{abstract}
\section{Introduction}\label{sc-intr}
\setcounter{equation}{0}
\setcounter{figure}{0}
\setcounter{table}{0}

Radiative transport equations \cite{agoshkov2012boundary,Davison-1957,Pomraning-1973,Lewis-Miller-1984,Mihalis-Mihalis-1999,graziani2006computational,Dautray-Lions-2000} describe the flows of particles, such as photons, neutrons, and electrons, as they pass through and interact with a background medium. These equations are used
in various applications, including astrophysics and nuclear reactor analysis.

In this paper, we consider the scaled, steady-state, linear transport equation
\begin{subequations}\label{eq-main}
	\begin{alignat}{3}
	\Omega \cdot \nabla \Psi(\Omega,x) + \left(\frac{\sig{s}(x)}{\veps}+\veps
	\sig{a}(x)\right) \Psi(\Omega,x) &=
	\frac{\sig{s}(x)}{\veps}\overline{\Psi}(x)
	+ \veps q(x), & \quad 
	&(\Omega,x)\in  S \times D,\\
	\Psi(\Omega,x) &= \alpha(\Omega,x), & &(\Omega,x)\in \Gamma^{-}.
	\end{alignat}
\end{subequations}
Here $D\subset \mathbb{R}^d$ $(d=1,2,3)$ is an open, bounded, and Lipschitz domain;  $S$ is the projection of the unit sphere in $\mathbb{R}^3$ into $\mathbb{R}^d$ (the interval $[-1,1]$ for $d = 1$ and  unit disk for $d = 2$); and $\Gamma^{-} = \{ (x,\Omega) \in S\times\partial
D\mid\Omega\cdot n(x)<0\}$, where $n(x)$ is the outward unit normal vector at any point $x \in \partial D$ where the boundary is $C^1$.   

The \textit{angular flux}
$\Psi$ is the flux of particles at the location $x$ moving with unit speed in the direction $\Omega$, and the \textit{scalar flux} $\overline{\Psi} = \frac{1}{|S|}\int_{S}\Psi d\Omega$ is the average of $\Psi$ over $S$.%
\footnote{Often the quantity $\Phi=4\pi\overline{\Psi}$ is referred to as the scalar flux.  The difference is simply a normalization factor from integration of the sphere.  Here, we borrow the convention used in \cite{Lewis-Miller-1984}.} 
The  functions $\sig{s}$ and $\sig{a}$ are (known) non-dimensionalized scattering and absorption 
cross-sections, respectively, and $q$ is a (known) non-dimensionalized source.  
The function 
$\alpha(\Omega,x)$ is the (known) incoming flux at $x\in \partial D$ moving in the direction $\Omega$.
The constant $\veps>0$ is a scaling parameter which characterizes the relative strength of scattering.

Designing effective numerical methods for \eqref{eq-main} is a serious 
challenge, and the intent of this paper is to address two of the main issues.
Firstly, for a three-dimensional problem, the unknown intensity $\Psi$ is a function 
of three spatial and two angular variables; the discretization of this five-dimensional phase space
usually requires significant computational resources. Secondly,
when the parameter $\veps$ is small, $\Psi$ is nearly independent of $\Omega$ and can be approximated by the solution of a diffusion equation in the variable $x$ only \cite{Habetler-Matkowsky-1975,bardos1984diffusion,bensoussan1979boundary}.  That is, away from the boundary, $\Psi(\Omega,x) = \Psi^{(0)} (x) + O(\veps)$ as $\veps \to 0$, where $\Psi^{(0)} $ satisfies 
\begin{equation}\label{eq-difflim}
-\nabla\cdot\left(\frac{1}{3\sig{s}}\nabla \Psi^{(0)} (x)\right) + \sig{a}
\Psi^{(0)}(x) = q(x), \quad x \in  D,
\end{equation}
along with appropriate boundary conditions.
A numerical method for \eqref{eq-main} should preserve this asymptotic limit without having to resolve the length scales associated with $\veps$ \cite{jin1999efficient}.  In other words, in the limit $\veps \to 0$, a discretization of the transport equation \eqref{eq-main} should become a consistent and stable discretization of the diffusion equation \eqref{eq-difflim}.  Otherwise
a highly refined mesh is needed to approximate the solution accurately \cite{larsen1987asymptotic}.%
\footnote{{This issue is also known as ``locking" in the elliptic literature \cite{babuvska1992locking}.}}

Classical approaches for discretizing \eqref{eq-main} often involve separate
treatment of the angular and spatial variables, and a variety of options are available. 
Among them, the $S_N$-DG method \cite{HHE2010, larsen1989asymptotic,adams2001discontinuous} has received significant attention due to it's robustness, computational efficiency, and convenient implementation.   The $S_N$ method (see\cite{larsen2010advances} for a substantial review and additional references) is a collocation method in which the angular variable $\Omega$ is discretized into a finite number of directions and a quadrature rule is used to evaluate $\overline{\Psi}$.
The $S_N$ discretization preserves non-negativity of $\Psi$ and can incorporate the boundary conditions from \eqref{eq-main} in a straightforward way.  It also preserves the characteristic structure of the advection operator in \eqref{eq-main}, which allows for the use of fast sweeping techniques for inverting the discrete form of the operator on the left-hand side of \eqref{eq-main}.  

Discontinuous Galerkin (DG) methods are a class of finite element methods that construct numerical solutions using piecewise
polynomial spaces.  The DG approach was introduced in \cite{reed1973triangular} for the express purpose of solving equations like \eqref{eq-main}, followed shortly thereafter by a rigorous analysis in \cite{lesaint1974finite}.  Since then, DG methods have been applied to nonlinear 
hyperbolic conservation laws and convection-dominated problems
\cite{cockburn2001runge}, elliptic problems \cite{arnold2002unified},
and equations with higher-order derivatives \cite{yan2002local,xu2010local}.
When used with upwind fluxes, DG methods preserve the characteristic structure of \eqref{eq-main} that enables sweeps.  Moreover, if the approximation space can support globally continuous linear polynomials, then DG methods with upwind fluxes will yield accurate numerical solutions for $\Psi$ without the need to resolve $\veps$ with the spatial mesh  \cite{larsen1989asymptotic, adams2001discontinuous, guermond2010asymptotic}.  However, this condition on the approximation space means that at least $P^1$ elements  must be used for a triangular mesh and $Q^1$ elements for a rectangular mesh.%
\footnote{{This condition can be circumvented for non-upwind methods. In \cite{ragusa2012robust}, the authors made the piecewise constant DG method asymptotic preserving with parameters adjusting numerical fluxes under different regimes. Similar techniques were introduced in finite volume contexts \cite{jin1996numerical} as well and were recently used in \cite{guermond2019positive} to develop a positive, asymptotic preserving method.} }

In order to reduce memory costs in the upwind $S_N$-DG method, while still preserving the asymptotic diffusion limit and maintaining the characteristic structure needed for sweeps, we propose in this paper to couple the finite element spaces between different collocation angles in the discrete ordinate approximation. Since the solution becomes isotropic
in the diffusion limit ($\veps \to 0$), we hypothesize that only a $P^1$ (for triangles) or $Q^1$ (for rectangles)  approximation of the angular average is necessary.  Thus, instead of using a tensor product
finite element space for the $S_N$-DG system, we seek the solution in a proper subspace, 
in which all the elements have isotropic slopes. This choice of finite
element space yields a significant reduction in memory per spatial cell, as illustrated in \cref{tab-sndg-costper}.

\begin{table}[!h]
	\centering
	\begin{tabular}{c|c|c}
		\hline
		Unknowns per cell & Triangles ($P^1$) 			& Rectangles ($Q^1$) \\
		\hline
		Standard $S_N$-DG					& $(d+1)n_\Omega$					& $2^{d}n_\Omega$						\\
		\hline
		low-memory $S_N$-DG					& 
		${n_\Omega + d}$ &
		${(n_\Omega -1)+ 2^{d}}$	\\
		\hline
		Memory cost ratio as $n_\Omega \gg 1$			& $d+1$						& $2^{d}$\\
		\hline	\end{tabular}
	\caption{Memory costs of standard $S_N$-DG and the low-memory variation, both for triangles and rectangles, for spatial dimension $d$. The first two rows give the number of unknowns per angle per spatial cell for each approach.  The last row is the asymptotic ratio of the memory costs by two methods when $n_\Omega$ becomes large. }
	\label{tab-sndg-costper}
\end{table}

In the diffusion limit, the low-memory approach typically displays second-order accuracy. However, because the finite element representation of each ordinate is coupled to all the other ordinates, the overall accuracy of the low-memory approach for fixed $\veps$ is only first-order.   
To address this drawback, we propose a modification of the low-memory scheme
that uses local reconstruction to improve accuracy. As long as the reconstruction uses upwind information, the resulting transport operator can still be inverted with sweeps.   While rigorous theoretic properties of this modified scheme are still under 
investigation, we observe numerically that it recovers second-order accuracy for arbitrary fixed
$\varepsilon$ and captures the asymptotic diffusion limit.  However, the method does generate some small numerical artifacts at the discontiuity of the cross section, which we point out in the numerical results of \Cref{sc-num}.  

The rest of the paper is organized as follows. In \Cref{sc-background}, we
introduce the background and revisit the $S_N$-DG method.  Low-memory
methods, including the original first-order approach and the second-order reconstructed scheme, are
detailed in \Cref{sc-lmdg}.  Numerical tests are provided in  \Cref{sc-num} to illustrate the behavior of both approaches.  Finally, conclusions and future work are discussed in \Cref{sc-conclude}.
\section{The $S_N$-DG method}\label{sc-background}
\setcounter{equation}{0}
\setcounter{figure}{0}
\setcounter{table}{0}

In this section, we review the $S_N$-DG scheme and discuss its asymptotic properties and implementation.  Throughout the paper, we consider the  case $\inf_{x\in D} \sig{s}(x) = \delta_s > 0$ and $\inf_{x\in D}\sig{a}(x)=\delta_{\rm{a}} >0$, unless otherwise stated.  In general, the well-posedness of \eqref{eq-main} also holds for $\sig{a} \geq 0$
\cite{wu2015geometric}.  In some places, we will also assume that the cross-section is piecewise constant, either to simplify the exposition or to make connections between first- and second-order forms of the diffusion limit.  In the numerics, we often consider nonzero boundary conditions.  However, in proofs we often assume that $\alpha = 0$.  When $\alpha$ is nonzero but isotropic, many of the results still hold.  However, when $\alpha$ is anisotropic, the diffusion equation requires a boundary layer correction in order to be uniformly accurate \cite{Habetler-Matkowsky-1975}. At the discrete level, this situation requires more sophisticated analysis \cite{larsen1989asymptotic,adams2001discontinuous,guermond2010asymptotic} than is presented here.

\subsection{Formulation}

Consider a quadrature rule with points $\{\Omega_j\}_{ j =1}^{n_\Omega}$ and positive weights $\{w_j\}_{ j =1}^{n_\Omega}$ such that
\begin{equation}
\frac{1}{|S|}\int_S f(\Omega) d\Omega \approx \sum_{j=1}^{n_\Omega} w_j f(\Omega_j), \quad 
\forall f \in C(S).
\end{equation}
We assume the quadrature is exact for polynomials in $\Omega$ up to degree two%
\footnote{Level symmetric quadratures of moderate size will satisfy these properties. See, e.g., \cite{Lewis-Miller-1984} and references therein.}%
; that is, 
\begin{eqnarray}
\label{eq-polyquad}
(i)~\sum_{j=1}^{n_\Omega} w_j = 1,\quad
(ii)~\sum_{j=1}^{n_\Omega} w_j \Omega_j = 0, \quand
(iii)~\sum_{j=1}^{n_\Omega} w_j \Omega_j \otimes \Omega_j = \frac{1}{3} \operatorname{Id}.
\end{eqnarray}
The $S_N$ method approximates the angular flux $\Psi$ at the quadrature points $\{\Omega_j\}_{ j =
	1}^{n_\Omega}$ by a vector-valued 
function $\psi(x) = (\psi_1(x),\psi_2(x),\dots,\psi_{n_\Omega}(x))$ whose components satisfy a coupled system with $n_\Omega$ equations
\begin{equation}
\label{eq:sn}
\Omega_j \cdot \nabla \psi_j(x) + \left(\frac{\sig{s}}{\veps}+\veps \sig{a}\right)\psi_j(x) 
= \frac{\sig{s}}{\veps} \overline{\psi}(x) + \veps q(x),
\qquad 
\overline{\psi}(x) = \sum_{j=1}^{n_\Omega} w_j \psi(\Omega_j, x).
\end{equation}

To formulate the upwind DG discretization of the $S_N$ system \eqref{eq:sn}, let $\cT_h = \{K\}$ be a quasi-uniform partition of the domain $D$. We assume $D = \cup_{K\in \cT_h} \mathrm{cl}(K)$ to avoid unnecessary technicalities. Let $\cF_h = \cup_{K\in \cT_h}\partial K$ be the collection of cell interfaces and let $\cF_h^\partial$ be the collection of boundary faces. Given a cell $K$, we denote by $\nu_K$ the outward normal on $\partial K$ and for any $x \in \partial K$, let $v^{\rm{int}}(x) = \lim_{\delta \to 0^+}v(x - \delta \nu_K)$ and $v^{\rm{ext}}(x) = \lim_{\delta \to 0^+} v(x + \delta \nu_K)$.  Given a face $F$, we denote by $\nu_F$ a prescribed normal (chosen by convention) and, for any $x \in F$, let  $v^{\pm} = \lim_ {\delta \to 0^+} v(x \pm \delta \nu_F)$.  For convenience, we assume trace values are identically zero when evaluated outside of $D$.

The standard $S_N$-DG method uses the tensor-product finite element space
\begin{equation}\label{eq-cV}
\cV_h= \prod_{j=1}^{n_\Omega} V_h,\qquad V_h = \{v_j: v_j|_K \in Z_1(K)\},
\end{equation}
where for  triangular or tetrahedral meshes, $Z_1(K)$ is the space $P^1(K)$ of linear polynomials
on $K$ and for Cartesian meshes $Z_1(K)$ is the space $Q^1(K)$ of multilinear polynomials
on $K$. The space $\cV_h$
can be equipped with an inner product $(\cdot,\cdot)$  and associated norm $\|\cdot\|$ given by
\begin{equation}
(u,v) = \sum_{K\in \cT_h} \sum_{j=1}^{n_\Omega} w_j \int_K u_j v_j dx
\qquad\text{and}\qquad  \|v\| = \sqrt{(v,v)}.
\end{equation}
The semi-norm induced by jumps at the cell interfaces
is given by 
\begin{equation}\label{eq-jump}
\lb v\rb = \left({\sum_{F\in \cF_h}\sum_{j=1}^{n_\Omega} w_j \int_F
	|\Omega_j\cdot \nu_F| (v_j^{-}-v_j^{+})^2 dx}\right)^{1/2}.
\end{equation} 
To construct the $S_N$-DG method, define the local operators
\begin{subequations}
	\begin{align}
	\label{eq-Ljk}
	L_{j,K}(u,v) =& -\int_K u_{j} \Omega_j \cdot \nabla v_j dx + \int_{\partial K}
	\widehat{u}_{j} \Omega_j\cdot \nu_K v^{\rm{int}}_j dx \\
	&+ \int_K
	\left(\frac{\sig{s}}{\veps}+\veps \sig{a}\right) u_jv_j dx,\nonumber
	\\
	\label{eq-Sjk}
	S_{j,K}(u,v) =&  \int_K\frac{\sig{s}}{\veps}\overline{u} v_j dx, \quad \text{with}~\overline{u} = \sum_{j=1}^{n_\Omega} w_j u_j,
	\\
	Q_{j,K,\alpha}(v) =& \int_K \veps q v_j dx - \int_{\partial K \cap
		\cF_h^{\partial}} \alpha \Omega_j\cdot \nu_K v_j^{\mathrm{int}} dx,
	\end{align}
\end{subequations}
where  $\widehat{u}_j (x)= \lim_{\delta \to 0^-} u(x + \delta \Omega_j) $ is the upwind trace at $x \in \partial K$, and is defined as zero when the limit is taken outside of $D$. Then set 
\begin{equation}\label{eq-Bdef}
B(u,v) = L(u,v) - S(u,v),
\end{equation} 
where
\begin{equation}
\label{eq-LandS}
L(u,v) = \sum_{K\in \cT_h}\sum_{j=1}^{n_\Omega} w_jL_{j,K}(u,v)
\quand 
S(u,v) = \sum_{K\in \cT_h}\sum_{j=1}^{n_\Omega} w_j S_{j,K}(u,v),
\end{equation}
and let 
\begin{equation}\label{eq-lalpha}
Q_{\alpha}(v) = \sum_{K\in \cT_h}\sum_{j=1}^{n_\Omega}
w_j Q_{j,K,\alpha}(v).
\end{equation} 
The $S_N$-DG method is then: \textit{find $\psi_{h}=(\psi_{h,1},\dots,\psi_{h,n_\Omega}) \in \cV_h $ such that} 
\begin{equation}\label{eq-sndg-scheme}
B(\psi_h,v) = Q_{\alpha}(v), \qquad \forall v \in \cV_h.
\end{equation}

\subsubsection{Implementation}\label{sc-basis}
Recall that $n_\Omega$ is the number of discrete ordinates in the $S_N$ discretization.  Let $n_x = |\cT_h|$ be the number of mesh cells in $\cT_h$ and let $n_P$
be the dimension of $Z_1(K)$.  Then the dimension of $\cV_h$ is $n_\Omega\cdot n_x \cdot n_P$. 

Let $\{b^{p,r}:p = 1,\ldots,n_x,r=0,\ldots,n_P-1\}$ be a set of basis functions for $V_h$, with $b^{p,r}$ locally supported on $K_p \in \cT_h$. Then the set $\mathbb{B} = \{\xi^{l,p,r}:l =1,\dots,n_\Omega,p=1,\dots,n_x,r=0,\dots,n_P-1\}$, where $\xi_j^{l,p,r}(x) = \delta_{l j} b^{p,r}(x)$ ($j = 1,\ldots n_\Omega$) and $\delta$ is the Kronecker delta,
gives a complete set of basis functions
for $\cV_h$. 
With this choice of basis functions, the variational formulation in \eqref{eq-sndg-scheme}, written as
\begin{equation}\label{eq-vari-LSQ}
L(\psi_h,v) = S(\psi_h,v) + Q_{\alpha}(v),\qquad \forall v\in \cV_h,
\end{equation}
can be assembled into a linear system (detailed in  \Cref{ap-sndg-matrix})
\begin{equation}\label{eq-sndg-mat}
\bL \mathbf{\Psi} = \bM \bP \mathbf{\Psi} + \bQ.
\end{equation}
In the above equation, $\bL$ is an $(n_\Omega\cdot n_x\cdot n_P)\times (n_\Omega\cdot n_x\cdot
n_P)$ block diagonal matrix, where the $j$-th block ($j = 1,\ldots n_\Omega$) corresponds to the discretization of the operator $\psi_j \to \Omega_j \cdot \nabla \psi_j + \left(\frac{\sig{s}}{\veps}+\veps \sig{a}\right)\psi_j$; $\bM$ is an injective  $(n_\Omega\cdot n_x\cdot n_P)\times (n_x\cdot n_P)$ matrix, 
$\bP$ is an $(n_x\cdot n_P) \times (n_\Omega\cdot n_x\cdot n_P)$ matrix; $\bQ$ is an $(n_\Omega\cdot n_x\cdot n_P)$ vector assembled from the source $q$ and the inflow boundary $\alpha$; and $\bPsi = (\psi^{l,p,r})$ is an $(n_\Omega\cdot n_x\cdot n_P)$ vector such that
$\psi_h = \sum_{l,p,r}\psi^{l,p,r}\xi^{l,p,r}$. 

If upwind values are used to evaluate the numerical trace $\widehat{u}_j$, each block of $\bL$ can be inverted efficiently with a sweep algorithm. 
The system in \eqref{eq-sndg-mat} can be solved numerically with a Krylov method by first solving the reduce system 
\begin{equation}\label{eq-sndg-phi}
\bPhi - \bP \bL^{-1}\bM \bPhi = \bP \bL^{-1} \bQ
\end{equation}
for the ${n_x\cdot n_P}$ vector $\bPhi:=\bP \bPsi$. This equation is derived by applying $\bL^{-1}$ and then $\bP$ to \eqref{eq-sndg-mat}.  In a second step $\bPsi$ is recovered from the relation  
\begin{equation}\label{eq-sndg-final-sweep}
\bPsi = \bL^{-1}\bM \bPhi + \bL^{-1} \bQ.
\end{equation} 
The following theorem is proven in \Cref{ap-sndg-mat-psi}.

\begin{THM} \label{thm-sndg-mat-psi}
	The matrix
	$\bI_{n_x\cdot n_P}-\bP\bL^{-1}\bM$ is invertible.
\end{THM}
\begin{REM}[Sherman--Morrison formula]  According to the
	Sherman-Morrison formula (see for example \cite[Section 2.1.3]{golub2012matrix}): given invertible matrices $\bB = \bA + \bU \bV$ and 
	$\bI + \bV \bA^{-1}\bU$,
	\begin{equation}\label{eq-sndg-mat-sm}
	\bB^{-1} = \bA^{-1} - \bA^{-1}\bU (\bI+\bV\bA^{-1}\bU)^{-1}\bV \bA^{-1}.
	\end{equation}
	The direct application of \eqref{eq-sndg-mat-sm}
	with $\bA = \bL$, $\bU = -\bM$ and $\bV = \bP$, yields the formula in \eqref{eq-sndg-final-sweep} with $\bPhi$ given by \eqref{eq-sndg-phi}. 
\end{REM}

\subsubsection{Asymptotic scheme}

As $\veps \to 0$, the $S_N$-DG scheme gives a consistent
approximation to the asymptotic diffusion problem. For simplicity, we focus here on the zero inflow boundary condition $\alpha
= 0$. The analysis of more general boundary conditions can be found in
\cite{adams2001discontinuous, guermond2010asymptotic, guermond2014discontinuous, larsen1989asymptotic}.

We use an overline to represent isotropic subspaces. For example, 
\begin{equation}
\cD_h = \{v = (v_1,\ldots,v_{n_\Omega}) \in \cV_h: v_i = \overline{v}, \forall i\}.
\end{equation}
We further define $\cC_{h,\zero}$ to be the space of continuous functions in $\cD_h$ that vanish on $\partial D$. $\cD_h^d = \{(\vphi_1,\dots,\vphi_d): \vphi_i \in \cD_h\}$ is used to represent the tensor product space of $\cD_h$ with an induced norm still denoted as $\|\cdot\|$.  In particular, since $\cD_h$ and $V_h$ are isomorphic, 
we often identify $\cD_h$ with $V_h$.
To facilitate the discussion, we also define  
\begin{equation}
\label{eq-Jh}
J_h = \frac{1}{\veps}\sum_{j=1}^{n_\Omega} w_j\Omega_j \psi_{h,j} = \sum_{j=1}^{n_\Omega} w_j\Omega_j \frac{\psi_{h,j}-\overline{\psi}_h}{\veps},
\end{equation} 
which is a vector field in $\mathbb{R}^d$.
The following result is proved in  \cite{guermond2010asymptotic}\footnote{The result in \cite{guermond2010asymptotic} is actually stated for more generally.  In particular it allows $\alpha$ to be nonzero and possibly anisotropic.}; see also \cite{adams2001discontinuous} and  \Cref{thm-lmdg-asympscheme} in this paper. 

\begin{THM}[Asymptotic scheme]\label{thm-sndg-asympscheme}
	Suppose $\alpha = 0$.  
	Then as $\veps \to 0$, $(\psi_h)_{\veps>0}$ and $(J_h)_{\veps>0}$ converge to ${\psi}_h^{(0)} = \overline{\psi}_h^{(0)} \in
	\cC_{h,\zero}$ and $J_h^{(0)}\in \cD_{h}^d$, respectively, that are the unique solution to the mixed problem:
	\begin{subequations}\label{eq-sndg-mix}
		\begin{gather}
		\sum_{K\in \cT_h} \int_K \left(-J_h^{(0)}\cdot \nabla \vphi +
		\sig{a}{\psi}_h^{(0)} \vphi\right) dx= \int_D q\vphi
		dx,\label{eq-sndg-lim1}\\
		\sum_{K\in \cT_h}\int_K\left(\frac{1}{3}\nabla {\psi}_h^{(0)}  +
		\sig{s}
		J_h^{(0)}\right)\cdot \zeta dx =
		0\label{eq-sndg-lim2},
		\end{gather}
	\end{subequations}
	$\forall \vphi \in \cC_{h,\zero}$ and $\forall \zeta \in \cD_h^{d}$.
\end{THM}

\section{Low-memory strategies}\label{sc-lmdg}
\setcounter{equation}{0}
\setcounter{figure}{0}
\setcounter{table}{0}
In this section, we generalize the statement of \Cref{thm-sndg-asympscheme} slightly to allow for proper subspaces of $\cV_h$ in the finite element formulation. Based on the analysis, a first-order low-memory scheme is constructed. We then apply the reconstruction technique to lift the accuracy of the method to second-order. 

\subsection{Asymptotic schemes with subspaces of $\cV_h$}\label{subsec-asymptotic-subspaces} The results of \cref{thm-sndg-asympscheme} suggest that, rather than $\psi_{h}$, it is the approximation of the integrated quantities $\overline{\psi}_h$ and $J_h$ that play an important role in the diffusion limit.  In particular, the continuity requirement on $\overline{\psi}_h^{(0)}$ plays a crucial role.  Indeed, as is well known \cite{adams2001discontinuous}, if the space $V_h$ is constructed from piecewise constants, then \eqref{eq-sndg-mix} implies that ${\psi}_h^{(0)}$ is a global constant and $J_h^{(0) }= 0$.   This solution is clearly inconsistent with the diffusion limit.   However, it is possible to construct a DG method: \textit{find $\psi_{h}=(\psi_{h,1},\dots,\psi_{h,n_\Omega}) \in \cW_h $ such that} 
\begin{equation}\label{eq-sndglm-scheme}
B(\psi_h,v) = Q_{\alpha}(v), \qquad \forall v \in \cW_h
\end{equation}
based on a proper subspace $\cW_h \subset \cV_h$ that maintains the diffusion limit, but requires fewer unknowns for a given mesh $\cT_h$.

\begin{THM}\label{thm-gene-uni-solvency}
	For each $\veps>0$ and linear subspace $\cW_h \subset \cV_h$, \eqref{eq-gene-scheme} has a unique solution. In particular, if $\alpha = 0$, the solution  
	satisfies the energy estimate 
	\begin{equation}\label{eq-stab}
	\frac{1}{\veps}\|\sig{s}^{\frac{1}{2}}(\psi_h - \overline{\psi}_h)\|^2
	+ \frac{\veps}{2}\|\sig{a}^\frac{1}{2}\psi_h\|^2 
	+\frac{1}{2}\lb \psi_h\rb^2
	\leq \frac{\veps}{2\delta_{\rm{a}}}\|q\|^2.
	\end{equation}
\end{THM}
The proof is based on coercivity of $B(\cdot,\cdot)$ and we refer to \cite{HHE2010} and \cite{guermond2010asymptotic} for details. Here, $\alpha = 0$ is assumed for simplicity. Energy estimates with general inflow boundary condition can be found in \cite[Lemma 4.2]{guermond2010asymptotic}. In \cite{HHE2010}, the case $\veps = 1$ is studied and error estimates are derived using the coercivity with respect to a modified norm. 

We next characterize sufficient conditions  for  $\cW_h $.  Define the spaces 
\begin{equation}
\overline{\Omega \cW_h}:= \{\sum_{j=1}^{n_{\Omega}} w_j \Omega_j v_j: v\in \cW_h\} \subset \cD_h^d
\quand
\Omega\cdot \overline{\Omega \cW_h} := \{\Omega\cdot \zeta:\zeta \in \overline{\Omega \cW_h}\} \subset \cV_h,
\end{equation}
where $\Omega\cdot \zeta:=(\Omega_1 \cdot \zeta, \ldots, \Omega_{n_\Omega} \cdot \zeta)$.
According to \eqref {eq-Jh}, $J_h\in \overline{\Omega \cW_h}$. \Cref{thm-sndg-asympscheme} can now be generalized to the space $\cW_h $.

\begin{THM}\label{thm-lmdg-asympscheme}
	Suppose $\alpha = 0$. Suppose $\cW_h \subset \cV_h$ is a linear space such that
	$\Omega\cdot \overline{\Omega \cW_h} \subset \cW_h$.
	Then as $\veps \to 0$, $({\psi_h})_{\veps>0}$ and $({J}_h)_{\veps>0}$ converge to ${\psi}_h^{(0)} = \overline{\psi}_h^{(0)}\in
	\cC_{h,\zero}\cap\cW_h$ and ${J}_h^{(0)}\in\overline{\Omega \cW_h}$, respectively, that are the unique solution to the mixed problem
	\eqref{eq-sndg-mix},
	
	$\forall \vphi \in \cC_{h,\zero}\cap \cW_h$ and $\forall \zeta \in \overline{\Omega \cW_h}$.

\end{THM}

\begin{proof}
	
	Because the proof follows the arguments in \cite[Section 4]{guermond2010asymptotic} closely, 
	we provide only 
	a brief outline, emphasizing where the condition on the space $\cW_h$ plays a role.  
	
	1. The stability estimate in \eqref{eq-stab} provides the following three bounds:
	\begin{equation}\label{eq-bound}
	(i)~\|\psi_h \|^2\leq   \frac{1}{\delta_{\rm{a}}^2}\|q\|^2
	,\quad 
	(ii)~\|\psi_h -\overline{\psi}_h \|^2 \leq \frac{\veps^2}{\delta_{\rm{a}}
		\delta_{\rm{s}}}\|q\|^2
	,\quand
	(iii)~ \lb\psi_h\rb^2 \leq
	\frac{\veps}{\delta_{\rm{a}}}\|q\| ^2.
	\end{equation}
	Bounds (i) and (ii)  imply that $\psi_h $ converges (via a subsequence) to a function  $ {\psi}_h^{(0)} \in  \overline{\cV}_h$.  Bound (iii)
	implies that $\psi_h^{(0)} \in \cC_{h,\zero} \cap \cW_h = \cC_{h,\zero} \cap \overline{\cW}_h$.
	
	2. Since, from the definition in \eqref{eq-Jh},
	\begin{equation}
	\|J_h \| \leq \sum_{j=1}^{n_\Omega} w_j  \frac{\|\psi_h-\overline{\psi}_h\|}{\veps},
	\end{equation}
	where $\|J_h\|$ is the tensor product norm of $J_h$ in  $\overline{\cV}^d_h$,
	the bound (ii) implies further that  $(J_h )_{\veps>0} \subset \overline{\Omega \cW_h}$ is uniformly bounded and hence converges subsequentially to a limit  $J_h^{(0)}\in \overline{\Omega \cW_h}$.

	3. The equation in \eqref{eq-sndg-lim1} is derived by testing \eqref{eq-sndglm-scheme} with $v =\vphi \in \cC_{h,\zero}\cap \cW_h$ and 
	using the fact that $\vphi$ is independent of $\Omega$ and continuous in $x$.
	
	4. It is the derivation of \eqref{eq-sndg-lim2} which uses the condition $\Omega\cdot \overline{\Omega \cW_h} \subset \cW_h$.  Specifically, if $v = \Omega\cdot \zeta$ with $\zeta \in \overline{\Omega\cW_{h}}$, then this condition implies that $v \in \cW_h$.  Therefore, we can test \eqref{eq-sndglm-scheme} with this choice of $v$ to find that
	\begin{equation}
	\begin{aligned}
	L(\psi_h,\Omega \cdot \zeta) - S(\psi_h,\Omega \cdot \zeta) 
	=&- \sum_{j=1}^{n_\Omega} w_j \sum_{K \in \cT_h} 
	\int_K \psi_{h,j} \Omega_j \cdot \nabla (\Omega_j \cdot \zeta) dx \\
	&+ \sum_{j=1}^{n_\Omega} w_j \sum_{K \in \cT_h} 
	\int_{\partial K} \widehat\psi_{h,j} (\Omega_j \cdot \nu_K) (\Omega_j \cdot \zeta^{\rm{int}}) dx \\
	&+ \sum_{j=1}^{n_\Omega} w_j \sum_{K \in \cT_h}
	\int_K \left(
	\left( \frac{\sig{s}}{\veps} + \veps \sig{a} \right) \psi_{h,j} 
	- \frac{\sig{s}}{\veps} \overline{\psi}_h  \right) (\Omega_j \cdot \zeta) dx \\
	=:&\ I + II + III.
	\end{aligned}
	\end{equation}
	We combine $I$ and $II$, using the fact that $ \overline{\psi}^{(0)}_{h} \in \cC_{h,\zero}$ and invoking \eqref{eq-polyquad}. This gives
	\begin{equation}
	\label{eq-fick1}
	\begin{aligned}
	\lim_{\veps\to 0} (I + II)  
	&=\sum_{j=1}^{n_\Omega} w_j (\Omega_j \otimes \Omega_j) : 
	\sum_{K \in \cT_h}  \left(- \int_K  \overline{\psi}^{(0)}_{h}  \nabla \zeta dx   + \int_{\partial K} \overline{\psi}^{(0)}_{h}  \nu_K \otimes \zeta^{\rm{int}}  dx \right)  \\
	& = \frac{1}{3} \operatorname{Id} : \sum_{K \in \cT_h}  \int_K \nabla \overline{\psi}^{(0)}_{h} \otimes \zeta dx
	=  \sum_{K \in \cT_h}  \int_K \frac{1}{3} \nabla \overline{\psi}^{(0)}_{h} \cdot \zeta  dx.
	\end{aligned}
	\end{equation}
	Since $\sum_{j=1}^{n_\Omega} w_j \overline{\psi}_{h}  \Omega_j =0 $,
	\begin{align}
	\label{eq-fick2}
	\lim_{\veps\to 0} III = 
	\lim_{\veps\to 0}  \sum_{K \in \cT_h}
	\int_K 
	\left( \frac{\sig{s}}{\veps} + \veps \sig{a} \right) \sum_{j=1}^{n_\Omega}(w_j \psi_{h,j} \Omega_j )\cdot \zeta dx =
	\int_K \sig{s} J^{(0)}_h  \cdot \zeta dx .
	\end{align}
	Finally, the right-hand side of \eqref{eq-sndglm-scheme} is (for $\alpha = 0$)
	\begin{equation}
	\label{eq-L0-test}
	Q_0(v) = \sum_{j=1}^{n_\Omega} w_j \sum_{K \in \cT_h} \int_K \Omega_j \cdot \zeta  q dx= 0.
	\end{equation}
	Combining  \eqref{eq-fick1}, \eqref{eq-fick2}, and \eqref{eq-L0-test} recovers \eqref{eq-sndg-lim2}.
	
	5. Uniqueness of the subsequential limits $\psi ^{(0)}_h$ and $J^{(0)}_h$ follows from the uni-solvency of  \eqref{eq-sndg-mix}.  Indeed if $(\widetilde  \psi_h,\widetilde  J_h)$ is the difference between any two solutions of  \eqref{eq-sndg-mix}, then 
	\begin{equation}
	3\sig{s} \|\widetilde J_h\|^2 + \sig{a}\|\widetilde \psi_h \|^2 = 0.
	\end{equation}
	Since $\sig{s}$ and $\sig{a}$ are assumed positive, it follows that $\widetilde \psi_h$ and $\widetilde J_h$ are identically zero.
\end{proof}

We then discuss the choice of $\cW_h$ and the corresponding space pair, $\cS_h := \cC_{h,\zero}\cap \cW_h$ and $\cJ_h := \overline{\Omega\cW_h}$, in the diffusion limit. Let $Z_0(K)$ be the space spanned by constants on $K$. Then we define the piecewise constant space
$
\cV_{h,0} = \{v\in \cV_h:v_j|_K \in Z_0(K),  \forall K \in \cT_h\}
$ and its orthogonal complement
$\cV_{h,1} = \{v \in \cV_h: \int_K v_j dx = 0,  \forall K \in \cT_h\}.
$
The isotropic subspace of $\cV_{h,r}$ is denoted by  $\cD_{h,r}$ and the subsequent product space is denoted by $\cD_{h,r}^d$, $r = 0,1$. 

1. When $\cW_h = \cV_{h,0}$ or $\cW_h =\{v\in \cV_h:v_j|_K \in P_1(K),\forall K \in \cT_h\}$, we have $\cS_h = \{0\}$, which implies $\psi_h^{(0)} = 0$ and $J_h^{(0)} = 0$. 

2. When $\cW_h = \cV_{h,0} + \cD_{h,1} + \Omega \cdot \cD_{h,1}^d$, it can be shown that $\cS_h = \cC_{h,\zero}$, $\cJ_h =\cD_h^d$\footnote{Since $\cD_h^d\supset \cJ_h = \overline{\Omega \cW_h} \supset \overline{\Omega \left(\Omega\cdot \cD_h^d\right)} = \cD_h^d$, which forces $\cJ_h = \cD_h^d$. Here we have used (iii) in \eqref{eq-polyquad} for the last equality.} and $\Omega\cdot \cJ_h \subset \cW_h$. The asymptotic scheme is the same as that of the original $S_N$-DG method.
If $\sig{s}$ and $\sig{a}$ are both piecewise constant, then the asymptotic scheme has the primal form: \textit{find $\psi_h^{(0)} \in \cC_{h,\zero}$, such that
	\begin{equation}\label{eq-lmdg-cg}
	\sum_{K\in \cT_h} \int_K \left(\frac{1}{3 \sig{s}}\nabla {\psi}_h^{(0)}
	\cdot \nabla \vphi+\sig{a}{\psi}_h^{(0)}\vphi\right) dx= \int_D q\vphi
	dx,
	\end{equation}
	$\forall \vphi \in \cC_{h,\zero}$.} This is the classical continuous Galerkin approximation, which is stable and second-order accurate.  

3. When $\cW_h =\cV_{h,0} + \cD_{h,1}$, then $\cS_h = \cC_{h,\zero}$, $\cJ_h =\cD_{h,0}^d$ and $\Omega\cdot \cJ_h \subset \cW_h$. With $P^1$ elements and triangular meshes, the asymptotic scheme is essentially the $P_N$ scheme suggested by Egger and Schlottbom in \cite{egger2012mixed} with $N = 1$. If $Q^1$ elements and Cartesian meshes are used, the scheme yields the same variational form as that in \cite{egger2012mixed},  while the space pair no longer satisfies the condition $\nabla \cS_h \subset \cJ_h$.

From another point of view, suppose $\sig{s}$ and $\sig{a}$ are piecewise constant, the primal form is: \emph{find $\psi_h^{(0)} \in \cC_{h,\zero}$, such that
	\begin{equation}\label{eq-difflimPF}
	\sum_{K\in \cT_h} \int_K\left( \frac{1}{3\sig{s}} \Pi_0(\nabla \psi_h^{(0)}) \cdot \Pi_0(\nabla \vphi) + \sig{a}\psi_h^{(0)}\vphi\right)dx  = 		
	\sum_{K\in \cT_h} \int_K q \vphi dx,
	\end{equation}
	$\forall \vphi \in \cC_{h,\zero}$.}
For $P^1$ elements on triangular meshes, \eqref{eq-difflimPF} is identical to \eqref{eq-lmdg-cg}.
For $Q_1$ elements on Cartesian meshes, one can show that \eqref{eq-difflimPF} is unisolvent. Furthermore, $\|\Pi_0(\nabla \psi_h^{(0)})\|^2 + \|\psi_h^{(0)}\|^2 \leq \max(\frac{3\sig{s}}{2\sig{a}},\sig{a}^{-2}) \|q\|^2$, if $\sig{a}\geq \delta_{\mathrm{a}} >0$. While the accuracy is hard to analyze under the finite element framework. Assume a uniform square mesh with cell length $h$. Let $\sig{s}$ and $\sig{a}$ be globally constant. Then \eqref{eq-difflimPF} can be rewritten as a finite difference scheme under the Lagrange basis functions.
\begin{align}
-\frac{\psi_{i-1,j-1}+\psi_{i-1,j+1}-4\psi_{i,j}+\psi_{i+1,j-1}+\psi_{i+1,j+1}}{3\sig{s}\cdot 2h^2}
+ \sig{a}A[\psi_{i,j}] = & A[q_{i,j}],\\
A[\psi_j] :=  \frac{1}{36} \left(\psi_{i-1,j-1}+\psi_{i-1,j+1}+\psi_{i+1,j-1}+\psi_{i+1,j+1}\right) &\\
+  \frac{1}{9}\left(\psi_{i-1,j}+\psi_{i,j-1}+\psi_{i,j+1}+\psi_{i+1,j}\right)\ + \frac{4}{9}\psi_{i,j}.& \nonumber
\end{align}
The truncation error of the method is $\cO(h^2)$.

\ 

At the first glance,
$\cW_h = \cV_{h,0} + \cD_{h,1} + \Omega \cdot \cD_{h,1}$ seems to be the natural choice for constructing the low-memory scheme that preserves the correct diffusion limit. However, coupling between angles requires special treatment for reducing the system dimension. The extra moments $\Omega \cdot \cD_{h,1}^{d}$ will make the resulting system even larger than that of the original $S_N$-DG method. Although it may be worth to include extra moments for problems with anisotropic scattering, for which a large system has to be solved anyway, we avoid this option for solving \eqref{eq-main}. We therefore explore the other choice $\cW_h = \cV_{h,0} + \cD_{h,1}$ in the rest of the paper. 
\subsection{Low-memory scheme}
Based on the analysis and discussion of \Cref {subsec-asymptotic-subspaces}, we propose a scheme that uses the finite element space
\begin{equation}
\cV_h^{\lm} = \cV_{h,0}+\cD_{h,1}.
\end{equation} 
The low-memory $S_N$-DG scheme
is written as follows: \textit{find $\psi_h \in \cW_h$, such that
	\begin{equation}\label{eq-gene-scheme}
	B(\psi_h,v) = Q_{\alpha}(v), \qquad \forall v \in \cV_h^\lm,
	\end{equation}}
where $B$ and $Q_\alpha$ are defined in \eqref{eq-Bdef} and \eqref{eq-lalpha}, respectively.

We now show that this scheme can be implemented using sweeps; i.e., a strategy analagous to the \eqref{eq-sndg-phi} and \eqref{eq-sndg-final-sweep}, which relies heavily on the fast inversion of the operator $\bL$.  For simplicity, we only consider the case $\sig{s}$ being piecewise constant.  The implementation is based on the block matrix formulation \eqref{eq-sndg-mat} of the $S_N$-DG method:

\begin{equation}
\bL = \left[\begin{matrix} \bL_{00}&\bL_{01}\\ \bL_{10}&
\bL_{11}\end{matrix}\right], \quad
\bS = \left[\begin{matrix}\bM_0\bP_0&\\&\bM_1\bP_1\end{matrix}\right],\quand
\bQ = \left[\begin{matrix}\bQ_0\\\bQ_1\end{matrix}\right].
\end{equation}
Here $\bL_{rr'}$ are matrix blocks associated to $L(u,v)$ with $u\in \cV_{h,r'}, v\in \cV_{h,r}$. The sizes of $\bL_{00}$, $\bL_{01}$, $\bL_{10}$ and $\bL_{11}$ are $(n_\Omega\cdot n_x)\times (n_\Omega\cdot n_x)$, $(n_\Omega\cdot n_x)\times (n_\Omega\cdot n_x\cdot (n_P-1))$, $(n_\Omega\cdot n_x\cdot (n_P-1))\times (n_\Omega\cdot n_x)$ and $(n_\Omega\cdot n_x\cdot (n_P-1))\times (n_\Omega\cdot n_x\cdot (n_P-1))$, respectively.
The block $\bS_{rr'} = \bM_{r} \bP_{r'}$ is associated to $S(u,v)$ with $u\in \cV_{h,r'}, v\in \cV_{h,r}$;  it has the same size as $\bL_{rr'}$. The matrices $\bM_0$ and $\bP_0$ have dimensions $(n_\Omega\cdot n_x)\times n_x$ and $n_x \times(n_\Omega\cdot n_x)$, respectively; the matrices $\bM_1$ and $\bP_1$ have dimensions  $(n_\Omega\cdot n_x\cdot(n_P-1))\times (n_x\cdot(n_P-1))$ and $(n_x\cdot (n_P-1)) \times(n_\Omega\cdot n_x\cdot (n_P-1))$, respectively. The vector block $\bQ_r$ is associated to $Q_{\alpha}(v)$ for $v\in \cV_{h,r}$, with $\bQ_0$ an $n_\Omega\cdot n_x$ vector and $\bQ_1$ an $n_\Omega\cdot n_x\cdot (n_P-1)$ vector. 

Recall from \Cref{sc-basis} that for each $p$, $\{b^{p,r}\}_{r=0}^{n_P-1}$ forms a basis for $Z_1(K_p)$, and $\xi^{l,p,r}_j = \delta_{lj}b^{p,r}$. We further assume $\{b^{p,r}\}_{r=0}^{n_P-1}$ is an orthogonal set and $\{b^{p,0}\}_{p=1}^{n_x}$ is a set of constant functions on $K_p$. Then $\mathbb{B}_0 = \{\xi^{l,p,0}:l = 1,\dots,n_\Omega, p = 1,\dots,n_x\}$ and  $\mathbb{B}_1 = \{\xi^{l,p,r}:l = 1,\dots,n_\Omega, p = 1,\dots,n_x,r = 1,\dots,n_P-1\}$ are sets of basis functions for $\cV_{h,0}$ and $\cV_{h,1}$, respectively.
Let $\mathbb{B}_1^\lm = \{\eta^{p,r}: \eta^{p,r}_j = b^{p,r},j =1,\dots,n_\Omega, p = 1,\dots,n_x, r =1,\dots,n_P-1\}$. Then $\mathbb{B}_1^\lm$ 
is a set of basis for $\cD_{h,1}$. Hence $\cV_h^\lm = \tspan \{\mathbb{B}_0,\mathbb{B}_1^\lm\}$. 
The dimension of $\cV_h^\lm$ is then $n_\Omega\cdot n_x + n_x\cdot (n_P-1)$.
Because $\eta^{p,r}_j = b^{p,r} =  \sum_{l=1}^{n_\Omega}\delta_{lj}b^{p,r}= \sum_{l=1}^{n_\Omega} \xi_j^{l,p,r}$, there exists a mapping from $\mathbb{B}_1$ to $\mathbb{B}_1^\lm$
\begin{equation}
\eta^{p,r} = \sum_{l=1}^{n_\Omega}\xi^{l,p,r} = \sum_{l',p',r'=1}^{n_\Omega}\Sigma^{(p,r),(l',p',r')}\xi^{l',p',r'},
\end{equation}
where $\bSigma = (\Sigma^{(p,r),(l',p',r')})$ is an $(n_x\cdot (n_P-1))\times
(n_\Omega\cdot n_x \cdot (n_P-1))$ matrix with components 
$\Sigma^{(p,r),(l',p',r')} = \delta_{pp'}\delta_{rr'}$. The matrix $\bSigma$ corresponds to a 
summation operator that maps an angular flux to a scalar flux, while $\bSigma^T$ copies
the scalar flux to each angular direction. 

Let the solution of the low-memory method be represented by 
$\bPsi = \left[\bPsi_0, \bSigma^T\bPhi_1\right]^T$.
Using the fact $\bP_1\bSigma^T = \bI_{n_x\cdot (n_P-1)}$, one can show $\bPsi$ satisfies the equations
\begin{subequations}\label{eq-lmdg-block}
	\begin{gather}
	\bL_{00}\bPsi_0+\bL_{01}\bSigma^T\bPhi_1 = \bM_0\bP_0\bPsi_0 +
	\bQ_0,\label{eq-lmdg-block0}\noeqref{eq-lmdg-block0}\\
	\bSigma\bL_{10}\bPsi_0+\bSigma\bL_{11}\bSigma^T\bPhi_1 =
	\bSigma\bM_1\bPhi_1 + \bSigma\bQ_1.\label{eq-lmdg-block1}\noeqref{eq-lmdg-block1}
	\end{gather}
\end{subequations}
As that in the original $S_N$-DG method, the system dimension of \eqref{eq-lmdg-block} can be reduced with the following procedure. 

\

1. Solve for $\bPhi_1$ in terms of $\bPsi_0$ through \eqref{eq-lmdg-block1}: 
\begin{equation}\label{eq-lmdg-psi1}
\bPhi_1 = \bB_{11}^{-1}\bSigma \left(-\bL_{10} \bPsi_0 +
\bQ_1\right), \qquad \bB_{11} = \bSigma \bL_{11}\bSigma^T -
\bSigma\bM_1.
\end{equation}

2. Substitute $\bPhi_1$ from \eqref{eq-lmdg-psi1} into \eqref{eq-lmdg-block0} to
obtain a closed equation for $\bPsi_0$:
\begin{equation}
\begin{aligned}\label{eq-lmdg-psi0}
\bPsi_0 
- \bL_{00}^{-1} \bM_0 (\bP_0 \bPsi_0)  
&-\bL_{00}^{-1} \bL_{01}\bSigma^T(\bB_{11}^{-1} \bSigma\bL_{10} \bPsi_0) 
= \bL_{00}^{-1} (\bQ_0-
\bL_{01}\bSigma^T\bB_{11}^{-1}\bSigma\bQ_1).
\end{aligned}
\end{equation}

3. Apply $\bP_0$ and $\bSigma\bL_{10}$ to \eqref{eq-lmdg-psi0} to 
obtain a closed system for $\bX_0 = \bP_0\bPsi_0$ and {$\bX_1 = \bB_{11}^{-1}\bSigma\bL_{10}\bPsi_0$}:
\begin{equation}\label{eq-lmdg-reducemat}
\begin{aligned}
\bK
\left[\begin{matrix}\bX_0\\\bX_{1} \end{matrix}\right]=\left[\begin{matrix}
\bP_0\\
\bSigma\bL_{10}
\end{matrix}\right]\bL_{00}^{-1} (\bQ_0-
\bL_{01}\bSigma^T\bB_{11}^{-1}\bSigma\bQ_1),
\end{aligned}
\end{equation}
where
\begin{equation}
\bK = \left[\begin{matrix}
\bI_{n_x}- \bP_0\bL_{00}^{-1} \bM_0 & -\bP_0\bL_{00}^{-1}\bL_{01}\bSigma^T\\
- \bSigma\bL_{10}\bL_{00}^{-1} \bM_0 &\bB_{11} -\bSigma\bL_{10}\bL_{00}^{-1}\bL_{01}\bSigma^T \\
\end{matrix}\right].
\end{equation}

4. Solve for $\bX_0$ and $\bX_1$ in \eqref{eq-lmdg-reducemat}. 
Then use \eqref{eq-lmdg-psi0} and \eqref{eq-lmdg-psi1} to obtain $\bPsi$:
\begin{subequations}
	\label{eq-lm-step4}
	\begin{align}
	\label{eq-lm-step4-phi}
	\bPsi_0 &=  \bL_{00}^{-1} \bM_0 \bX_0 + \bL_{00}^{-1}\bL_{01}\bSigma^T\bX_1 + \bL_{00}^{-1} (\bQ_0-
	\bL_{01}\bSigma^T\bB_{11}^{-1}\bSigma\bQ_1),\\
	\label{eq-lm-step4-psi}
	\bPhi_1 &= \bB_{11}^{-1} \bSigma\left(-\bL_{10} \bPsi_0 + \bQ_1\right).
	\end{align}
\end{subequations}

\

Only Step 4 above is needed to implement the algorithm.
If one solves for $\bPsi_0$ directly from \eqref{eq-lmdg-psi0}, 
then an $(n_\Omega\cdot n_x)\times(n_\Omega\cdot n_x)$ matrix should be
inverted. While with \eqref{eq-lmdg-reducemat}, 
the matrix dimensions are reduced to $(n_x\cdot n_P)\times(n_x\cdot n_P)$. 
Typically  $n_P$ is much smaller than $n_\Omega$.

We state the following theorems on the invertibility of $\bB_{11}$ and $\bK$, whose proof can be found in \Cref{ap-B11} and \Cref{ap-K}, respectively.

\begin{THM}\label{lem-B11}
	$\bB_{11}$ is invertible.
	Furthermore, if the quadrature rule is central symmetric, then
	$\bB_{11}$ is symmetric positive definite. Here, central symmetry means $\Omega_j$ and $-\Omega_j$ are
	both selected in the quadrature rule and their weights are equal $w_j = w_{-j}$. 
\end{THM}

\begin{THM}\label{thm-K}
	$\bK$ is invertible.
\end{THM}

\begin{REM}
	Typically, the linear system in such context is solved using the Krylov
	method, in which one needs to evaluate the multiplication 
	of a vector with $\bK$ in each iteration. We can use 
	the following formula to avoid repeated evaluation 
	in the left multiplication of $\bK$. 
	\begin{equation}\label{eq-nonrec-block}
	\bK \left[\begin{matrix}\bX_0\\\bX_{1} \end{matrix}\right]=
	\left[\begin{matrix}
	\bI_{n_x} &   \\
	&\bB_{11}\\
	\end{matrix}\right]
	\left[\begin{matrix}
	\bX_0 \\ \bX_1
	\end{matrix}\right]
	-
	\left[\begin{matrix}
	\bP_0    \\
	\bSigma\bL_{10}\\
	\end{matrix}\right]\bL_{00}^{-1} 
	\left(\bM_0\bX_0+\bL_{01}\bSigma^T\bX_1\right).
	\end{equation}
\end{REM}

\begin{REM}
	As demonstrated in \cite{heningburg2019hybrid}, the inversion of the block $\bL_{00}$ in \eqref{eq-lmdg-reducemat}, rather than the full matrix $\bL$ in \eqref{eq-sndg-phi}, results in a significant savings in terms of floating point operations (and hence time-to-solution).  This savings will be partially offset by the need to invert the matrix $\bB _{11}$ in \eqref{eq-lmdg-psi1}.  However, since the overall effect on time-to-solution depends heavily on the details of implementation, we do not investigate this aspect of the low-memory method in the  numerical results, but instead leave such an investigation to future work.
\end{REM}

\subsection{Reconstructed low-memory scheme}\label{sc-rlmdg}
Because the low-memory scheme couples the angular components of $\cV_{h,1}$, it is only first-order for fixed $\veps > 0$.  To recover second-order accuracy (formally), we introduce a spatial reconstruction procedure to approximate the anisotropic parts of $\cV_{h,1}$.

\subsubsection{Numerical scheme}\label{sc-rec}

We denote by
$\Pi_i$ the orthogonal projection from $\cV_{h}$ to $\cV_{h,i}$, $i = 0,1$.
The only information from the low-memory space $\cV_h^{\lm}$ retains from $v \in \cV_{h,1}$ is $\overline{\Pi_1(v)}$; the information contained in
$\Pi_1 (v) - \overline{\Pi_1 (v)}$ is missing. We therefore introduce an operator $R_\alpha^*v =
R_\alpha \Pi_0(v) - \overline{R_\alpha \Pi_0(v)}$, where $R_\alpha\Pi_0$ is an operator that returns the reconstructed 
slopes using piecewise constants and the boundary condition $\alpha$, to rebuild the difference.
Then the reconstructed scheme is written as: \textit{find $\psi_h \in \cV_h^\lm$ such that
	\begin{equation}\label{eq-rlmdg-vari}
	B(\psi_h + R_\alpha^*\psi_h, v) = Q_{\alpha}(v),\qquad 
	\forall v\in \cV_h^\lm.
	\end{equation}}
The reconstruction $\psi_h + R_\alpha^*\psi_h$ then gives a more accurate approximation to $\Psi$. 

Equivalently, by assembling all boundary terms into the right hand side, the
reconstructed scheme can also be formulated as a Petrov--Galerkin method with 
trial function space 
\begin{equation}\cV_h^{\rlm} = \{v + R_0^*v: v\in \cV_{h}^{\lm}\}.\end{equation}
Since $R_0^* 0 = 0$, $\cV_h^{\rlm}$ is in fact a linear space.  With this formulation, the reconstructed method solves the following problem: \textit{find $\psi_{h,R_0} \in
	\cV_{h}^{\rlm}$, such that 
	\begin{equation}\label{eq-rlmdg-scheme}
	B(\psi_{h,R_0},v) = \widetilde{Q}_{\alpha}(v), \qquad \forall v \in \cV_h^\lm.
	\end{equation}}

The use of different trial and test functions spaces make the analysis of this scheme less transparent.
Currently, we have no theoretical guarantee of unisolvency or the numerical diffusion limit.  We observe, however, that the method recovers second-order convergence for several different test problem across a wide range of $\veps$.

\

In this paper, we apply the reconstruction  suggested in \cite{heningburg2019hybrid} to recover slopes for simplicity, although in general other upwind approaches can also be used\footnote{For example, one can apply upwind reconstruction with wider stencils to improve the accuracy with an increased computational costs.
	Furthermore, the reconstruction can also be different at different spatial cells along different collocation angles, which may lead to an adaptive version of the reconstructed method. We postpone the discussion on numerical efficiency with different reconstruction methods to future work.}.
For illustration,
we consider a uniform Cartesian mesh
on $[0,1]\times [0,1]\times [0,1]$. The grid points are labeled from
$\frac{1}{2}$
to $n+\frac{1}{2}$ respectively. We denote by $u^0_{i,j,k}$ the cell average of
$u$ on the cell $K_{i,j,k}$ that 
centers at $(x_i, y_j, z_k)$. Along each direction
$\Omega = (\Omega_x,\Omega_y,\Omega_z)$, 
\begin{equation}(R_\alpha\Pi_0 (u))|_{K_{i,j,k}} = (\delta_x^{s_x}u^0_{i,j,k}) (x - x_i) +
(\delta_y^{s_y}u^0_{i,j,k}) (y - y_j) + (\delta_z^{s_z}u^0_{i,j,k}) (z -
z_k),\end{equation}
where $s_x = - \sign (\Omega_x)$,
\begin{align}
\delta^-_x u^0_{i,j,k} &= 
\begin{cases}
\dfrac{u^0_{i,j,k} -	u^{0}_{i-1,j,k}}{h}, 
& 2\leq i\leq n, \\
\dfrac{u^{0}_{1,j,k}-\alpha(\Omega,(0,y_j,z_k))}{h/2},
& i = 1,
\end{cases} \\
\delta^+_x u^0_{i,j,k} &= 
\begin{cases}
\dfrac{u^{0}_{i+1,j,k} - u^0_{i,j,k}}{h}
& 1\leq i\leq n-1, \\
\dfrac{\alpha(\Omega,(1,y_j,z_k))-u^{0}_{n,j,k}}{h/2},
& i = n.
\end{cases} 
\end{align}
$\delta_y^\pm$ and $\delta_z^\pm$ are defined similarly. For numerical results in the next section, we only reconstruct the $P^1$ slopes to recover the second-order accuracy; 
$Q^1$ type reconstruction gives similar results in terms of the
convergence rate.
\subsubsection{Implementation}

Let $\mathbb{B}_0^\rlm= \{\xi^{l,p,0}+R_0^*\xi^{l,p,0}:l=1,\dots,n_\Omega, p =1,\dots,n_x\}$
and $\cV_h^{\rlm} = \tspan \{\mathbb{B}_0^\rlm,\mathbb{B}_1^\lm\}$. 
As in the first-order method, the total degrees of freedom is $n_\Omega\cdot n_x+ n_x\cdot(n_P-1)$. The
boundary terms are assembled into a vector $\br_\alpha$. Here we use 
$\bPsi = \left[\bPsi_0, \bPsi_1 \right]^T
$
to represent the solution of the reconstructed method,  where $\bPsi_1 = \bSigma^T\bPhi_1+(\bI_{n_\Omega\cdot
	n_x\cdot(n_P-1)}-\bSigma^T\bP_1)(\bR
\bPsi_0 +\br_\alpha)$. Note $\bP_1\bSigma^T = \mathbf{I}_{n_x\cdot (n_P-1)}$, which implies $\bP_1\bPsi_1 = \bPhi_1$. The block matrix form can then be written as follows.
\begin{subequations}\label{eq-rlmdg-block}
	\begin{eqnarray}
	\bL_{00}\bPsi_0+\bL_{01}\bPsi_1&=& \bM_0\bP_0\bPsi_0 +
	\bQ_0,\\
	\bSigma\bL_{10}\bPsi_0+\bSigma\bL_{11}
	\bPsi_1 &=&
	\bSigma\bM_1\bPhi_1 + \bSigma\bQ_1.
	\end{eqnarray}
\end{subequations}
With
\begin{eqnarray}
\widetilde{\bL}_{00} &=& \bL_{00} + \bL_{01} \bR,\\
\widetilde{\bL}_{10} &=& \bL_{10} + \bL_{11}(\bI_{n_\Omega\cdot n_x\cdot(n_P-1)}-\bSigma^T\bP_1)\bR,\\
\widetilde{\bQ}_0 &=& \bQ_0 - \bL_{01}(\bI_{n_\Omega\cdot
	n_x\cdot(n_P-1)}-\bSigma^T\bP_1)\br_\alpha,\\
\widetilde{\bQ}_1 &=& \bQ_1 - \bL_{11}(\bI_{n_\Omega\cdot
	n_x\cdot(n_P-1)}-\bSigma^T\bP_1)\br_\alpha,
\end{eqnarray}
one can rewrite \eqref{eq-rlmdg-block} as
\begin{subequations}\label{eq-rlmdg-block-tilde}
	\begin{eqnarray}
	\widetilde{\bL}_{00}\bPsi_0+\bL_{01}\bSigma^T\left(\bPhi_1 -\bP_1\bR
	\bPsi_0 \right)&=& \bM_0\bP_0\bPsi_0 +
	\widetilde{\bQ}_0,\label{eq-rlmdg-block0}\\
	\bSigma\widetilde{\bL}_{10}\bPsi_0+\bSigma\bL_{11}\bSigma^T\bPhi_1
	&=&
	\bSigma\bM_1\bPhi_1 + \bSigma\widetilde{\bQ}_1.\label{eq-rlmdg-block1}
	\noeqref{eq-rlmdg-block0,eq-rlmdg-block1}
	\end{eqnarray}
\end{subequations}
We follow the procedure as before to reduce the system dimension.

\

1. Solve for $\bPhi_1$ in terms of $\bPsi_0$ through \eqref{eq-rlmdg-block1}: 
\begin{equation}\label{eq-rlmdg-psi1}
\bPhi_1 = \bB_{11}^{-1} \bSigma\left(-\widetilde{\bL}_{10} \bPsi_0 +
\widetilde{\bQ}_1\right), \qquad \bB_{11} = \bSigma {\bL}_{11}\bSigma^T -
\bSigma\bM_1.
\end{equation}

2. Substitute $\bPhi_1$ from \eqref{eq-rlmdg-psi1} into \eqref{eq-rlmdg-block0} to
obtain a closed equation for $\bPsi_0$:
\begin{equation}\label{eq-rlmdg-psi0}
\begin{aligned}
\bPsi_0 
- \widetilde{\bL}_{00}^{-1} \bM_0 (\bP_0 \bPsi_0) 
&-\widetilde{\bL}_{00}^{-1}  {\bL}_{01}\bSigma^T(\bB_{11}^{-1}
\bSigma\widetilde{\bL}_{10} +\bP_1
\bR)\bPsi_0\\
&= \widetilde{\bL}_{00}^{-1} (\widetilde{\bQ}_0-
{\bL}_{01}\bSigma^T\bB_{11}^{-1}\bSigma\widetilde{\bQ}_1).
\end{aligned}
\end{equation}

3. Applying $\bP_0$ and $\bSigma\widetilde{\bL}_{10}$ to \eqref{eq-rlmdg-psi0}, to obtain a closed system for $\bX_0 = \bP_0\bPsi_0$ and $\bX_1 = (\bB_{11}^{-1}\bSigma\widetilde{\bL}_{10}+\bP_1\bR)\bPsi_0$:
\begin{equation}\label{eq-rlmdg-reducemat}
\begin{aligned}
\widetilde{\bK}
\left[\begin{matrix}\bX_0\\\bX_{1} \end{matrix}\right]
&=\left[\begin{matrix}
\bP_0\\
\bSigma\bL_{10}
\end{matrix}\right]\widetilde{\bL}_{00}^{-1} (\widetilde{\bQ}_0-
{\bL}_{01}\bSigma^T\bB_{11}^{-1}\bSigma\widetilde{\bQ}_1),
\end{aligned}
\end{equation}		
where
\begin{equation}
\widetilde{\bK} =	\left[\begin{matrix}
\bI_{n_x}- \bP_0\widetilde{\bL}_{00}^{-1} \bM_0 &
-\bP_0\widetilde{\bL}_{00}^{-1}{\bL}_{01}\bSigma^T\\
- \bSigma\widetilde{\bL}_{10}\widetilde{\bL}_{00}^{-1} \bM_0 &\bB_{11}
-\bSigma\widetilde{\bL}_{10}\widetilde{\bL}_{00}^{-1}{\bL}_{01}\bSigma^T \\
\end{matrix}\right].
\end{equation}

4. Solve for $\bX_0$ and $\bX_1$ in \eqref{eq-rlmdg-reducemat}. 
Use \eqref{eq-rlmdg-psi0} and \eqref{eq-rlmdg-psi1} to recover $\bPsi$.
\begin{eqnarray}
\bPsi_0 &=& \widetilde{\bL}_{00}^{-1}  \bL_{01}\bX_0 + \widetilde{\bL}_{00}^{-1}
\bM_0 \bX_1 + \bL_{00}^{-1} (\widetilde{\bQ}_0-
{\bL}_{01}\bSigma^T\bB_{11}^{-1}\bSigma\widetilde{\bQ}_1),\\
\bPhi_1 &=& \bB_{11}^{-1}\bSigma \left(-\widetilde{\bL}_{10} \bPsi_0 +
\widetilde{\bQ}_1\right).
\end{eqnarray}

\

As the first-order method, only Step 4 is used in the implementation. Since only upwind information is used, $\widetilde{\bL}_{00}$ is invertible and can be inverted with sweeps along each
angular direction. Note $\bB_{11}$ is invertible, as has been pointed out in \Cref{ap-B11}. One can follow the argument in \Cref{ap-K} to show $\widetilde{\bK}$ is invertible if the scheme \eqref{eq-rlmdg-vari} is unisolvent.

\section{Numerical tests}\label{sc-num}
\setcounter{equation}{0}
\setcounter{figure}{0}
\setcounter{table}{0}
In this section, we present numerical tests to examine performance of the methods.

\subsection{One dimensional tests (slab geometry)} 
In slab geometries, the radiative transport equation takes the form (see, e.g., \cite[Page 28]{Lewis-Miller-1984}). 
\begin{align}\label{eq-1d}
\mu \partial_x\psi (\mu,x) + \left(\frac{\sig{s}}{\veps} +\veps
\sig{a}\right) \psi(\mu,x) &=
\frac{\sig{s}}{2\veps}\int_{-1}^1\psi(\mu',x)d\mu' + \veps
q(x),\\
\psi(\mu,x_a) = \psi_l(\mu), \text{ if }\mu\geq0, &\quand \psi(\mu,x_b) = \psi_r(\mu), \text{ if }\mu<0,
\end{align}
where $x\in [x_a,x_b]$ and $\mu \in [-1,1]$. We will compare the $S_N$-$P^0$-DG scheme, $S_N$-$P^1$-DG scheme, low-memory
scheme (LMDG)
and the reconstructed scheme (RLMDG). Numerical error is evaluated in $L^1$ norm.
\begin{example}\label{examp-1dfab} 
	We first examine convergence rates of the methods using fabricated solutions. Let $\veps = 1$, $\sig{s} = 1$, 
	$\sig{a} = 1$ and $D = [0,1]$.
	Assuming the exact solution $\psi$, we compute the source term 
	$q$ and the inflow boundary conditions $\psi_l$ and $\psi_r$ accordingly. 
	With this approach, it may happen that $q$ depends on $\mu$.  We use the $32$ points Gauss quadrature on $[-1,1]$ for $S_N$ discretization.
	
	We consider the case $\psi = \cos x$ and $\psi = \cos(x+\mu)$. The 
	results are documented in \Cref{tab-1d-aniso}. 
	When $\psi$ is isotropic, the low-memory scheme
	exhibits second-order convergence. 
	For the anisotropic case, the LMDG scheme
	degenerates to first-order accuracy, while the RLMDG scheme
	remains second-order accurate. 
	\begin{table}[h!]
		\footnotesize
		\centering
		\begin{tabular}{c|c|c|c|c|c|c|c|c|c}
			\hline
			&&\multicolumn{2}{c|}{$P^0$-DG}&\multicolumn{2}{c|}{$P^1$-DG}&\multicolumn{2}{c|}{LMDG
			}&\multicolumn{2}{c}{RLMDG}\\
			\hline
			$\psi$&$h$&error&order&error&order&error&order&error&order\\
			\hline
			\multirow{2}{*}{$\cos x$} 
			&$1/20$ &3.79e-3&    - & 2.48e-5&    -& 2.24e-5&    -& 9.14e-5&
			-  \\
			&$1/40$ &1.91e-3& 0.99 & 6.27e-6& 1.98& 5.62e-6& 2.00& 2.31e-5&
			1.99\\
			&$1/80$ &9.55e-4& 1.00 & 1.58e-6& 1.99& 1.41e-6& 2.00& 5.80e-6&
			1.99\\
			&$1/160$&4.78e-4& 1.00 & 3.96e-7& 2.00& 3.52e-7& 2.00& 1.45e-6&
			2.00\\
			\hline
			\multirow{2}{*}{$\cos (x+\mu)$} 
			&$1/20$&  4.70e-3&     -& 2.11e-5&   -& 3.20e-3&    -& 7.74e-5&
			- \\ 
			&$1/40$&  2.36e-3&  0.99& 5.36e-6&1.98& 1.60e-3& 1.00& 1.95e-5&
			1.99\\
			&$1/80$&  1.18e-3&  1.00& 1.35e-6&1.99& 8.01e-4& 1.00& 4.89e-6&
			1.99\\
			&$1/160$& 5.91e-4&  1.00& 3.39e-7&1.99& 4.01e-4& 1.00& 1.23e-6&
			2.00\\
			\hline
		\end{tabular}
		\caption{Accuracy test for \Cref{examp-1dfab}.}\label{tab-1d-aniso}
	\end{table}
	
	\begin{figure}
		\begin{subfigure}[b]{0.3\textwidth}
			\includegraphics[width=\textwidth]{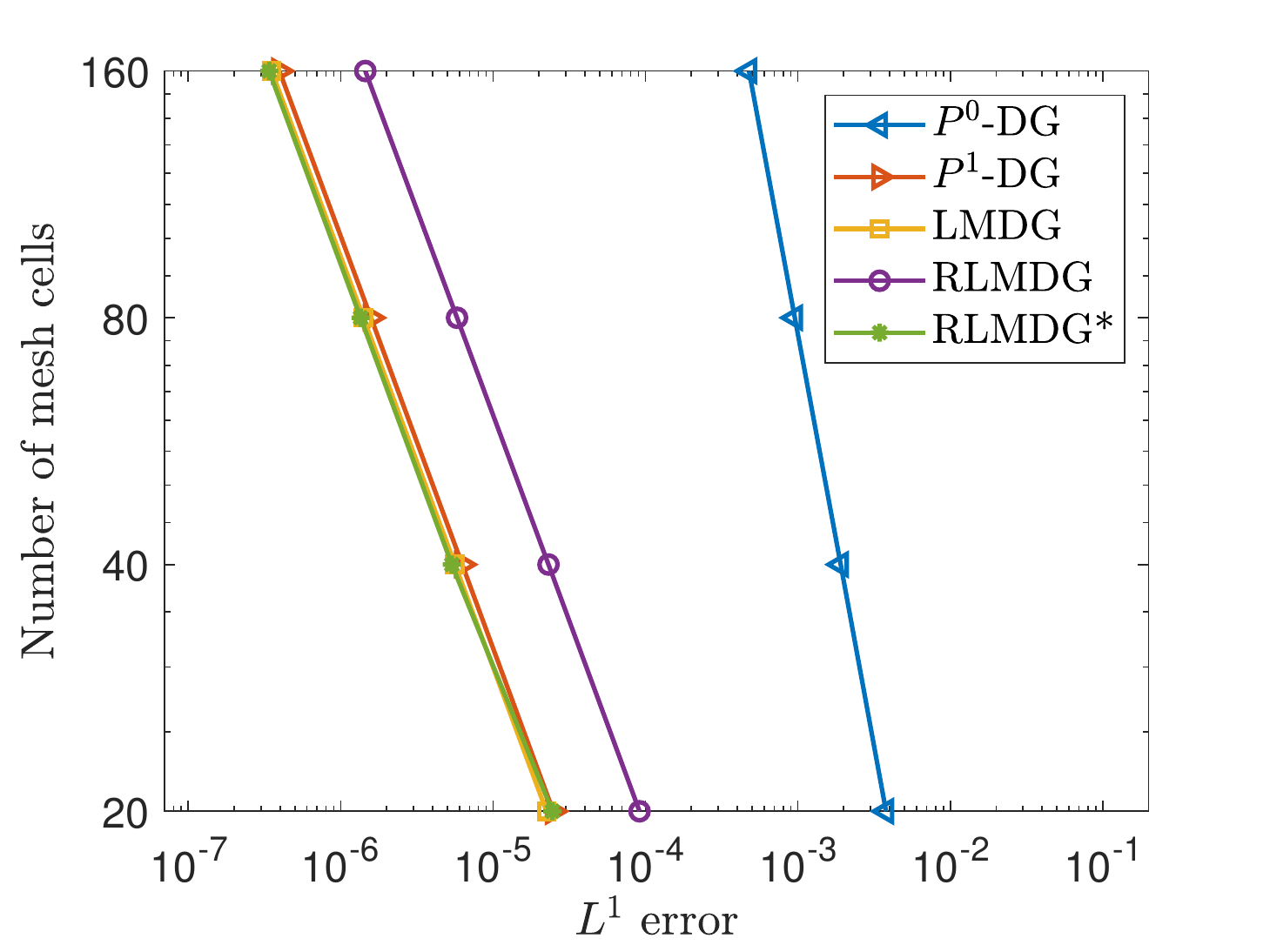}  
			\caption{Number of mesh cells.}
		\end{subfigure}
		~
		\begin{subfigure}[b]{0.3\textwidth}
			\includegraphics[width=\textwidth]{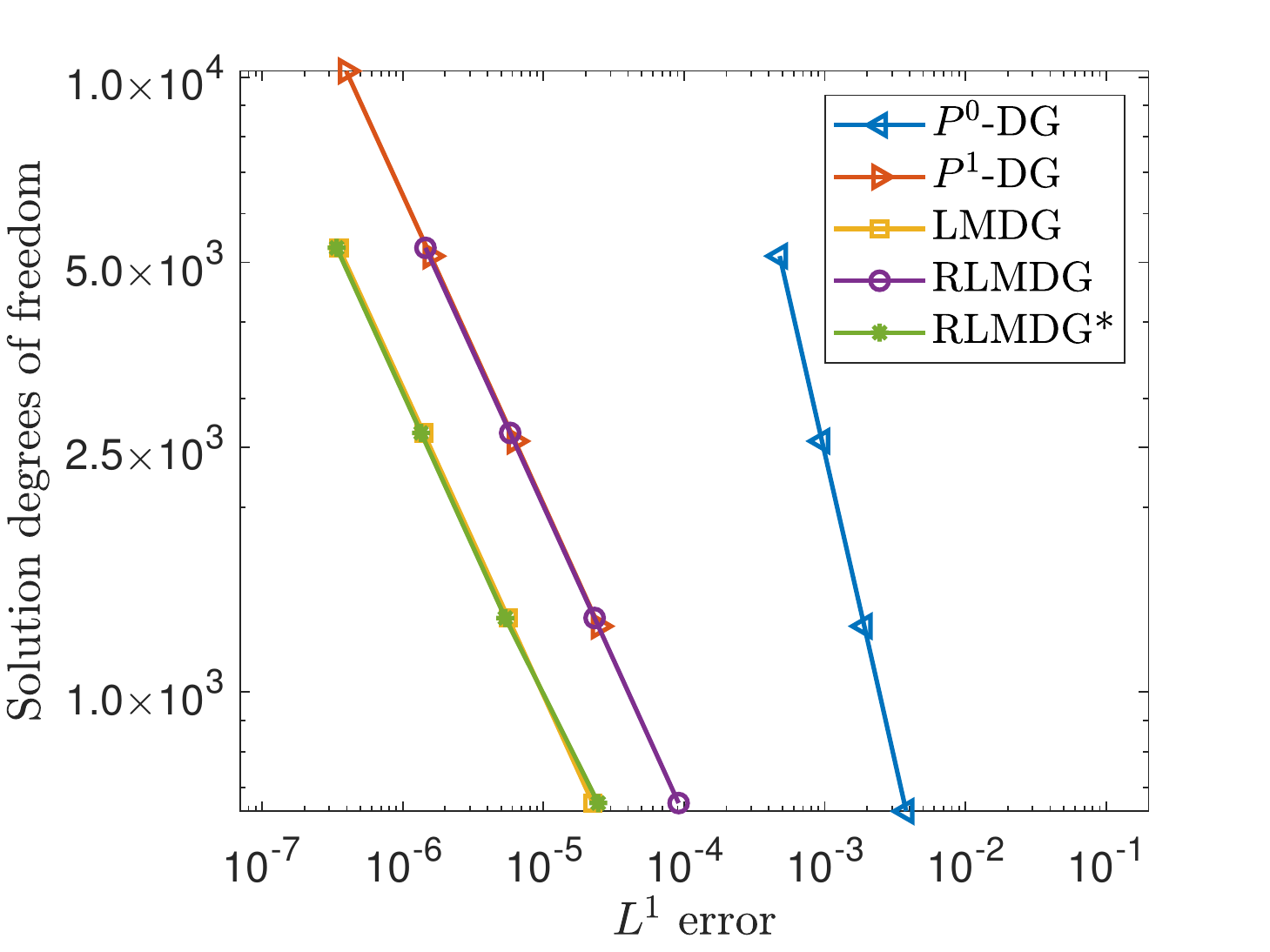}  
			\caption{Degrees of freedom.}
		\end{subfigure}
		~
		\begin{subfigure}[b]{0.3\textwidth}
			\includegraphics[width=\textwidth]{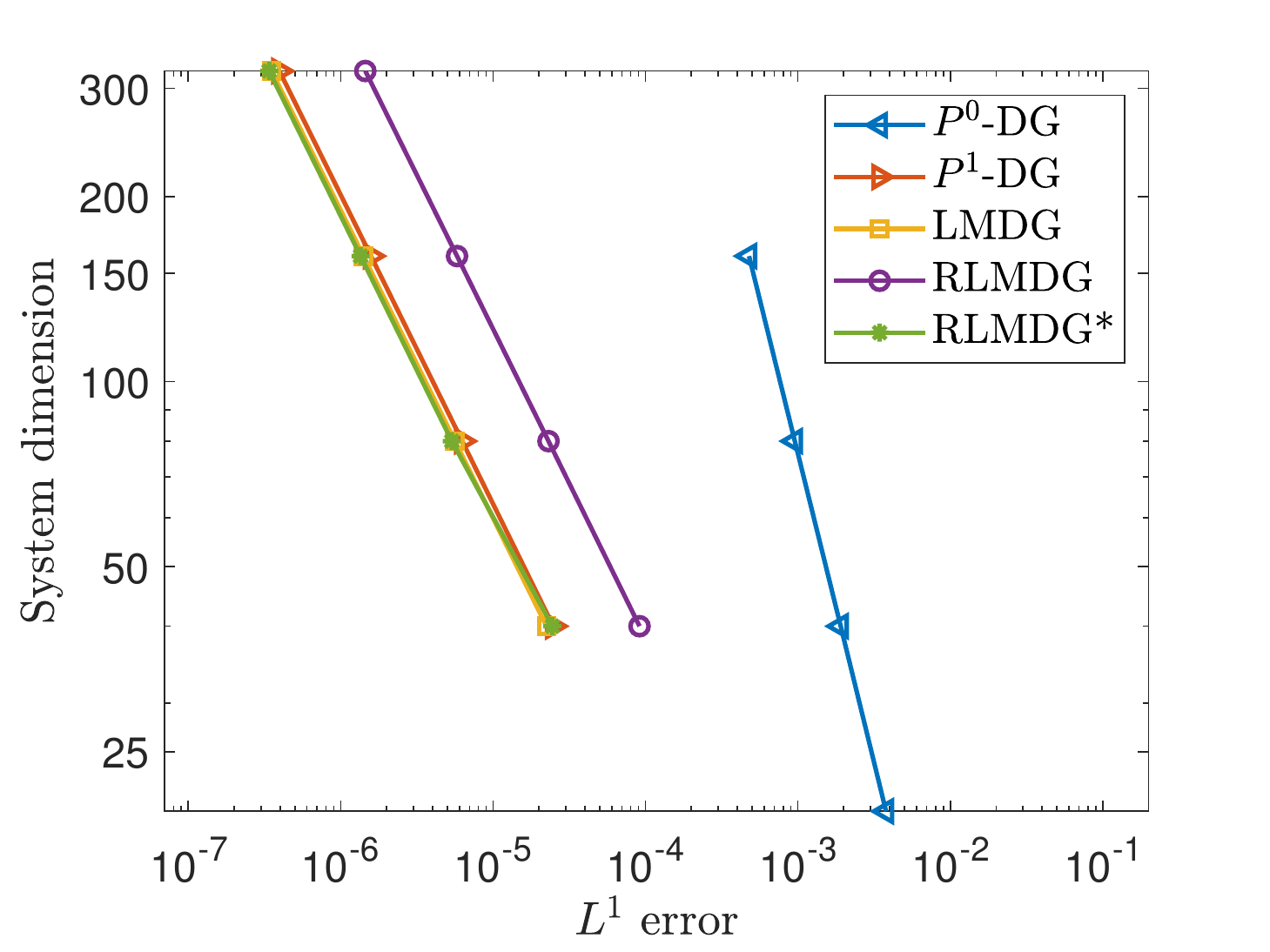}  
			\caption{System dimension.}
		\end{subfigure}
		\\
		\begin{subfigure}[b]{0.3\textwidth}
			\includegraphics[width=\textwidth]{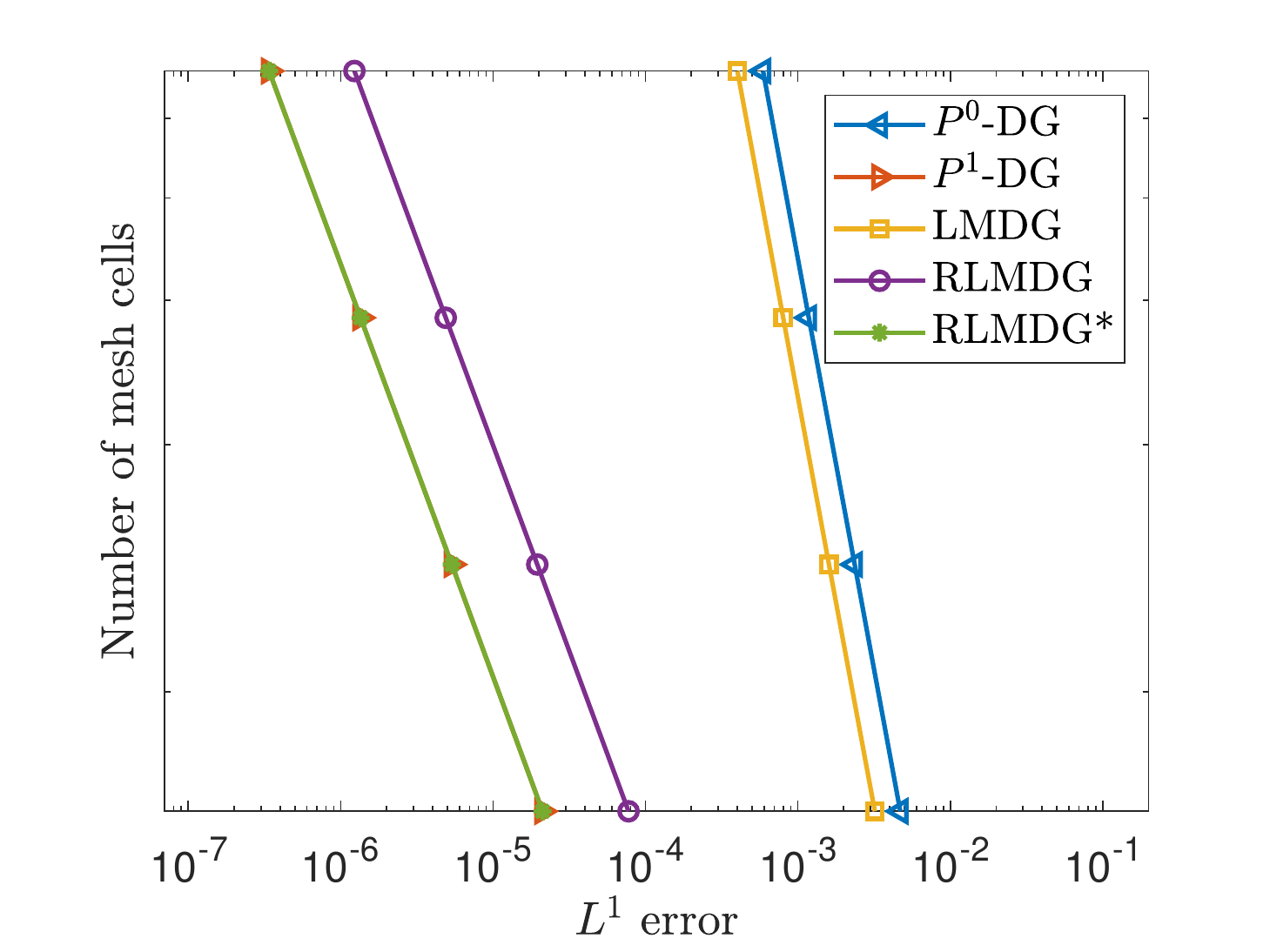}  
			\caption{Number of mesh cells.}
		\end{subfigure}
		~
		\begin{subfigure}[b]{0.3\textwidth}
			\includegraphics[width=\textwidth]{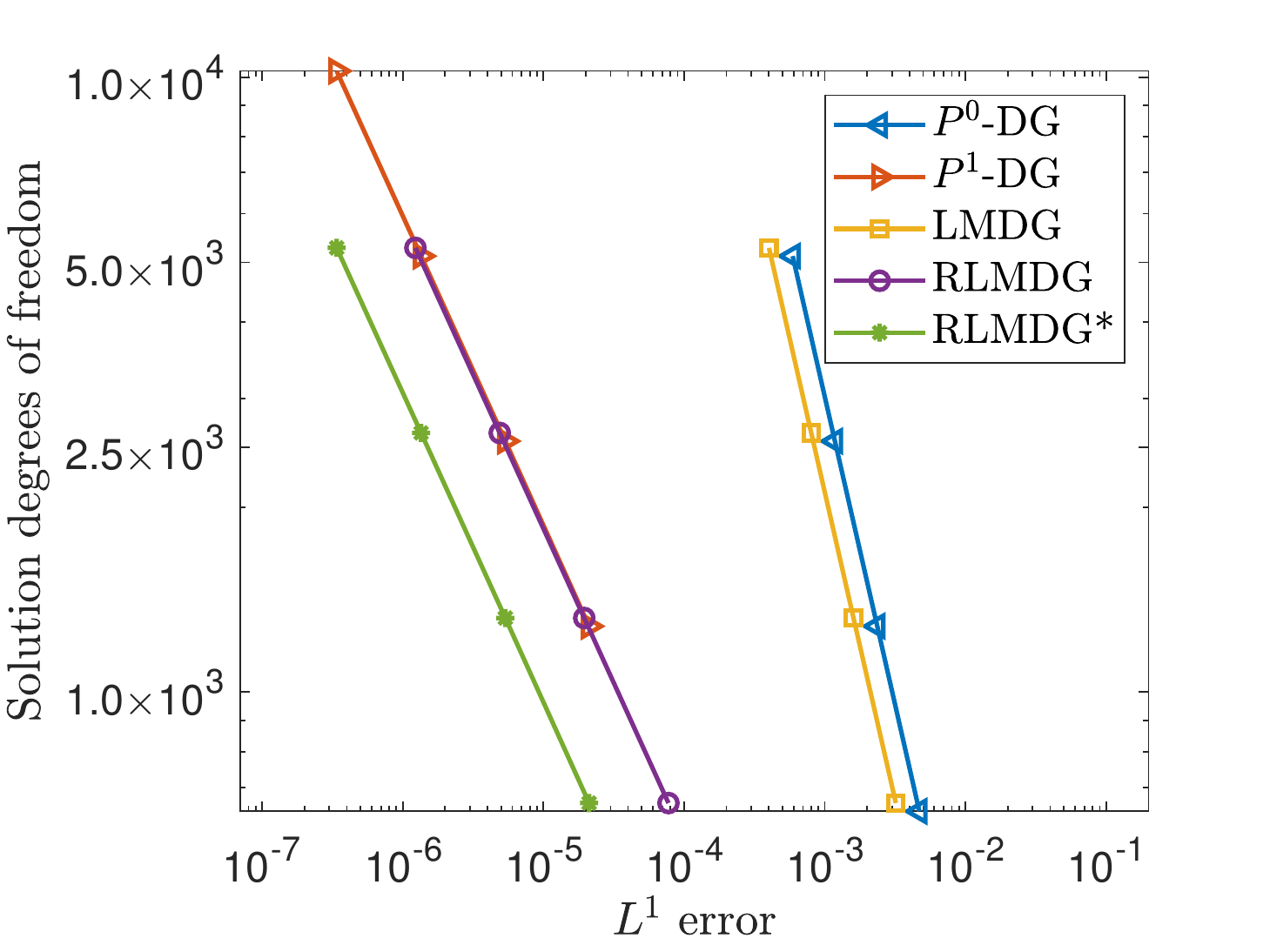}  
			\caption{Degrees of freedom.}
		\end{subfigure}
		~
		\begin{subfigure}[b]{0.3\textwidth}
			\includegraphics[width=\textwidth]{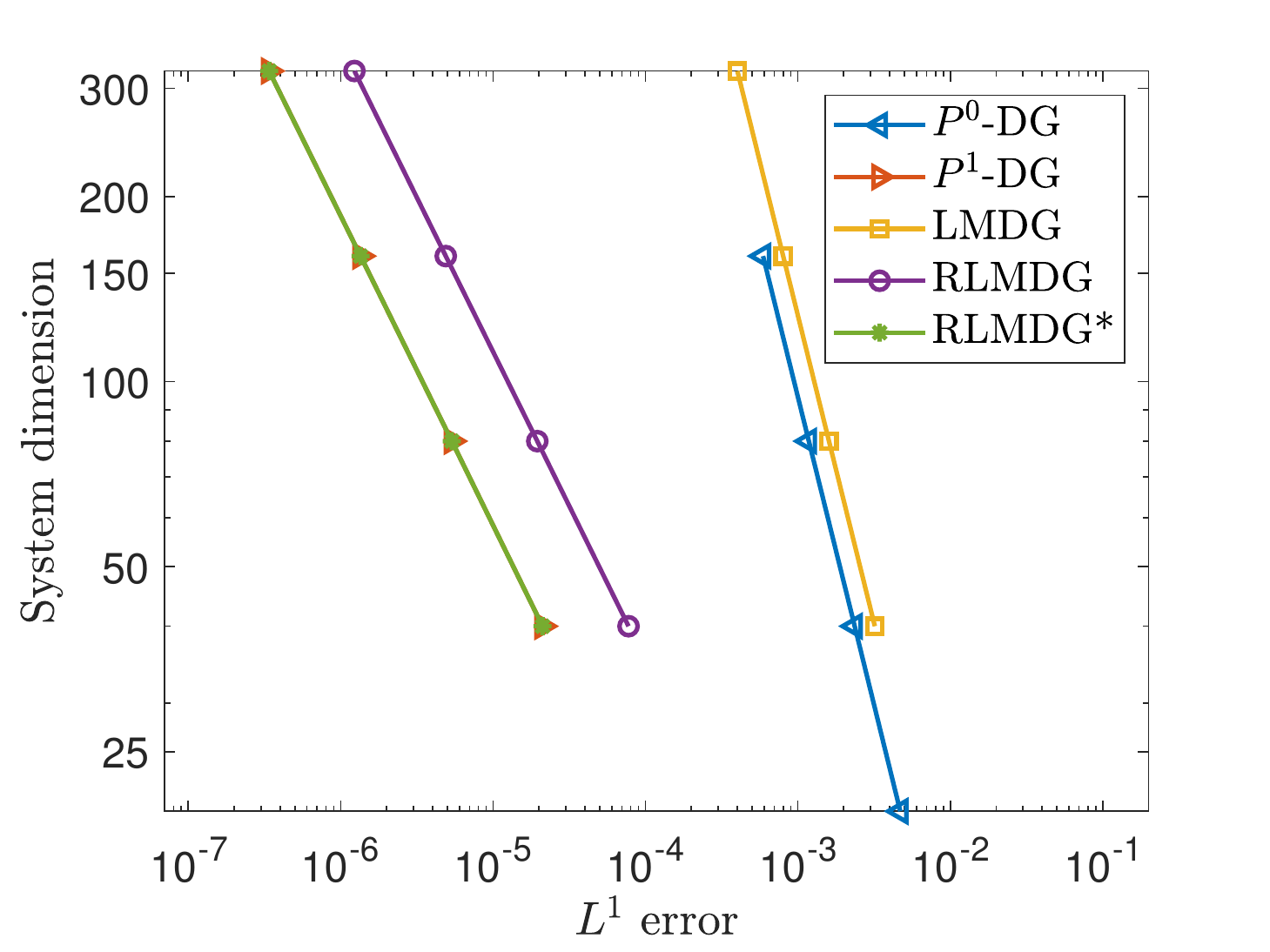}  
			\caption{System dimension.}
		\end{subfigure}
		\caption{Numerical efficiency in \Cref{examp-2daccu}. The first row is for isotropic test $\psi = \cos(x)$ and the second row is for anisotropic test $\psi = \cos(x+\mu)$. }\label{fig-efficiency}
	\end{figure}
	
	To better understand numerical efficiency, we analyze results in \Cref{tab-1d-aniso} by plotting $L^1$ error  versus number of mesh cells, total degrees of freedom of the solution (memory costs), and number of equations in the reduced linear system (either \eqref{eq-sndg-phi}, \eqref{eq-lmdg-reducemat}, or \eqref{eq-rlmdg-reducemat}).
	
	For the LMDG method, when the solution is isotropic, the method uses similar number of mesh cells as the $P^1$-DG method to reach the same accuracy. As a result, a reduced linear system of similar size is solved, but the degrees of freedom is smaller. For the anisotropic case, the LMDG method is first-order accurate. Compared with the $P^0$-DG method, it is able to reach similar accuracy on a coarser mesh. The reduced linear system is larger, but the number of degrees of freedom is indeed smaller. 
	
	For the RLMDG method, it seems to be less accurate compared with $P^1$-DG method, and a finer mesh has to be used to achieve the same accuracy. As a result, the solution degrees of freedom is similar to that of the $P^1$-DG method but the reduced system is even larger. However, we point out a more accurate reconstruction may solve this problem. For example, instead of using two cells, one can recover slopes in the interior region with a three-cell upwind reconstruction (which we call RLMDG$^*$). This new method is still second-order accurate, but its error is comparable to the $P^1$-DG method and significantly smaller than the current RLMDG method.  Efficiency results for RLMDG$^*$ are depicted by green lines in \Cref{fig-efficiency} (they overlap with red lines in (d) and (f)). These results show that RLMDG$^*$ yields reduced systems of similar size to those of the $P^1$-DG method, but it uses less overall memory.
\end{example}
\begin{example}\label{examp-1ddifflim}
	In the second numerical test, we examine the convergence rate and asymptotic preserving property of the methods. Let $\sig{s} = \sig{a} = 1$ in \eqref{eq-1d}. The computational domain
	is set as $D =[0,{\pi}]$. We take $\psi_l =
	\psi_r
	= 0$ and $q = \frac{4}{3}\sin(x)$. The $32$-point Gauss quadrature 
	is used for $S_N$ discretization.
	
	Numerical error at $\veps = 10^{-5}$ and $\veps = 1$ is listed in \Cref{tab-1d-eps10-5}, respectively. 
	The reference solutions are set as the numerical solutions 
	with $P^1$-DG scheme on a mesh with $1280$ cells.
	One can see from \Cref{tab-1d-eps10-5}, the LMDG 
	scheme exhibits second-order convergence rate at $\veps = 10^{-5}$, when the solution is
	almost isotropic, while it converges at a first-order rate when $\veps = 1$. 
	The RLMDG method is second-order in both cases.
	
	Solution profiles of different schemes on a sparse
	uniform mesh, with $h = \pi/8$, are shown in \Cref{fig-1dguermond-1}. When $\veps = 10^{-5}$, both LMDG and RLMDG methods preserve the correct diffusion limit, unlike the $P^0$-DG method. When $\veps = 1$, all schemes give valid approximations.
	\begin{table}[h!]
		\footnotesize
		\centering
		\begin{tabular}{c|c|c|c|c|c|c|c|c|c}
			\hline
			&&\multicolumn{2}{c|}{$P^0$-DG}&\multicolumn{2}{c|}{$P^1$-DG}&\multicolumn{2}{c|}{LMDG}
			&\multicolumn{2}{c}{RLMDG}\\
			\hline
			$\veps$&$h$&error&order&error&order&error&order&error&order\\
			\hline
			\multirow{2}{*}{$10^{-5}$} 
			&$1/20$&  2.00e-0& -  &1.89e-3&-&	1.89e-3&-&7.52e-3&-\\
			&$1/40$&  2.00e-0&0.00&4.70e-4&2.01&4.71e-4&2.00&1.88e-3&2.00\\
			&$1/80$&  2.00e-0&0.00&1.17e-4&2.01&1.16e-4&2.03&4.70e-4&2.00\\
			&$1/160$& 1.99e-0&0.00&2.91e-5&2.00&3.06e-5&1.92&1.17e-4&2.00\\
			\hline
			\multirow{2}{*}{$1$} 
			&$1/20$&  1.06e-1&-&	2.91e-3&-	&3.08e-2&-&		9.55e-3& -\\
			&$1/40$&  5.38e-2&0.98& 7.72e-4&1.92&1.59e-2&0.95&2.60e-3&1.88\\
			&$1/80$&  2.71e-2&0.99& 1.99e-4&1.95&8.09e-3&0.98&6.90e-4&1.91\\
			&$1/160$& 1.35e-2&1.00& 5.03e-5&1.99&4.08e-3&0.99&1.80e-4&1.94\\
			
			\hline
		\end{tabular}
		\caption{Accuracy test for \Cref{examp-1ddifflim}.}\label{tab-1d-eps10-5}
	\end{table}
	\begin{figure}[h!]
		\centering
		\begin{subfigure}[b]{0.45\textwidth}
			\includegraphics[width=\textwidth]{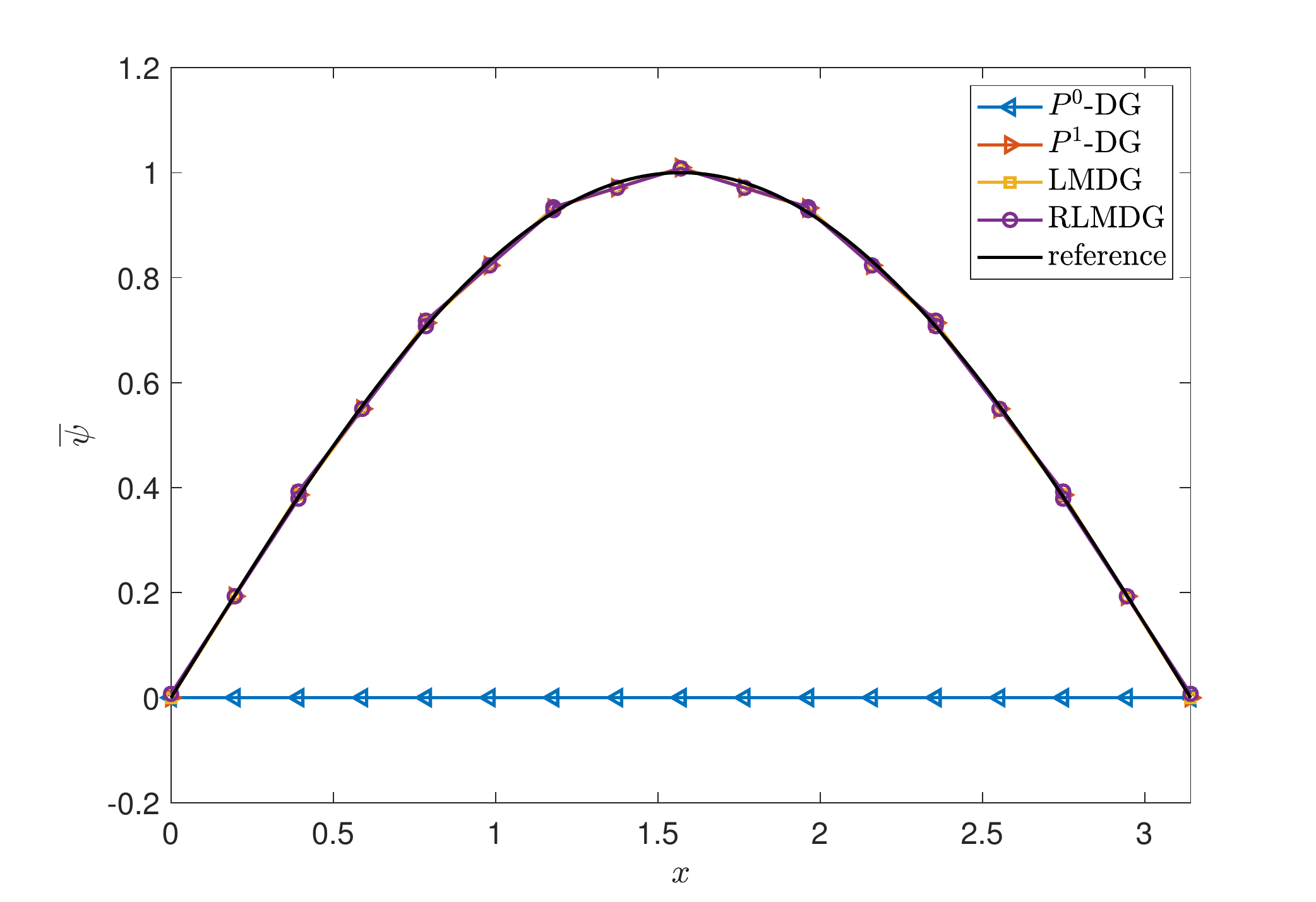}
			\caption{$\veps = 10^{-5}$.}
		\end{subfigure}
		~
		\begin{subfigure}[b]{0.45\textwidth}
			\includegraphics[width=\textwidth]{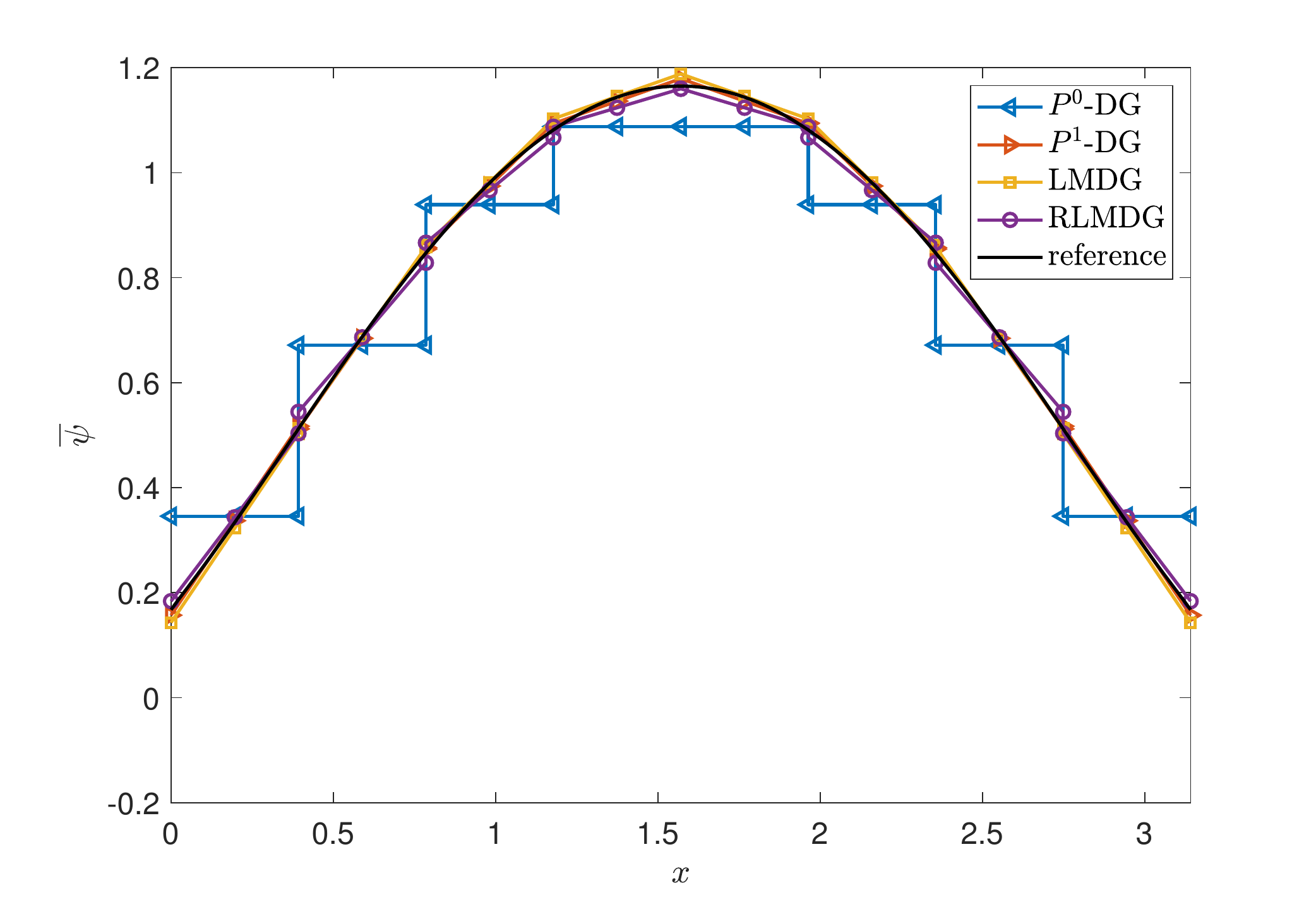}  
			\caption{$\veps = 1$.}
		\end{subfigure}
		\caption{Profiles of numerical scalar fluxes in \Cref{examp-1ddifflim}.}
	\end{figure}
\end{example}

\begin{example}\label{examp-guermond}
	We then consider a test
	from \cite{ragusa2012robust} with discontinuous cross-sections. 
	The problem is defined on $[0,1]$ and is purely scattering, i.e., 
	$\sig{a}\equiv 0$. The cross-section is $\sig{s} = \sig{s,1}=100$ 
	on the left part of the domain $[0,0.5]$, and is $\sig{s} = \sig{s,2} =
	100, 1000 \text{ or } 10000$ on the right part $[0.5,1]$. 
	The source term is constant $q = 0.01$. In the numerical test, 
	we set the mesh size to be $h =0.1$ and $h = 0.02$, and solutions are
	depicted in \Cref{fig-1dguermond-1} and \Cref{fig-1dguermond-2}, respectively. 
	As one can see, unlike the $P^0$-DG scheme, 
	both LMDG and RLMDG schemes provide correct
	solution profiles. Since the problem is diffusive, the LMDG scheme
	gives accurate approximations that are almost indistinguishable with the
	$P^1$-DG solutions. The reconstructed scheme has difficulty resolving the
	kink at $x = 0.5$, likely because the reconstruction is no longer accurate at this point.  This artifact can indeed be alleviated as the mesh is refined comparing \Cref{fig-1dguermond-1} and \Cref{fig-1dguermond-2}.
	\begin{figure}[h!]
		\begin{subfigure}[b]{0.3\textwidth}
			\includegraphics[width=\textwidth]{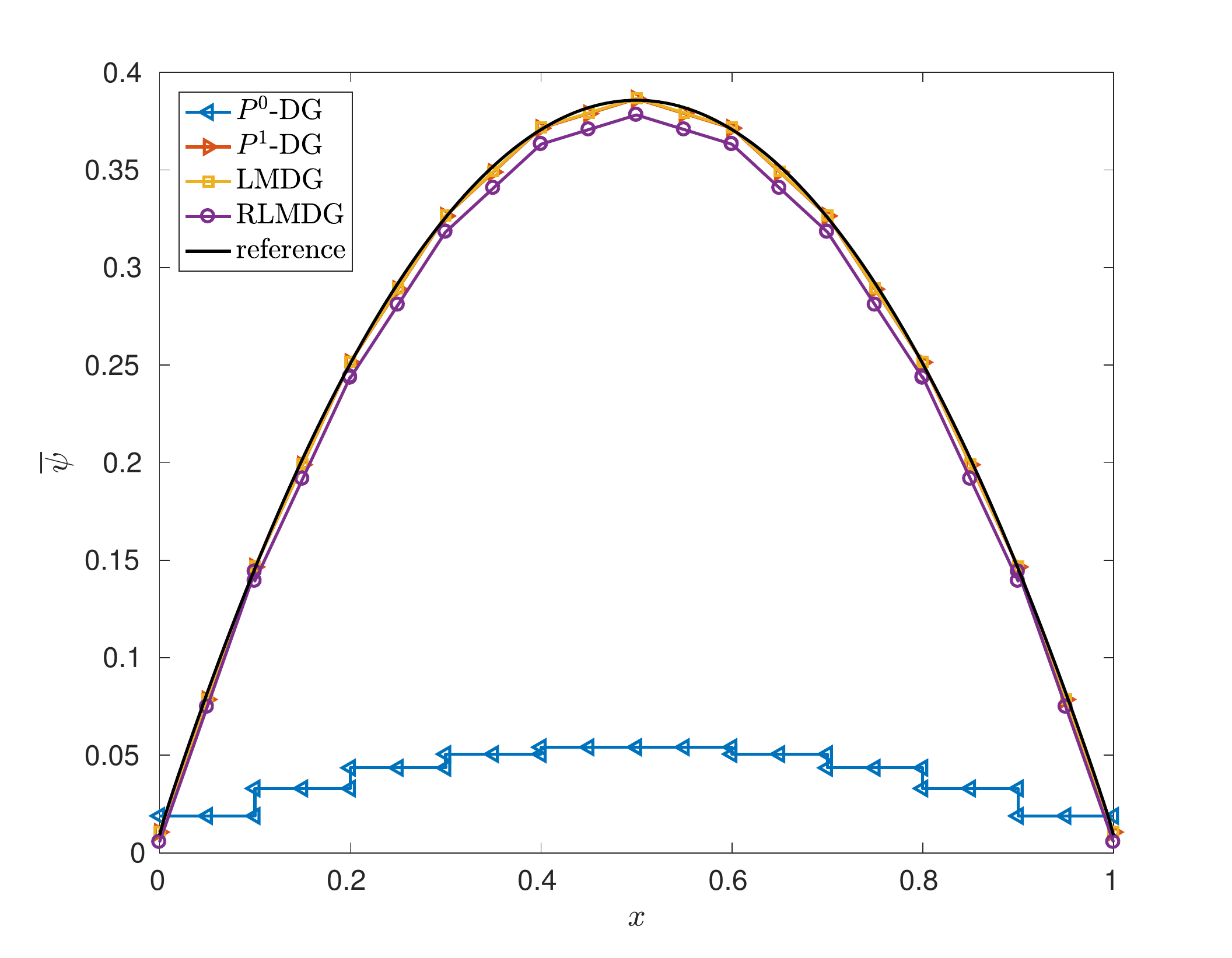}  
			\caption{$\sig{s,2} = 100$.}
		\end{subfigure}
		~
		\begin{subfigure}[b]{0.3\textwidth}
			\includegraphics[width=\textwidth]{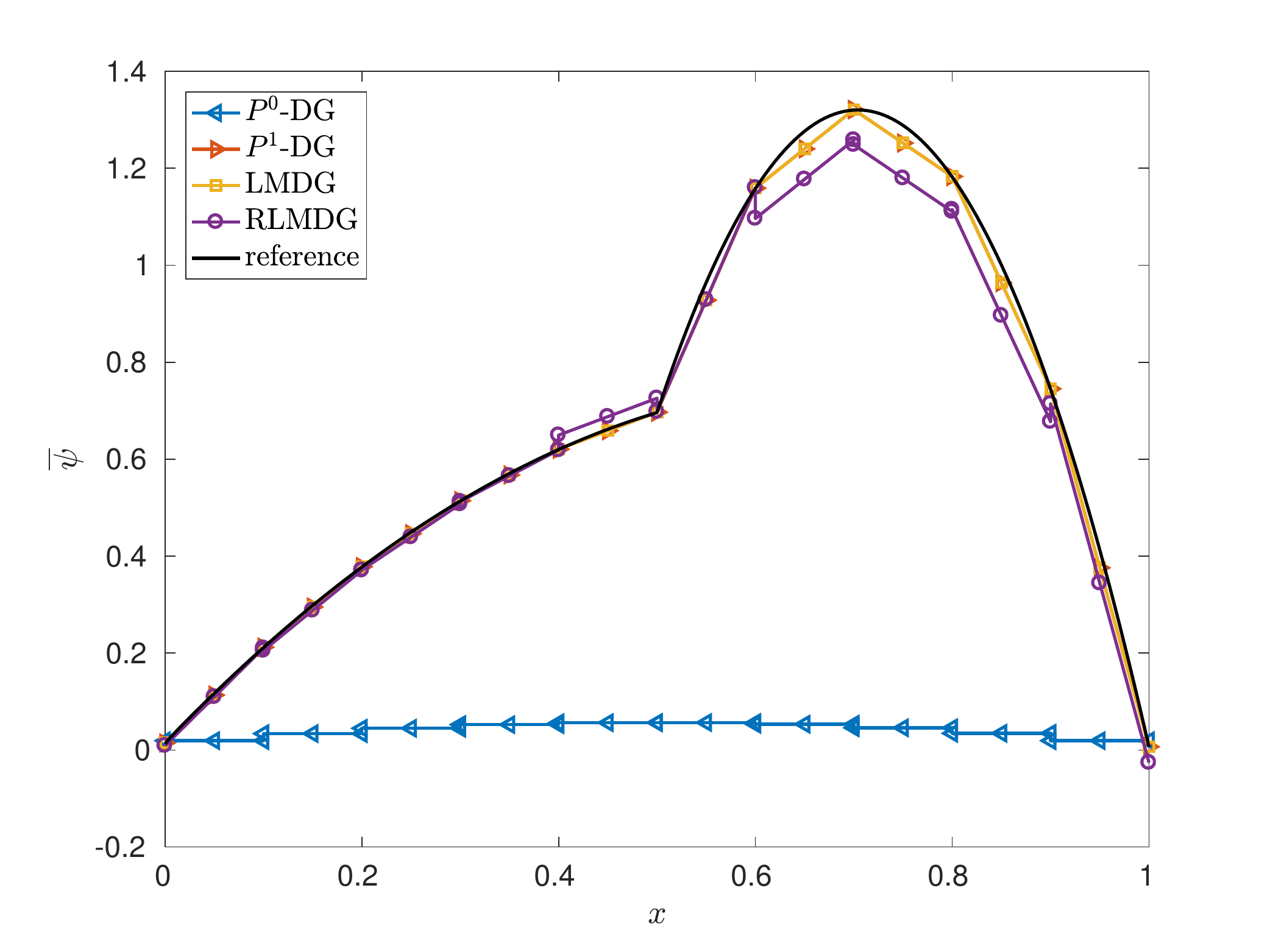}  
			\caption{$\sig{s,2} = 1000$.}
		\end{subfigure}
		~
		\begin{subfigure}[b]{0.3\textwidth}
			\includegraphics[width=\textwidth]{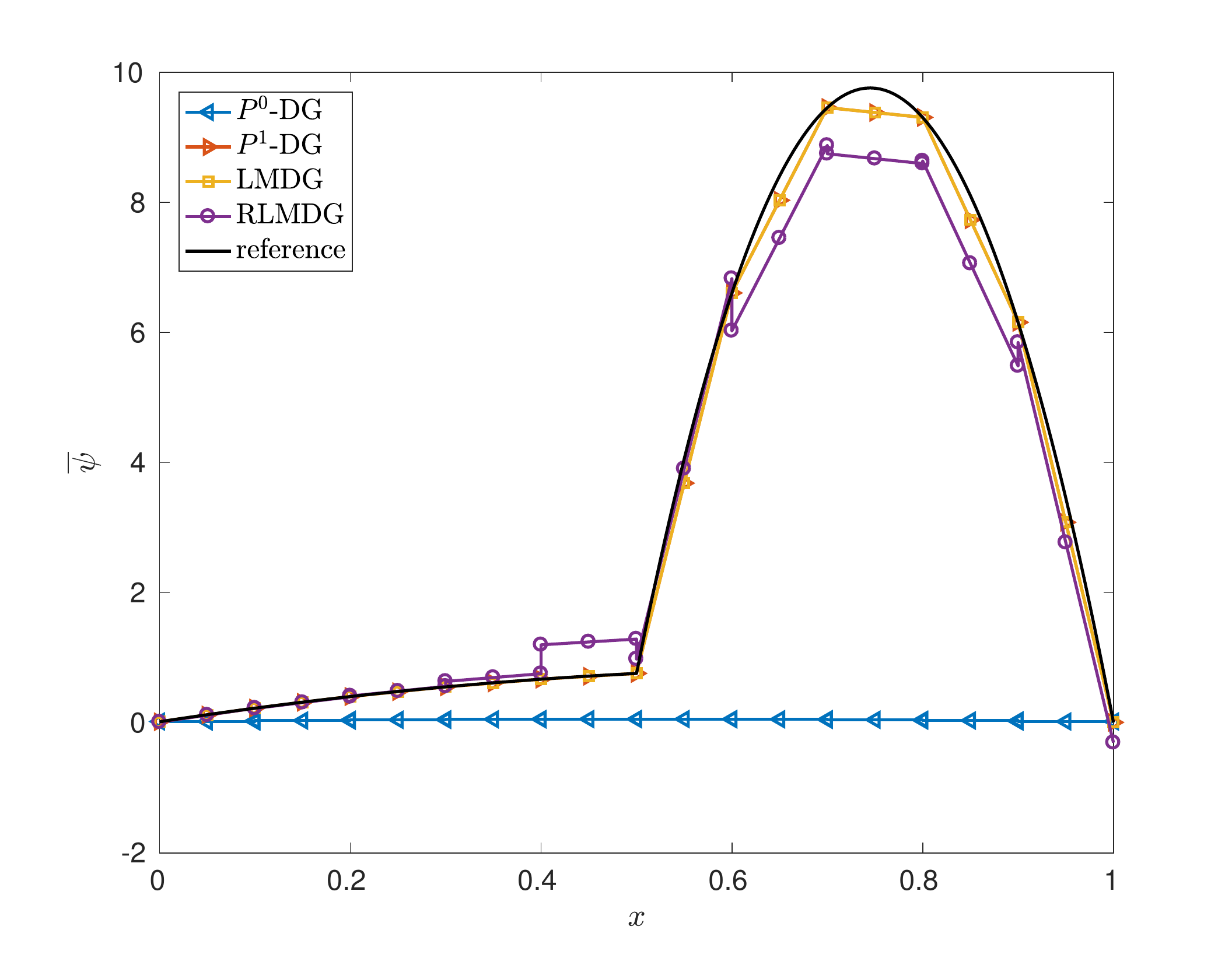}  
			\caption{$\sig{s,2} = 10000$.}
		\end{subfigure}
		\caption{Profiles of numerical scalar fluxes in 
			\Cref{examp-guermond}, $h = 0.1$.}\label{fig-1dguermond-1}
	\end{figure}
	\begin{figure}[h!]
		\begin{subfigure}[b]{0.3\textwidth}
			\includegraphics[width=\textwidth]{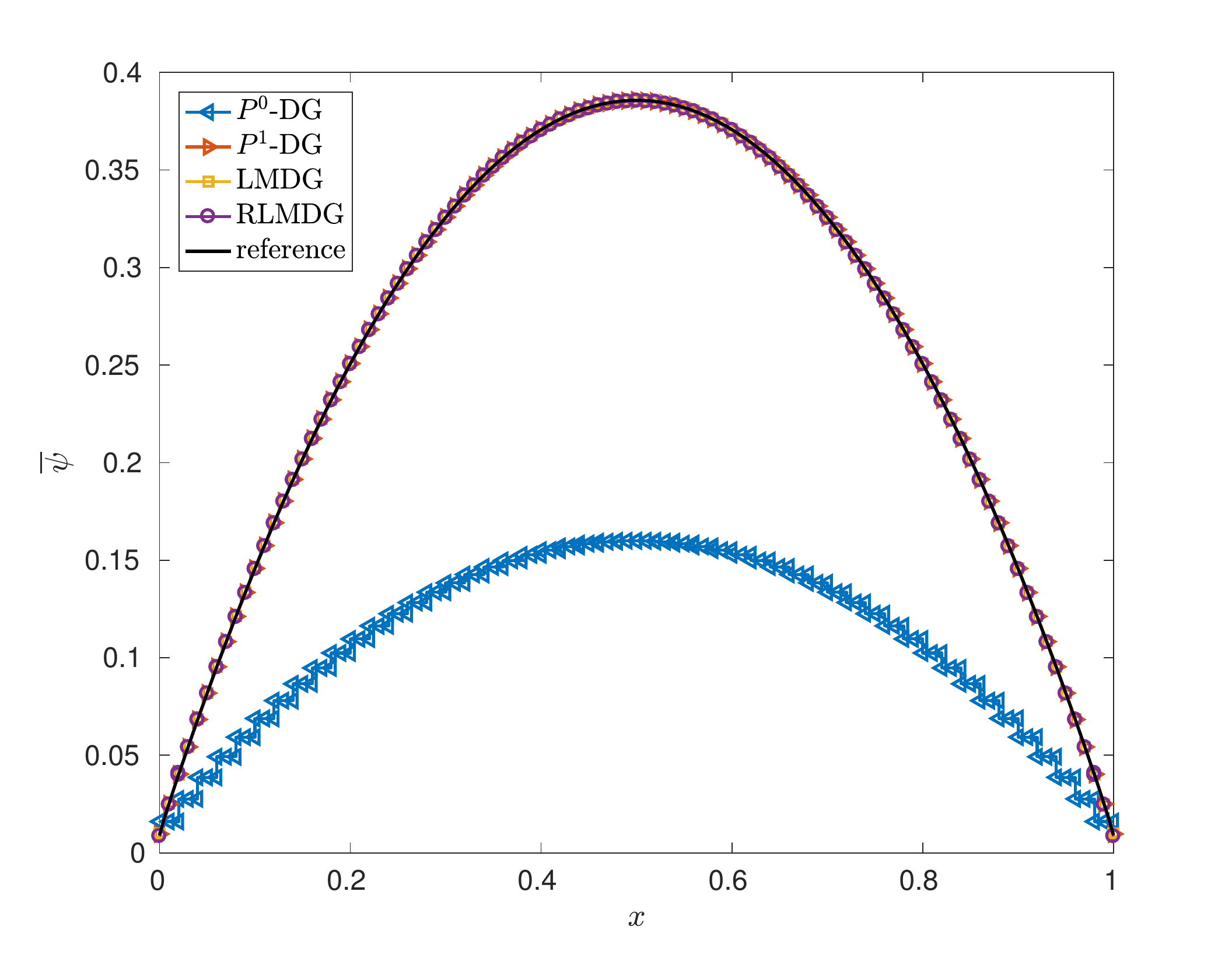}  
			\caption{$\sig{s,2} = 100$.}
		\end{subfigure}
		~
		\begin{subfigure}[b]{0.3\textwidth}
			\includegraphics[width=\textwidth]{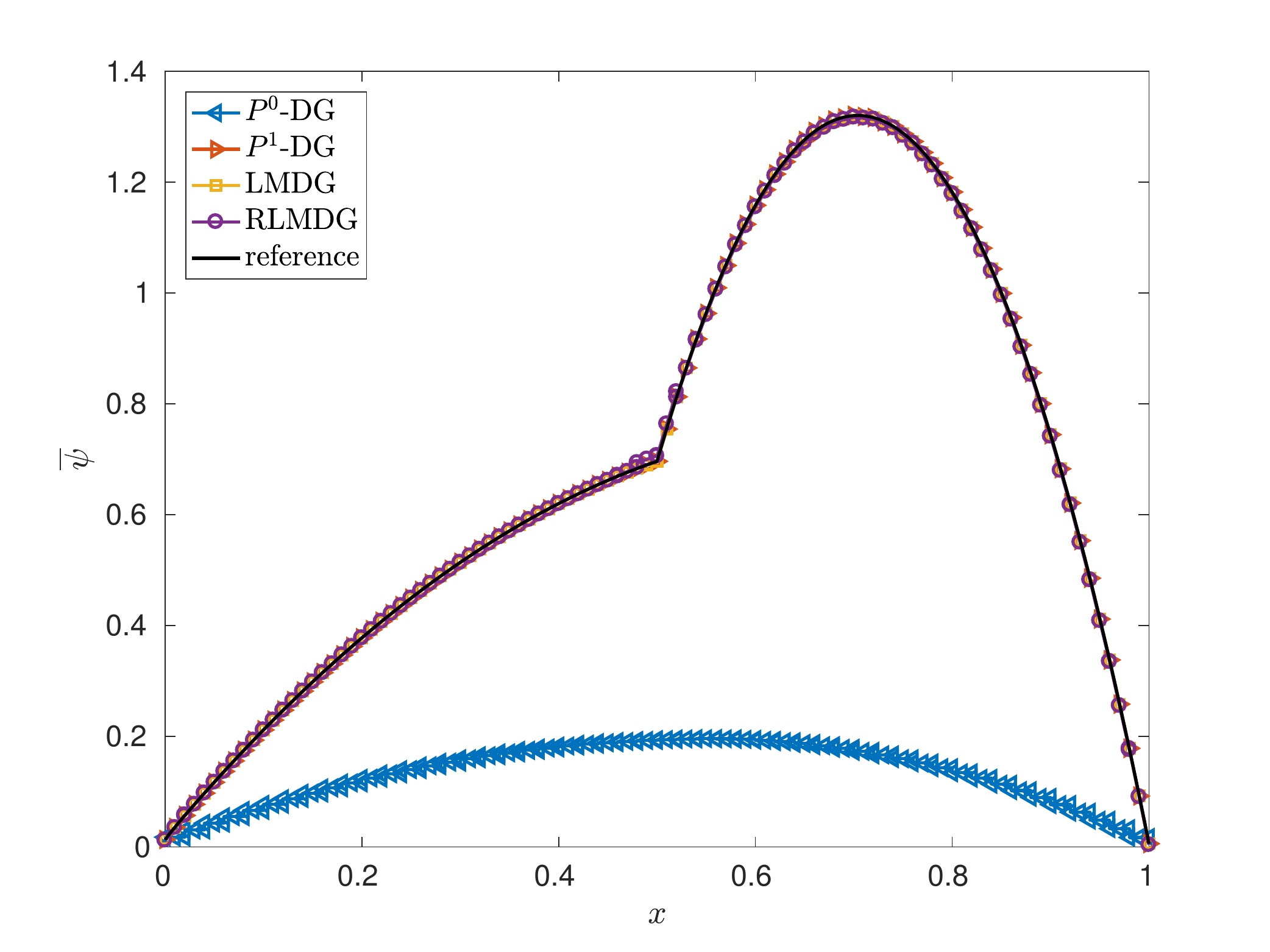}  
			\caption{$\sig{s,2} = 1000$.}
		\end{subfigure}
		~
		\begin{subfigure}[b]{0.3\textwidth}
			\includegraphics[width=\textwidth]{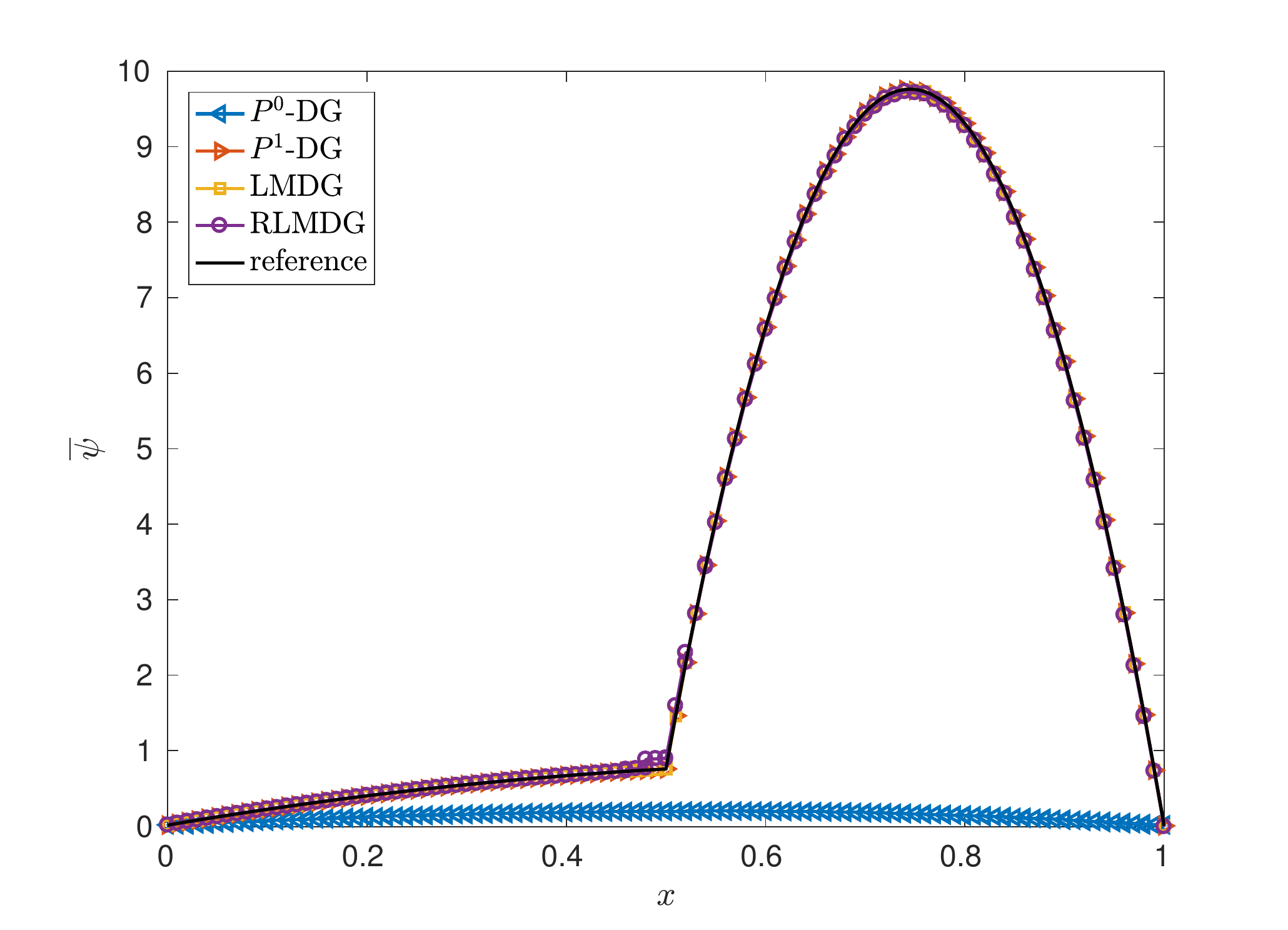}  
			\caption{$\sig{s,2} = 10000$.}
		\end{subfigure}
		\caption{Profiles of numerical scalar fluxes in 
			\Cref{examp-guermond}, $h = 0.02$.}\label{fig-1dguermond-2}
	\end{figure}
\end{example}
\begin{example}\label{examp-disccross1}
	In this numerical test, we solve a test problem from
	\cite{larsen1989asymptotic} with discontinuous
	cross-sections. We take $q = 0$ with the left inflow $\psi_l = 1$ 
	at $x_a = 0$ and $\psi_r = 0$ at $x_b = 11$. Let 
	$\frac{\sig{s}}{\veps} = \left\{\begin{matrix} 0,&0<x<1 \\100,&
	1<x<11\end{matrix}\right.$ and 
	$\veps{\sig{a}} = \left\{\begin{matrix} 2,&0<x<1 \\0,&
	1<x<11\end{matrix}\right.$.
	The $16$-point Gauss quadrature rule is used for angular discretization.
	The spatial mesh is set as
	$h = \left\{\begin{matrix}0.1,&0<x<1\\1,&1<x<11\\\end{matrix}\right.$.
	
	Profiles of the scalar flux obtained with various schemes are depicted in 
	\Cref{fig-disccross1-1}.  The reference solutions are obtained with the $P^1$-DG
	scheme on a refined mesh. The solution of the LMDG scheme is satisfactory. As before, the RLMDG scheme gives an
	accurate approximation to the scalar flux, except for kinks near the
	discontinuity.  However, this numerical artifact can also be alleviated by suppressing the reconstruction across the  discontinuity; see \Cref{fig-disccross1-2}. 
	\begin{figure}[h!]
		\centering
		\begin{subfigure}[h]{0.45\textwidth}
			\includegraphics[width=\textwidth]{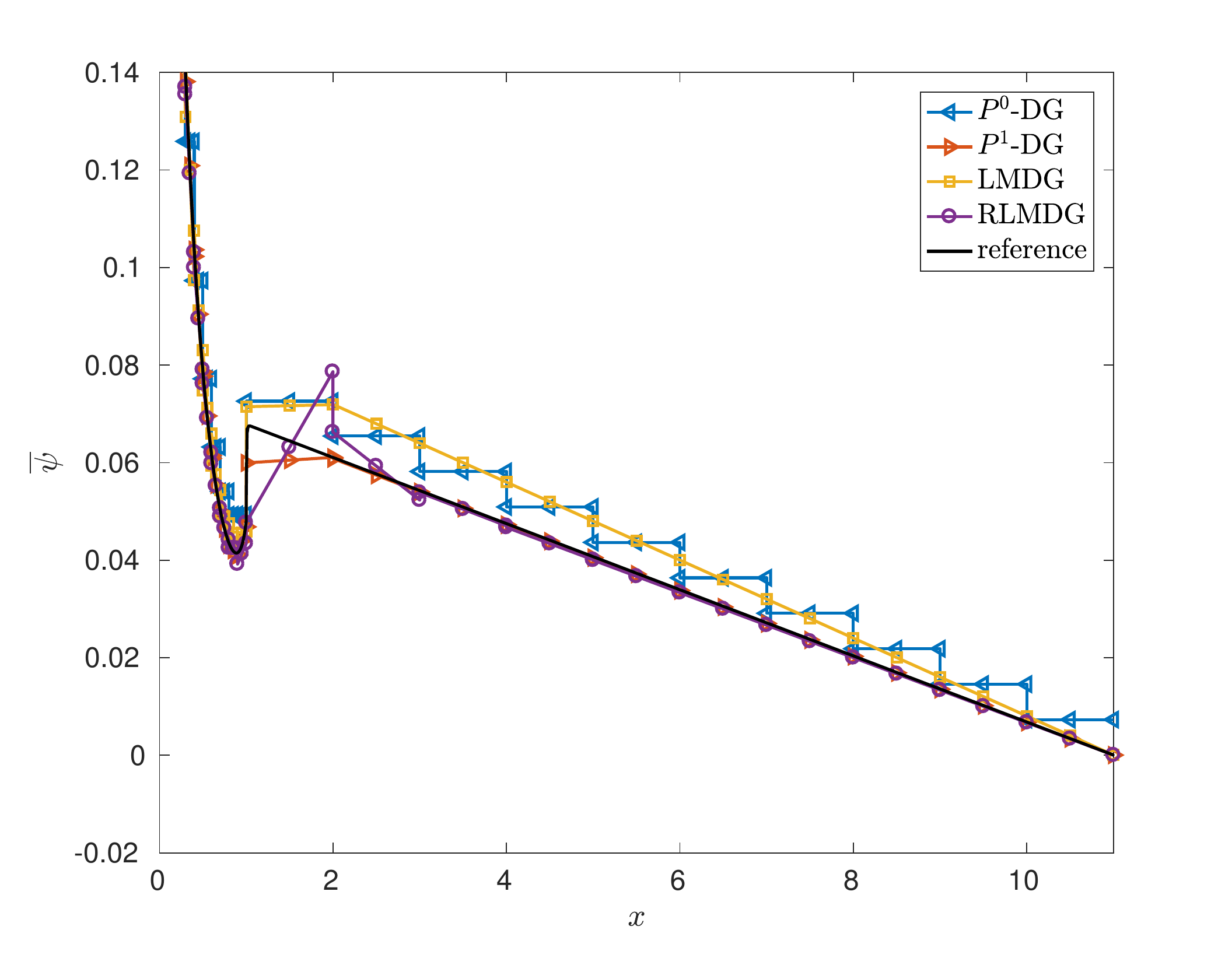}
			\caption{Numerical scalar fluxes.}\label{fig-disccross1-1}
		\end{subfigure}
		\begin{subfigure}[h]{0.45\textwidth}
			\includegraphics[width=\textwidth]{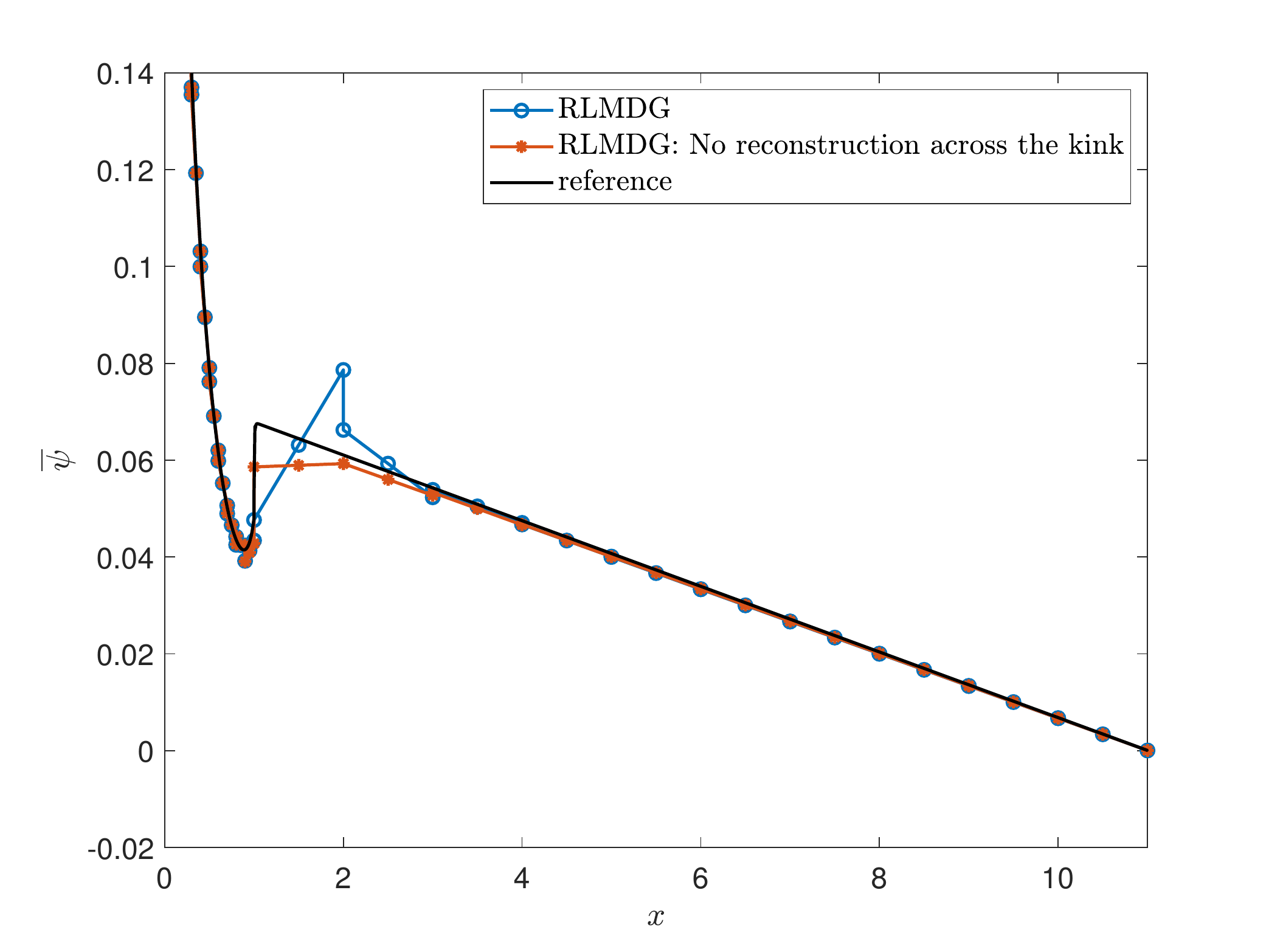}
			\caption{With suppressed reconstruction.}\label{fig-disccross1-2}
		\end{subfigure}
		\caption{Profiles of numerical scalar fluxes in 
			\Cref{examp-disccross1}.}
	\end{figure}
\end{example}
\begin{example}\label{examp-disccross2}
	This test is also from \cite{larsen1989asymptotic}, with $D = [0,20]$ and $\psi_l = \psi_r =0$. The cross-sections are
	$\frac{\sig{s}}{\veps} = \left\{\begin{matrix} 90,&0<x<10 \\100,&
	10<x<20\end{matrix}\right.$
	and 
	$\veps{\sig{a}} = \left\{\begin{matrix} 10,&0<x<10 \\0,&
	10<x<20\end{matrix}\right.$.
	We solve the problem using the ${16}$-point Gauss quadrature rule and 
	the spatial mesh is uniform with $h=1$.
	For this numerical test, the system has smaller changes among different directions. 
	Both the LMDG and RLMDG schemes give accurate approximations.
	Solution profiles are give in \Cref{fig-disccross2}.
	\begin{figure}[h!]
		\centering
		\includegraphics[width=0.45\textwidth]{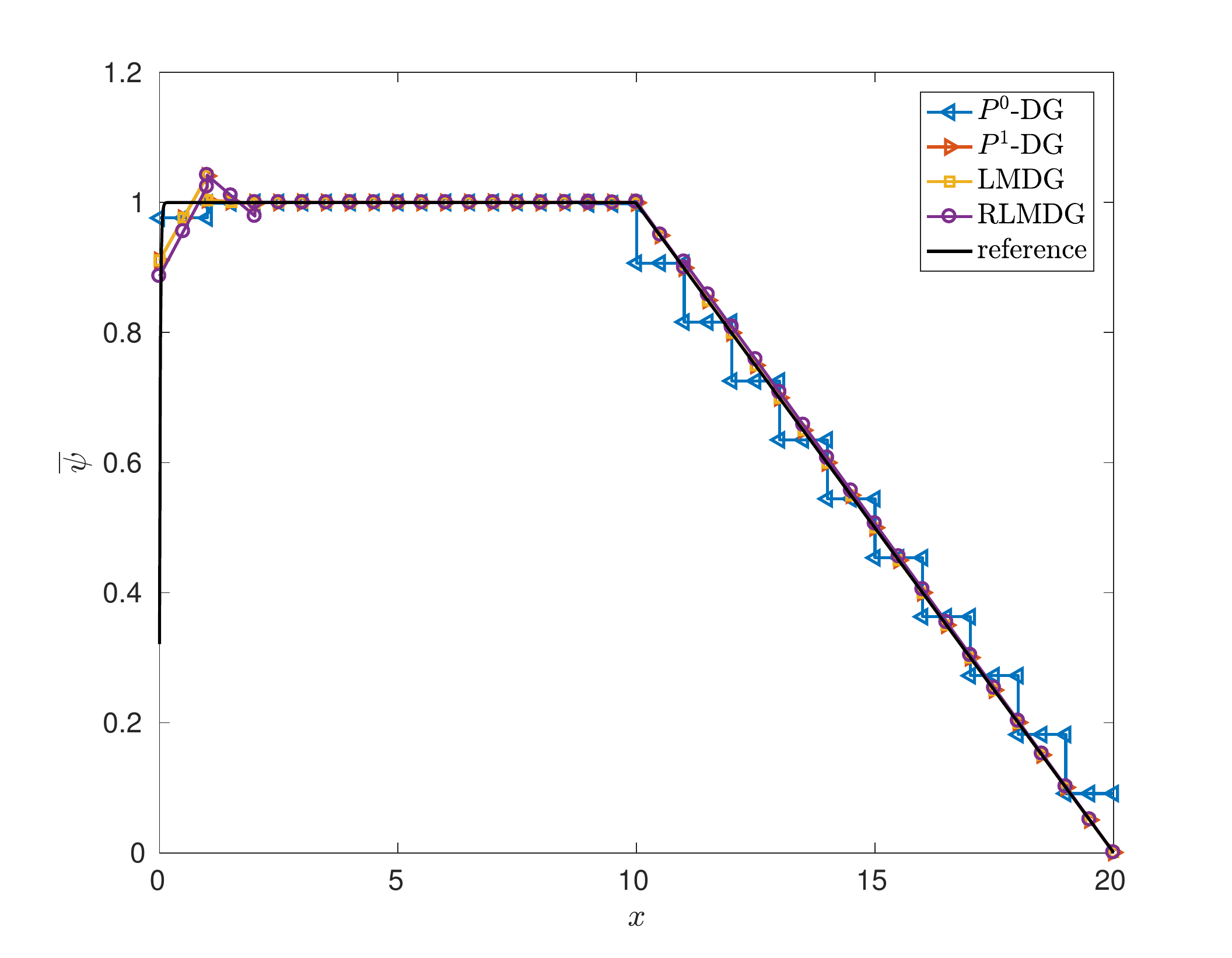}
		\caption{Profiles of numerical scalar fluxes in 
			\Cref{examp-disccross2}.}\label{fig-disccross2}
	\end{figure}
\end{example}
\subsection{Two dimensional tests}
We consider two dimensional problems on Cartesian meshes in this section.
\begin{example}\label{examp-2daccu}
	We set $\veps = 1$ and $\sig{s} = \sig{a} =1$ and test the accuracy with exact solutions $\psi = \sin(x+y)$ and
	$\psi = (\Omega_x-3\Omega_y)^2\sin(2x+y)$.
	As can be seen from \Cref{tab-2d-aniso}, 
	for $\psi  = \sin(x+y)$, both LMDG and RLMDG
	schemes are second-order accurate. While for the anisotropic problem with $\psi = (\Omega_x-3\Omega_y)^2\sin(2x+y)$, 
	the RLMDG scheme is still second-order accurate and the LMDG
	scheme is first-order accurate.
	\begin{table}[h!]
		\centering
		\footnotesize
		\hskip-1.0cm
		\begin{tabular}{c|c|c|c|c|c|c|c|c|c|c}
			\hline
			\multicolumn{11}{c}{$\psi = \sin(x+y)$}\\
			\hline
			&\multicolumn{2}{c|}{$P^0$-DG}&\multicolumn{2}{c|}{$P^1$-DG}&\multicolumn{2}{c|}{$Q^1$-DG}&\multicolumn{2}{c|}{LMDG}&\multicolumn{2}{c}{RLMDG}\\
			\hline
			$h/\sqrt{2}$&error&order&error&order&error&order&error&order&error&order\\
			\hline
			$1/20$& 2.04e-2&-   &1.45e-4&-   &1.40e-4&-   &1.24e-4&-   &4.59e-4&-\\
			$1/40$& 1.10e-2&0.89&3.42e-5&2.08&3.53e-5&1.98&3.12e-5&1.99&1.18e-4&1.96\\
			$1/80$& 5.77e-3&0.94&8.28e-6&2.04&8.88e-6&1.99&7.82e-6&2.00&2.98e-5&1.98\\
			$1/160$&2.96e-3&0.96&2.04e-6&2.02&2.26e-6&2.00&1.96e-6&2.00&7.51e-6&1.99\\
			\hline
			\multicolumn{11}{c}{$\psi = (\Omega_x - 2\Omega_y)^2 \sin(2x + y)$}\\
			\hline
			&\multicolumn{2}{c|}{$P^0$-DG}&\multicolumn{2}{c|}{$P^1$-DG}&\multicolumn{2}{c|}{$Q^1$-DG}
			&\multicolumn{2}{c|}{LMDG}&\multicolumn{2}{c}{RLMDG}\\
			\hline
			$h/\sqrt{2}$&error&order&error&order&error&order&error&order&error&order\\
			\hline
			$1/20$&  7.84e-2	&     & 1.64e-3&		 -& 1.39e-3&	-& 5.04e-2&		 -& 4.81e-3&   -\\   
			$1/40$&  4.18e-2	&0.91 & 4.12e-4&	2.00&	3.53e-4&	1.98&	2.57e-2&	0.97&	1.21e-3&1.99\\
			$1/80$&  2.12e-2	&0.98 & 1.01e-4&	2.03&	8.87e-5&	1.99&	1.30e-2&	0.99&	3.05e-4&1.99\\
			$1/160$& 1.07e-2	&0.97 & 2.52e-5&	2.00&	2.22e-5&	2.00&	6.51e-3&	0.99&	7.63e-5&2.00\\
			\hline
		\end{tabular}
		\caption{2D accuracy test with fabricated solutions.}\label{tab-2d-aniso}
	\end{table}
	
\end{example}
\begin{example}\label{examp-2ddifflim}
	To examine the asymptotic preserving property, we consider the problem defined on $[-1,1]\times[-1,1]$ with 
	zero inflow boundary conditions. Let $\sig{s} =
	\sig{a} = 1$. We assume $q =
	(\frac{\pi^2}{6}+1)\cos(\frac{\pi}{2}x)\cos(\frac{\pi}{2}y)$. The
	asymptotic solution is $\psi^{(0)}=
	\cos(\frac{\pi}{2}x)\cos(\frac{\pi}{2}y)$. We
	test with $\veps= 1,2^{-6},2^{-10},2^{-14}$; the numerical results are given in \Cref{fig-2d}. For the $P^0$-DG and
	$P^1$-DG schemes, solutions become zero near the
	diffusion limit, while for the $Q^1$-DG scheme, LMDG scheme and RLMDG scheme, the correct asymptotic profile is maintained.
	\begin{figure}[h!]
		\centering
		\begin{subfigure}[b]{0.23\textwidth}
			\includegraphics[width=\textwidth]{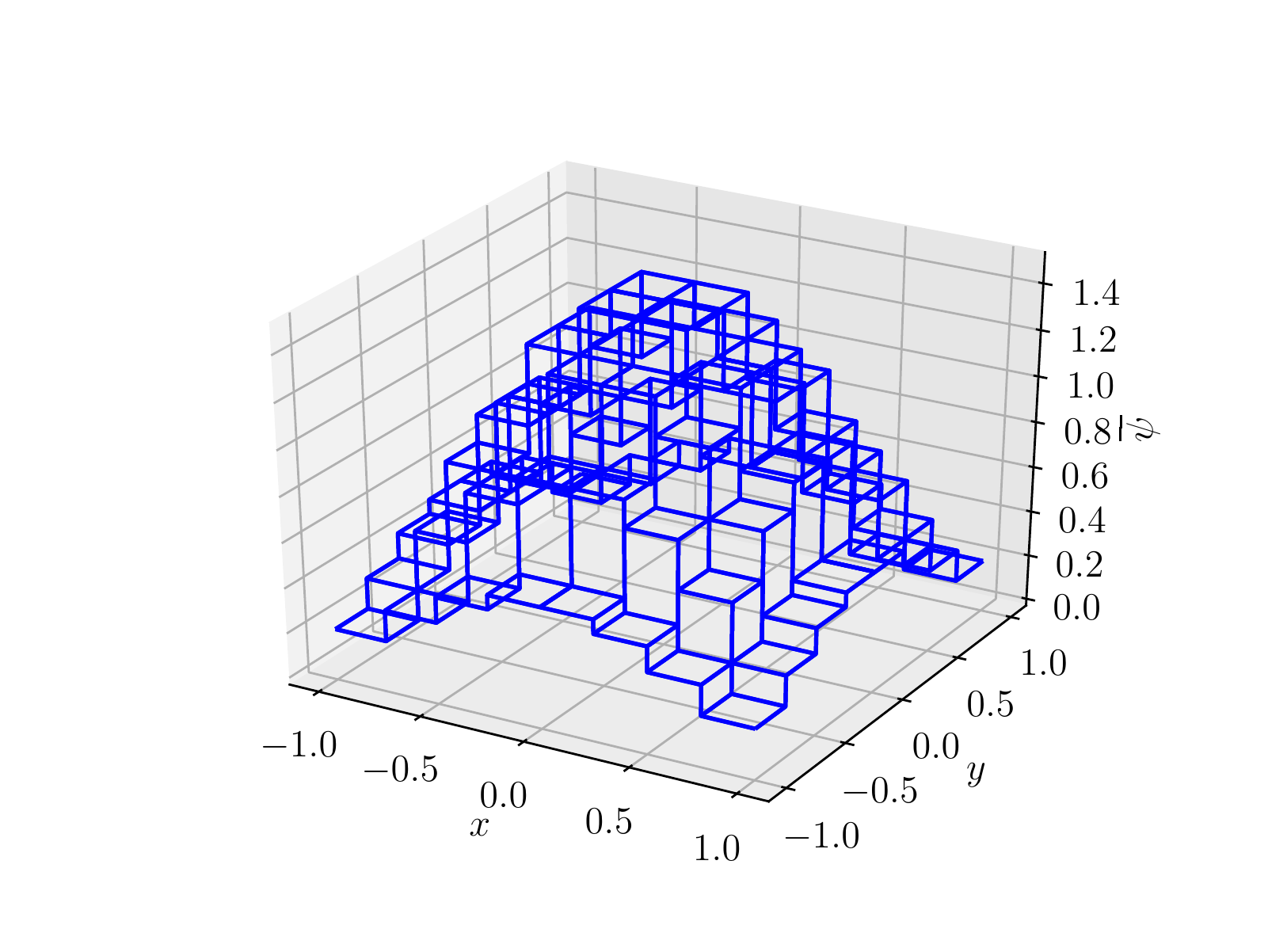}  
			\caption{$P^0$, $\veps = 1$.}
		\end{subfigure}
		~
		\begin{subfigure}[b]{0.23\textwidth}
			\includegraphics[width=\textwidth]{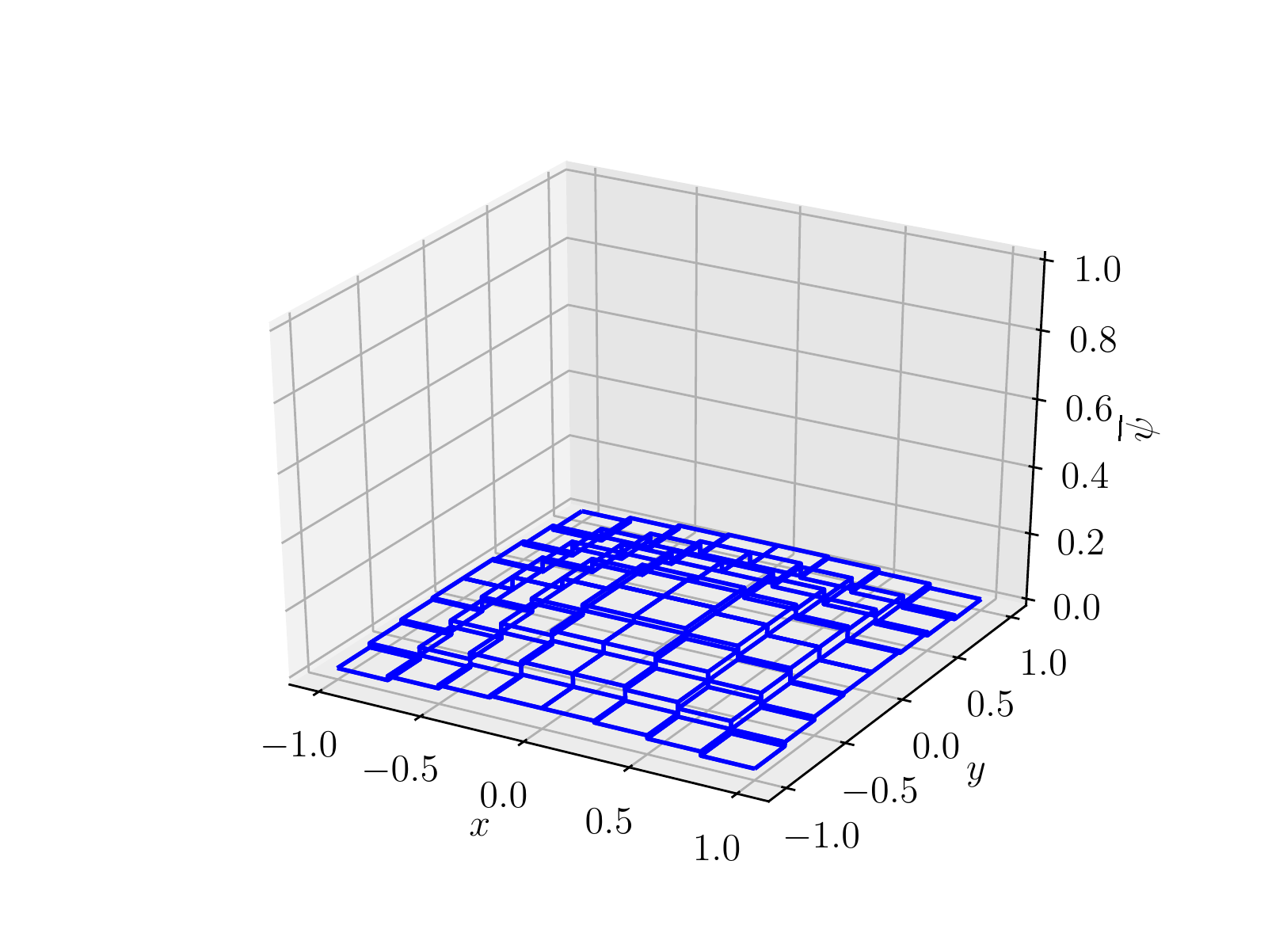}  
			\caption{$P^0$, $\veps = 2^{-6}$.}
		\end{subfigure}
		~
		\begin{subfigure}[b]{0.23\textwidth}
			\includegraphics[width=\textwidth]{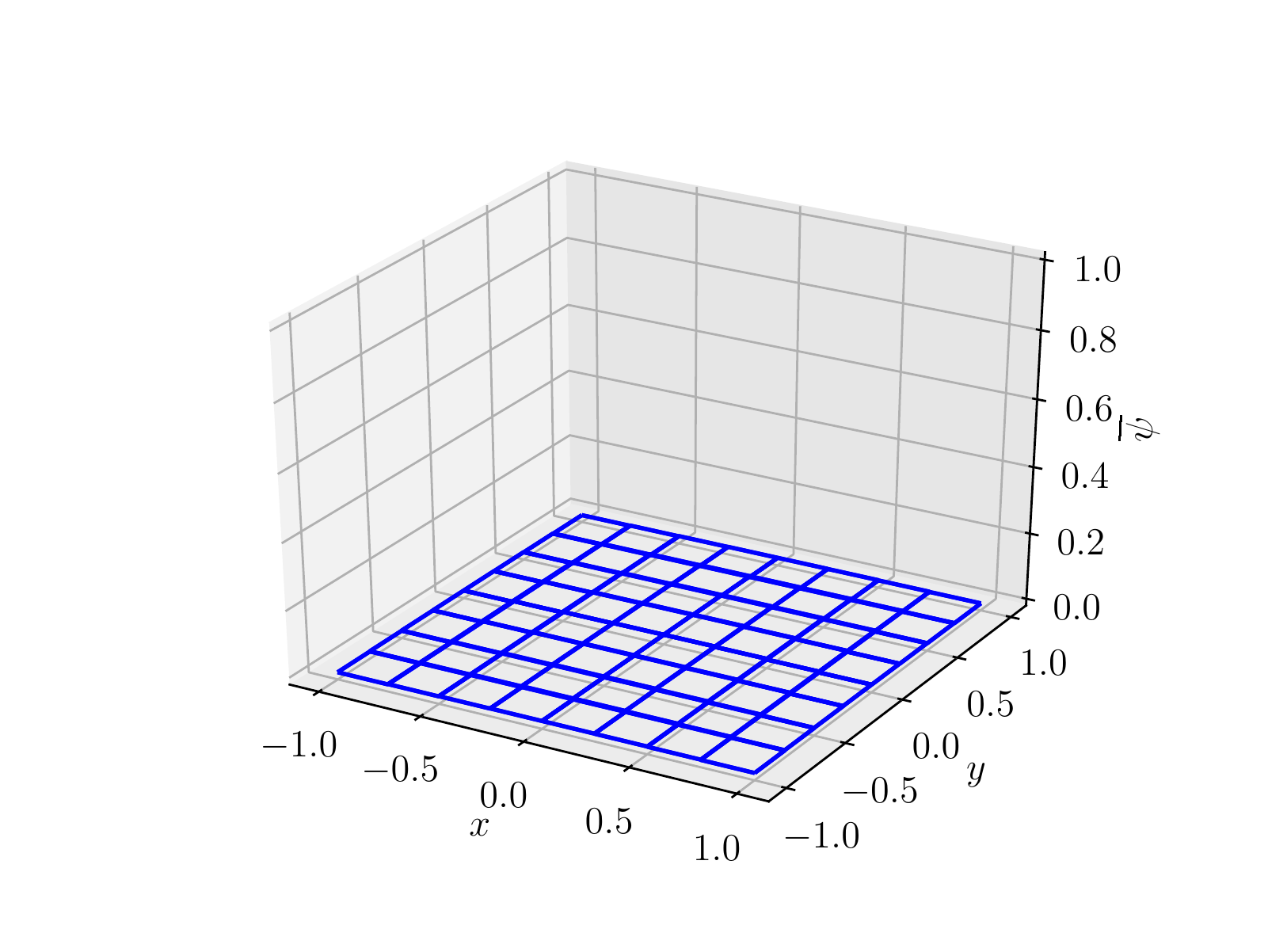}  
			\caption{$P^0$, $\veps = 2^{-10}$.}
		\end{subfigure}
		~
		\begin{subfigure}[b]{0.23\textwidth}
			\includegraphics[width=\textwidth]{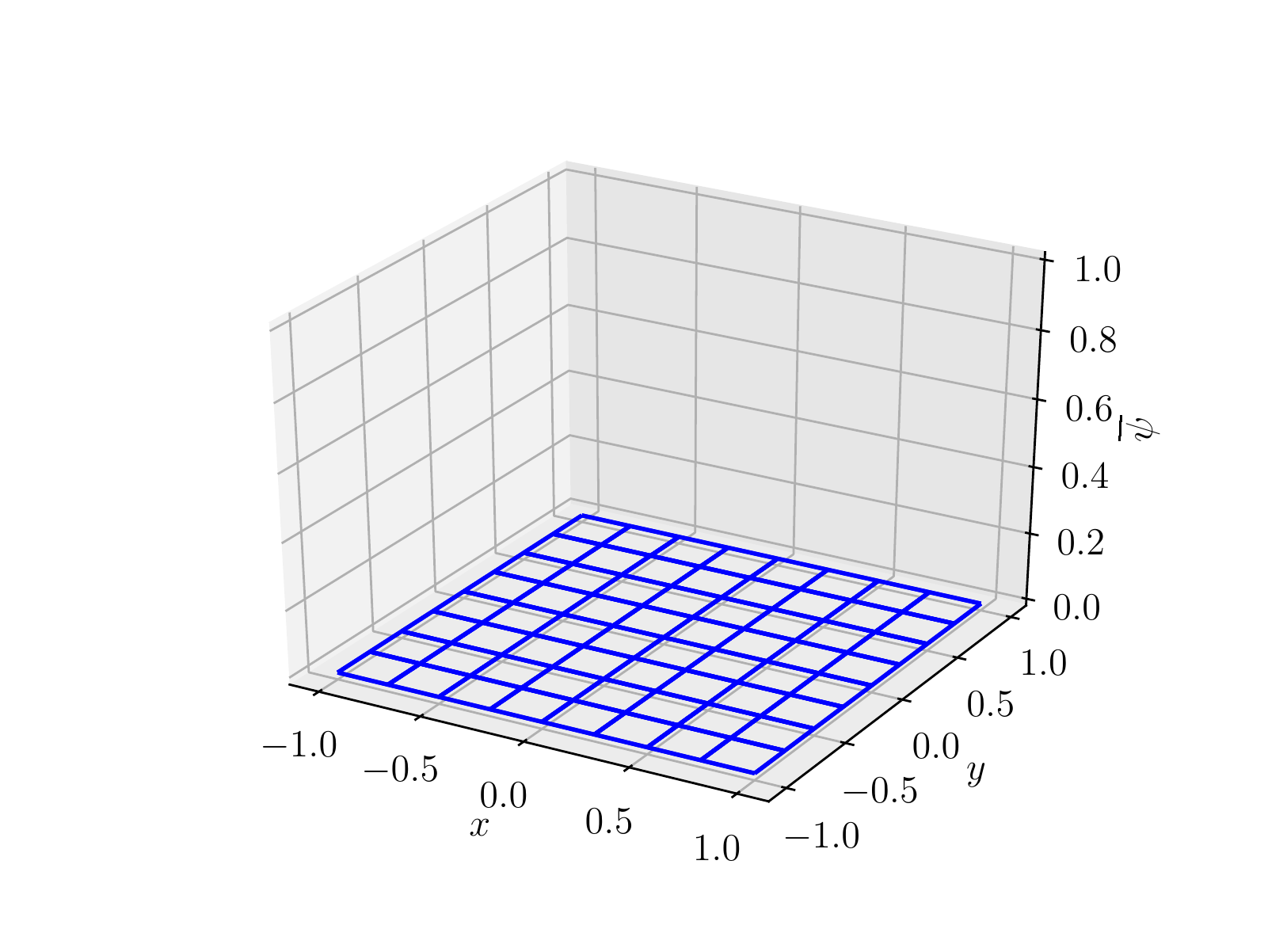}  
			\caption{$P^0$, $\veps = 2^{-14}$.}
		\end{subfigure}\\
		\begin{subfigure}[b]{0.23\textwidth}
			\includegraphics[width=\textwidth]{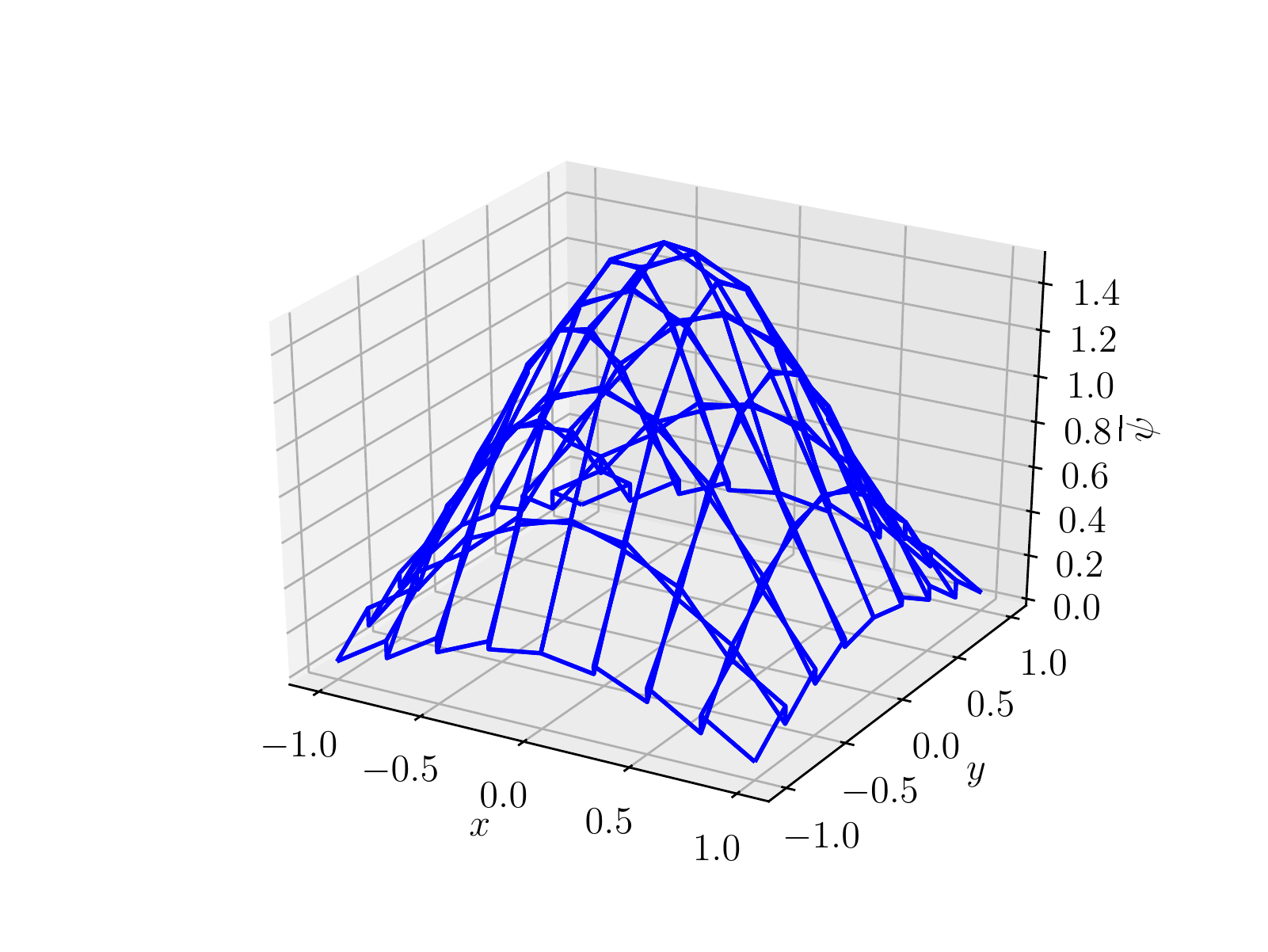}  
			\caption{$P^1$, $\veps = 1$.}
		\end{subfigure}
		~
		\begin{subfigure}[b]{0.23\textwidth}
			\includegraphics[width=\textwidth]{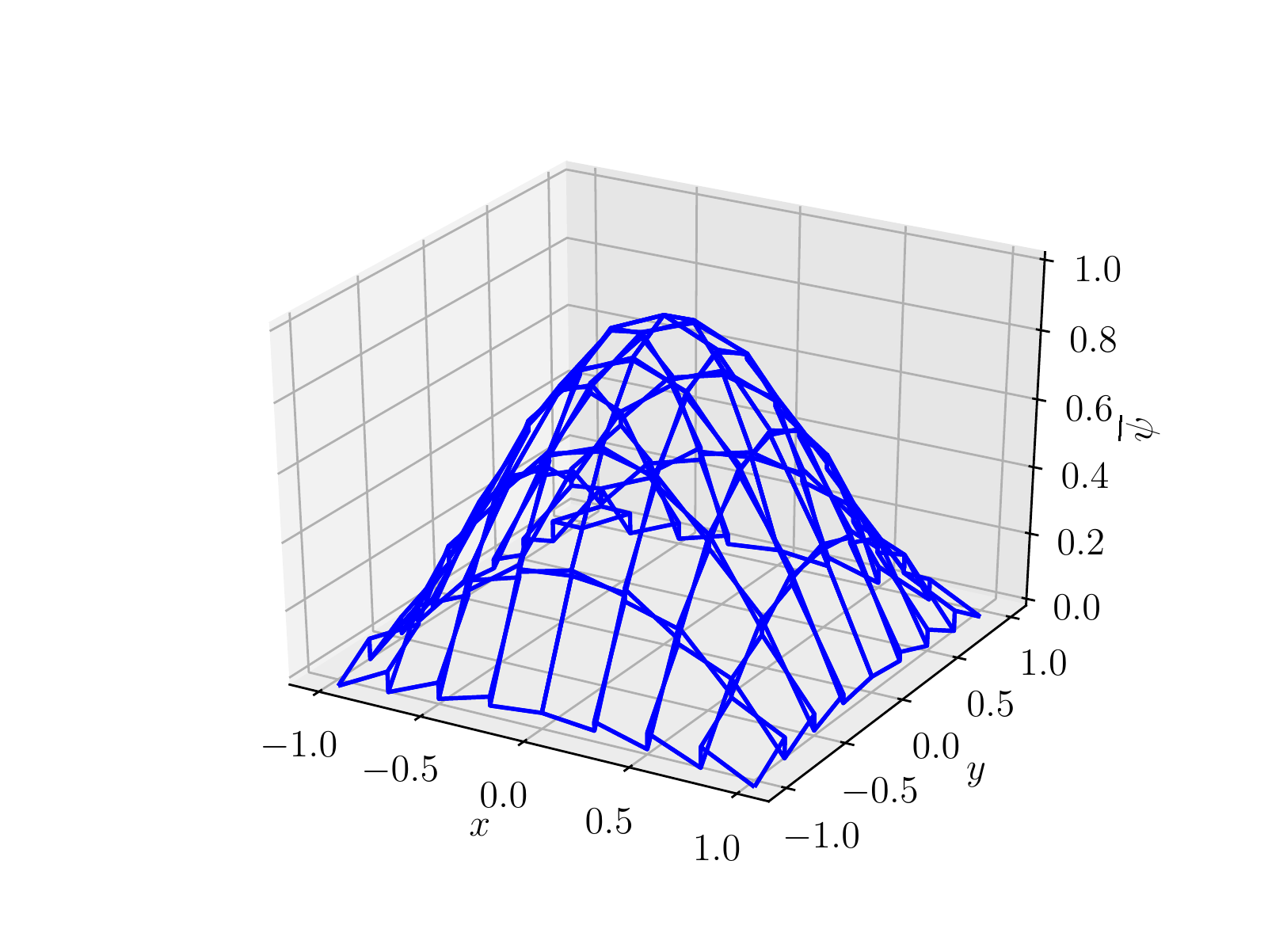}  
			\caption{$P^1$, $\veps = 2^{-6}$.}
		\end{subfigure}
		~
		\begin{subfigure}[b]{0.23\textwidth}
			\includegraphics[width=\textwidth]{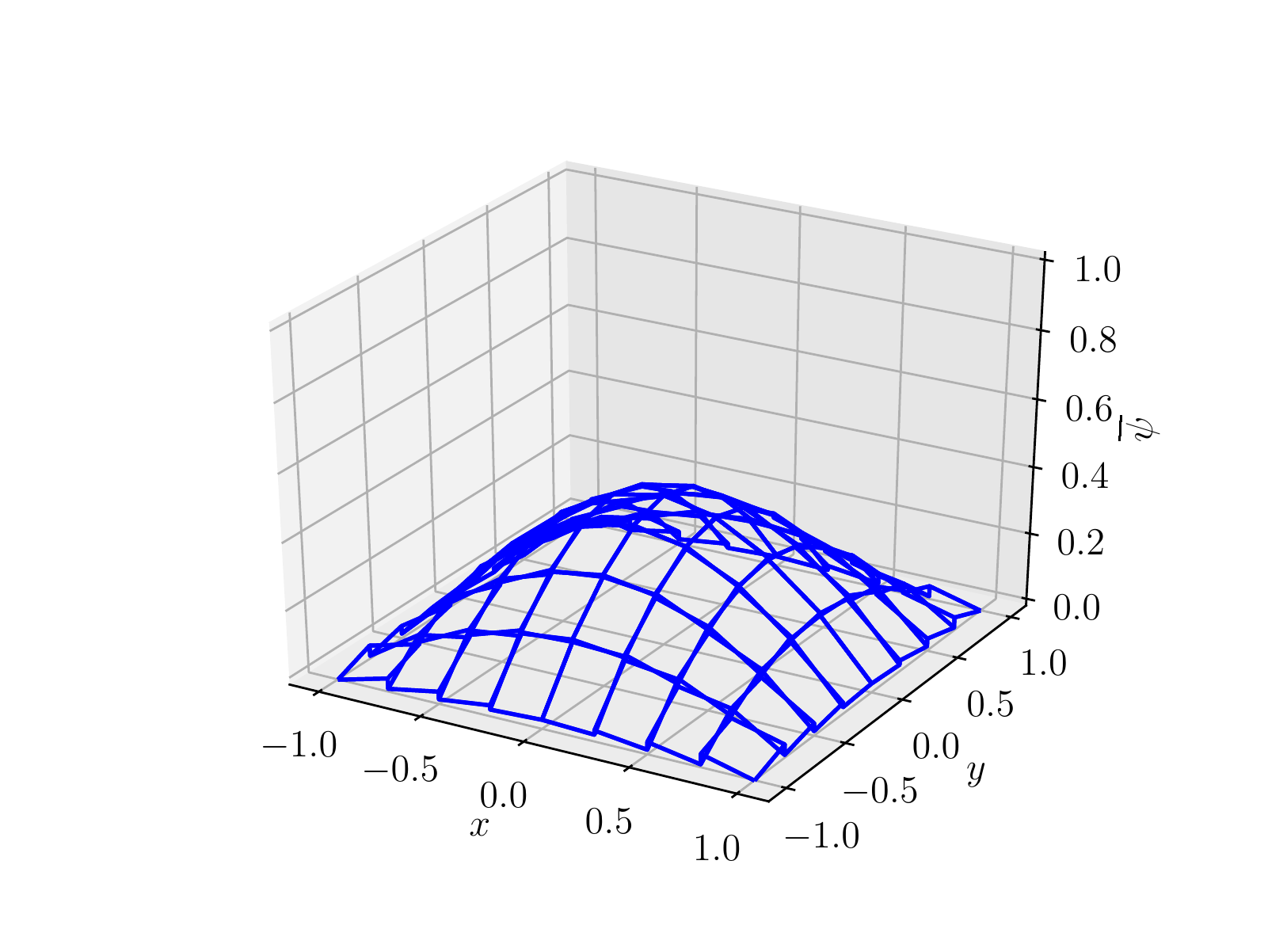}  
			\caption{$P^1$, $\veps = 2^{-10}$.}
		\end{subfigure}
		~
		\begin{subfigure}[b]{0.23\textwidth}
			\includegraphics[width=\textwidth]{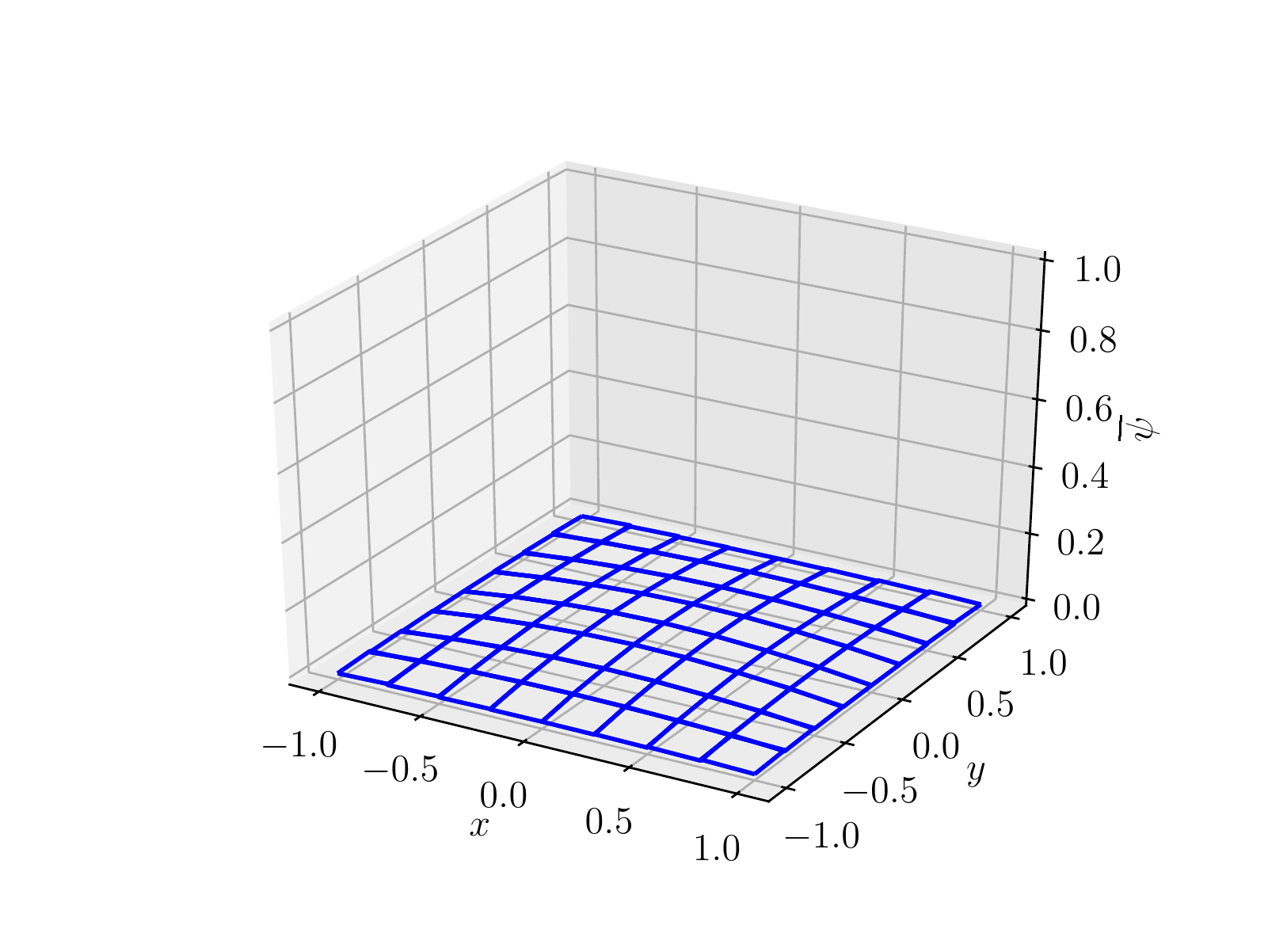}  
			\caption{$P^1$, $\veps = 2^{-14}$.}
		\end{subfigure}\\
		\begin{subfigure}[b]{0.23\textwidth}
			\includegraphics[width=\textwidth]{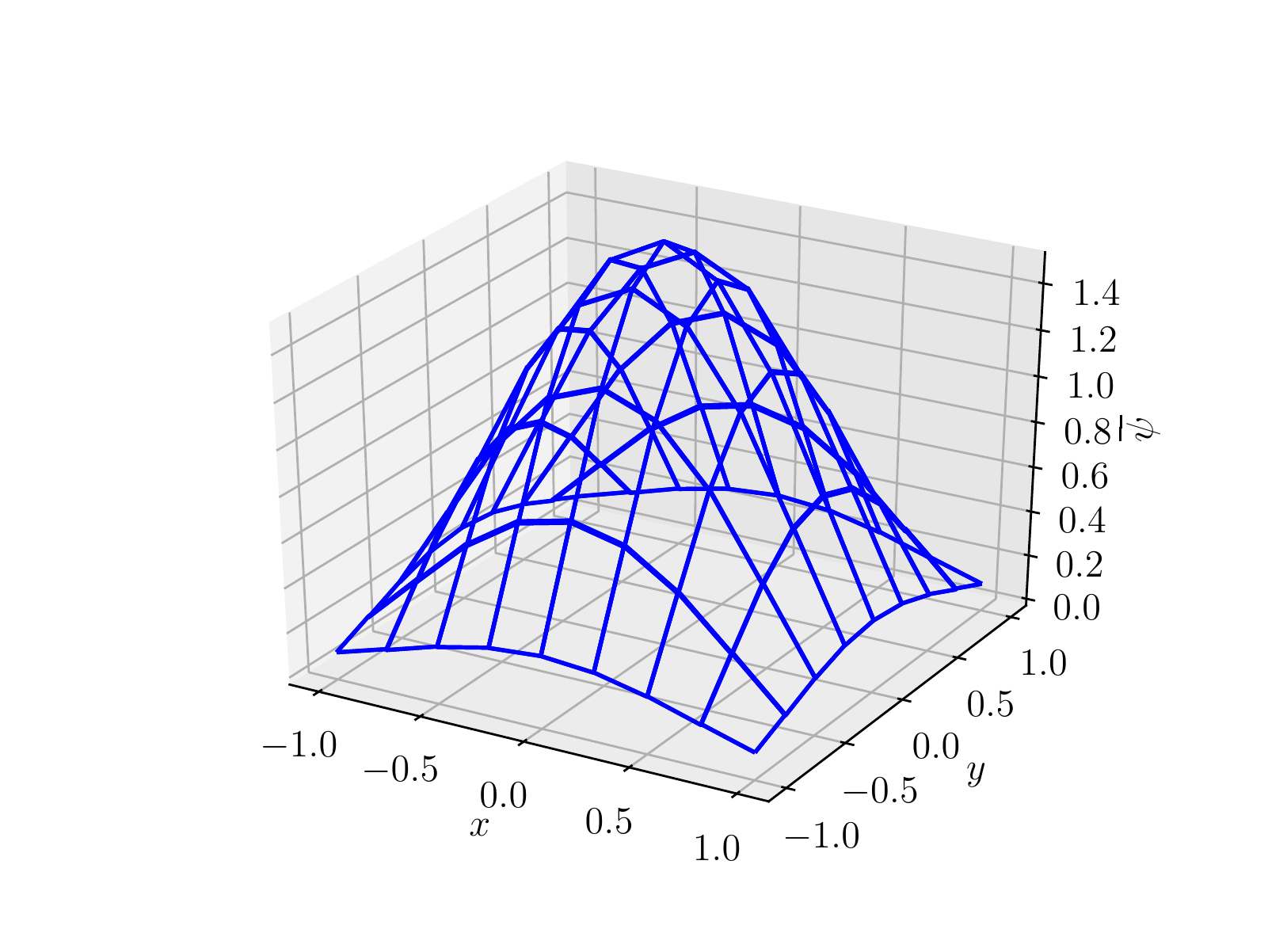}  
			\caption{$Q^1$, $\veps = 1$.}
		\end{subfigure}
		~
		\begin{subfigure}[b]{0.23\textwidth}
			\includegraphics[width=\textwidth]{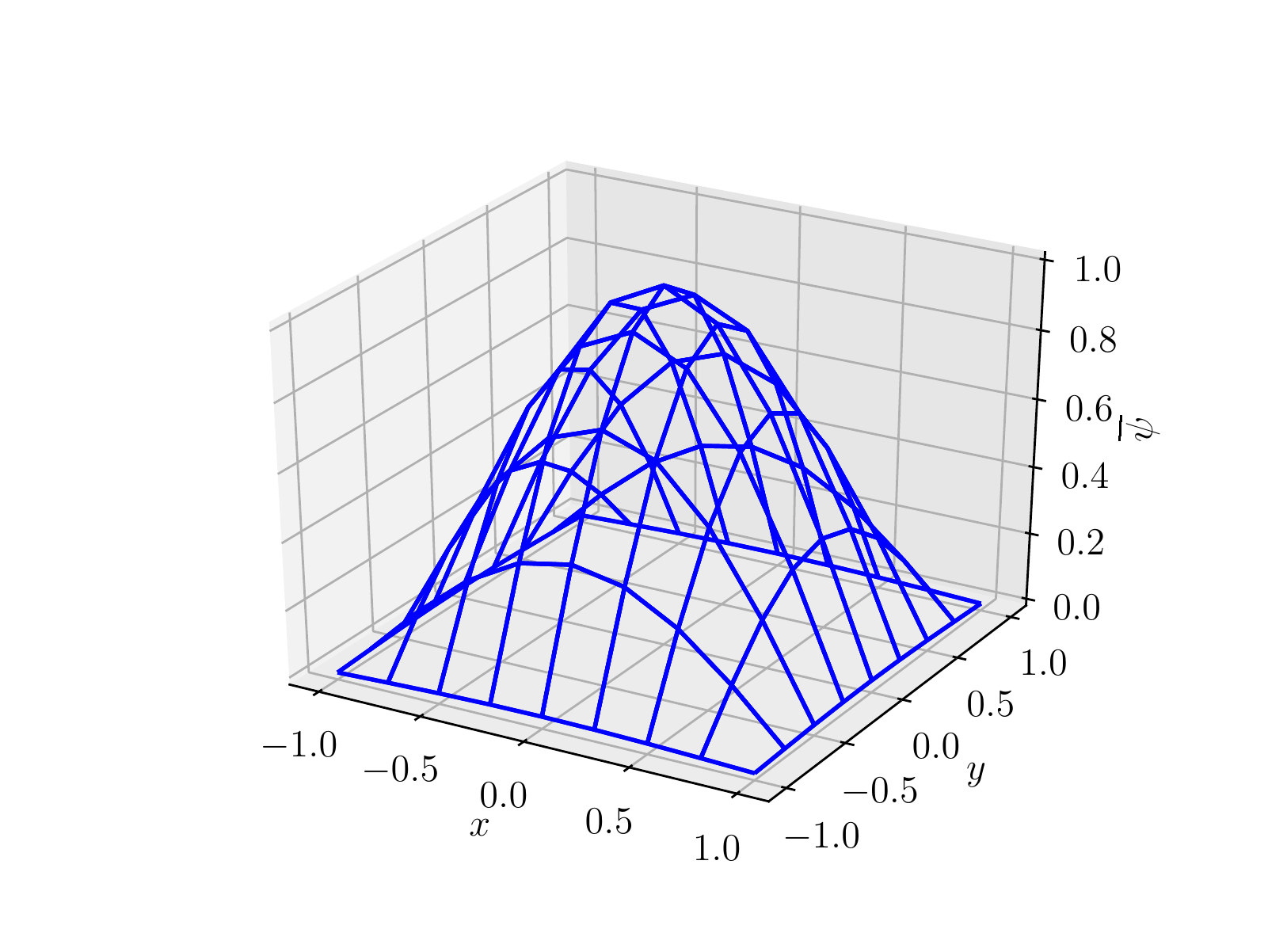}  
			\caption{$Q^1$, $\veps = 2^{-6}$.}
		\end{subfigure}
		~
		\begin{subfigure}[b]{0.23\textwidth}
			\includegraphics[width=\textwidth]{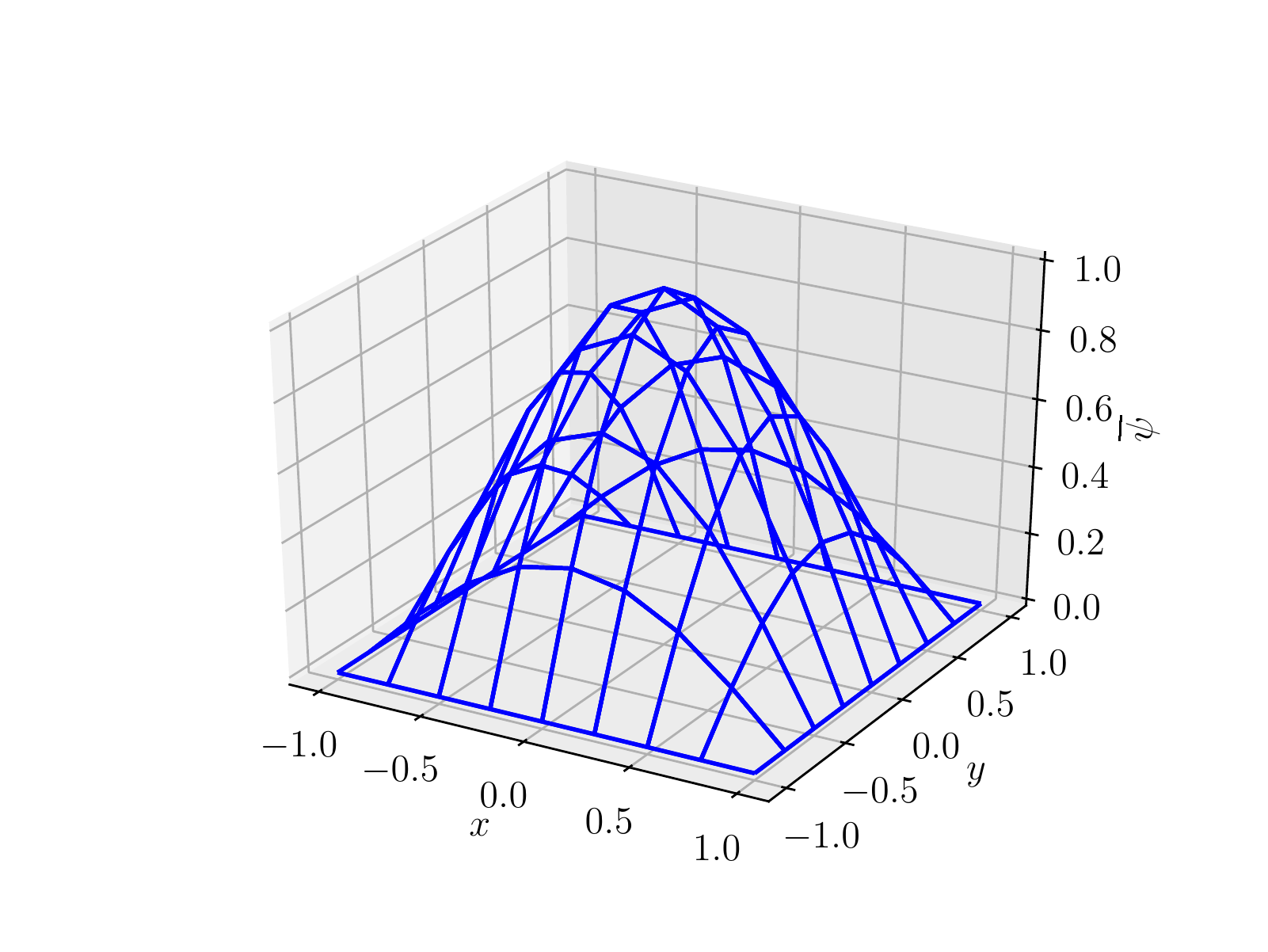}  
			\caption{$Q^1$, $\veps = 2^{-10}$.}
		\end{subfigure}
		~
		\begin{subfigure}[b]{0.23\textwidth}
			\includegraphics[width=\textwidth]{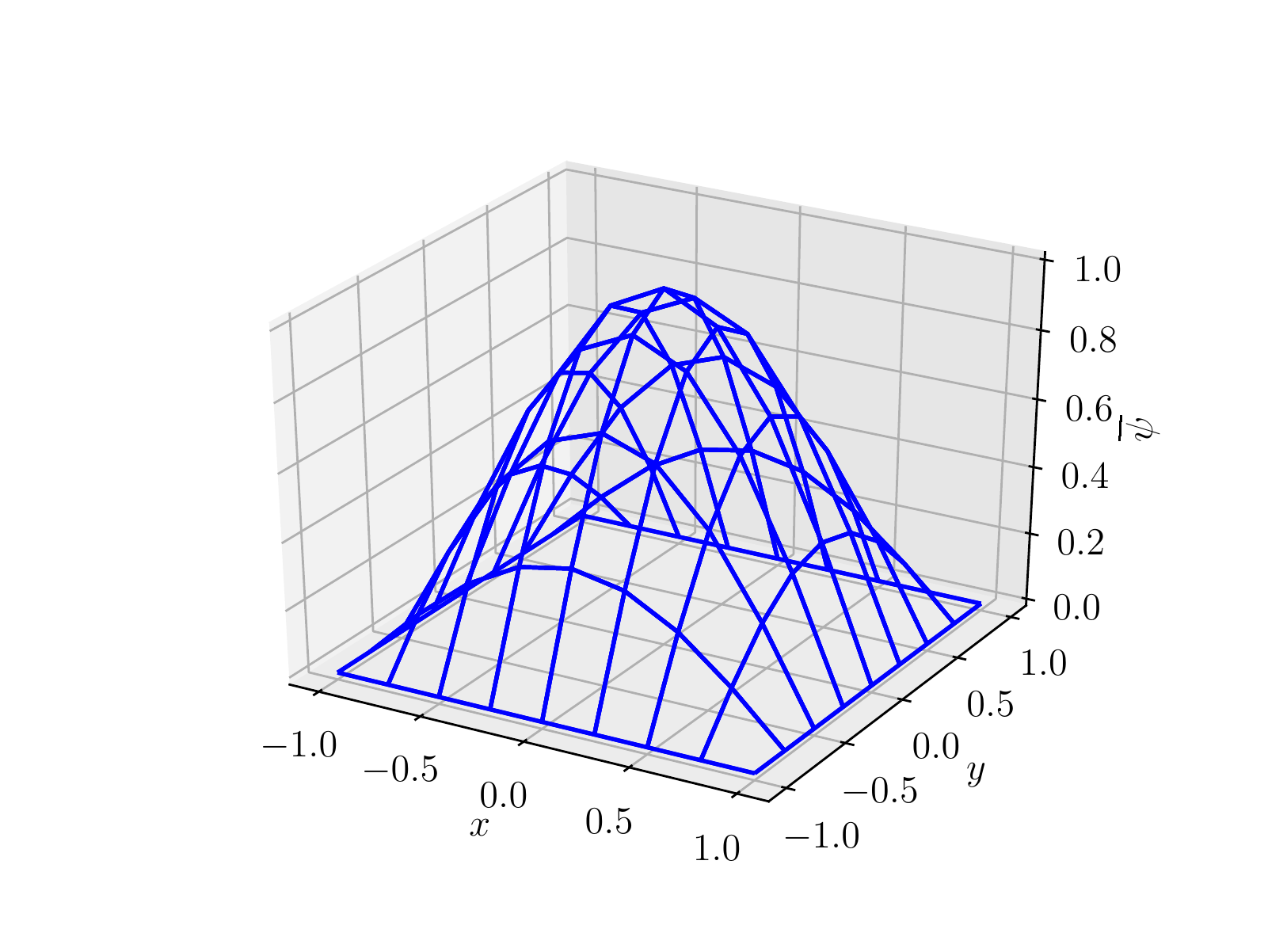}  
			\caption{$Q^1$, $\veps = 2^{-14}$.}
		\end{subfigure}\\
		\begin{subfigure}[b]{0.23\textwidth}
			\includegraphics[width=\textwidth]{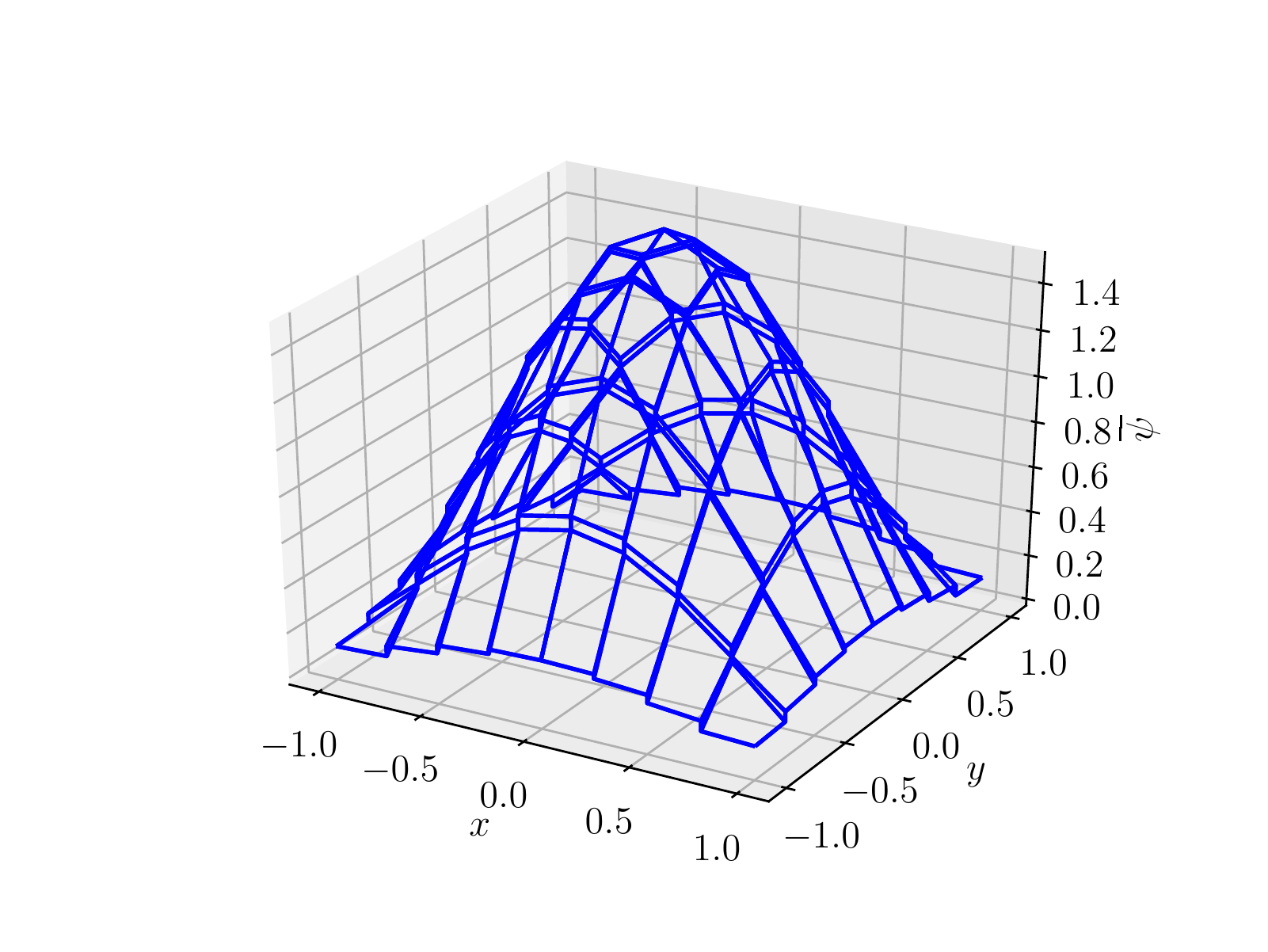}  
			\caption{LM, $\veps = 1$.}
		\end{subfigure}
		~
		\begin{subfigure}[b]{0.23\textwidth}
			\includegraphics[width=\textwidth]{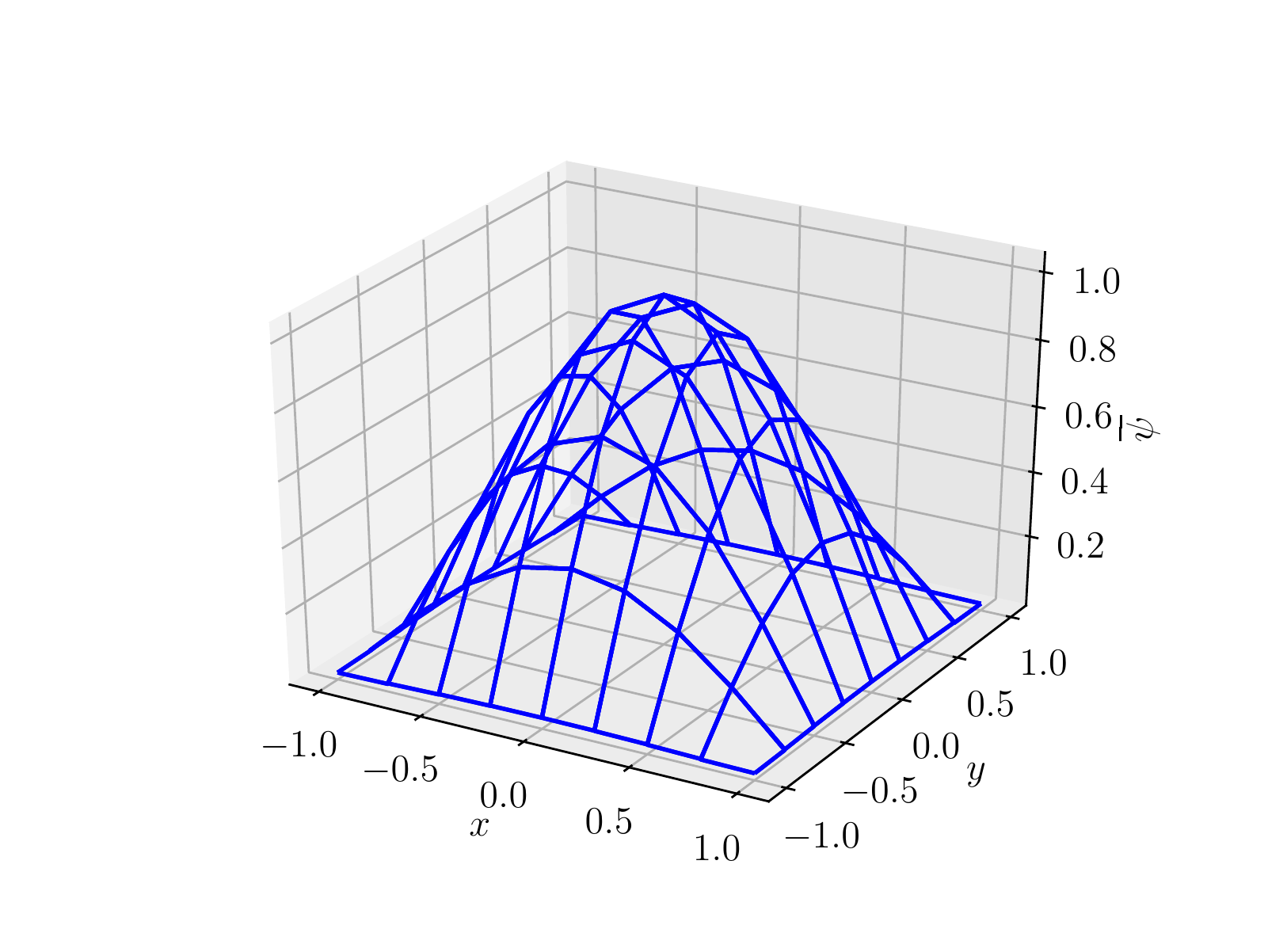}  
			\caption{LM, $\veps = 2^{-6}$.}
		\end{subfigure}
		~
		\begin{subfigure}[b]{0.23\textwidth}
			\includegraphics[width=\textwidth]{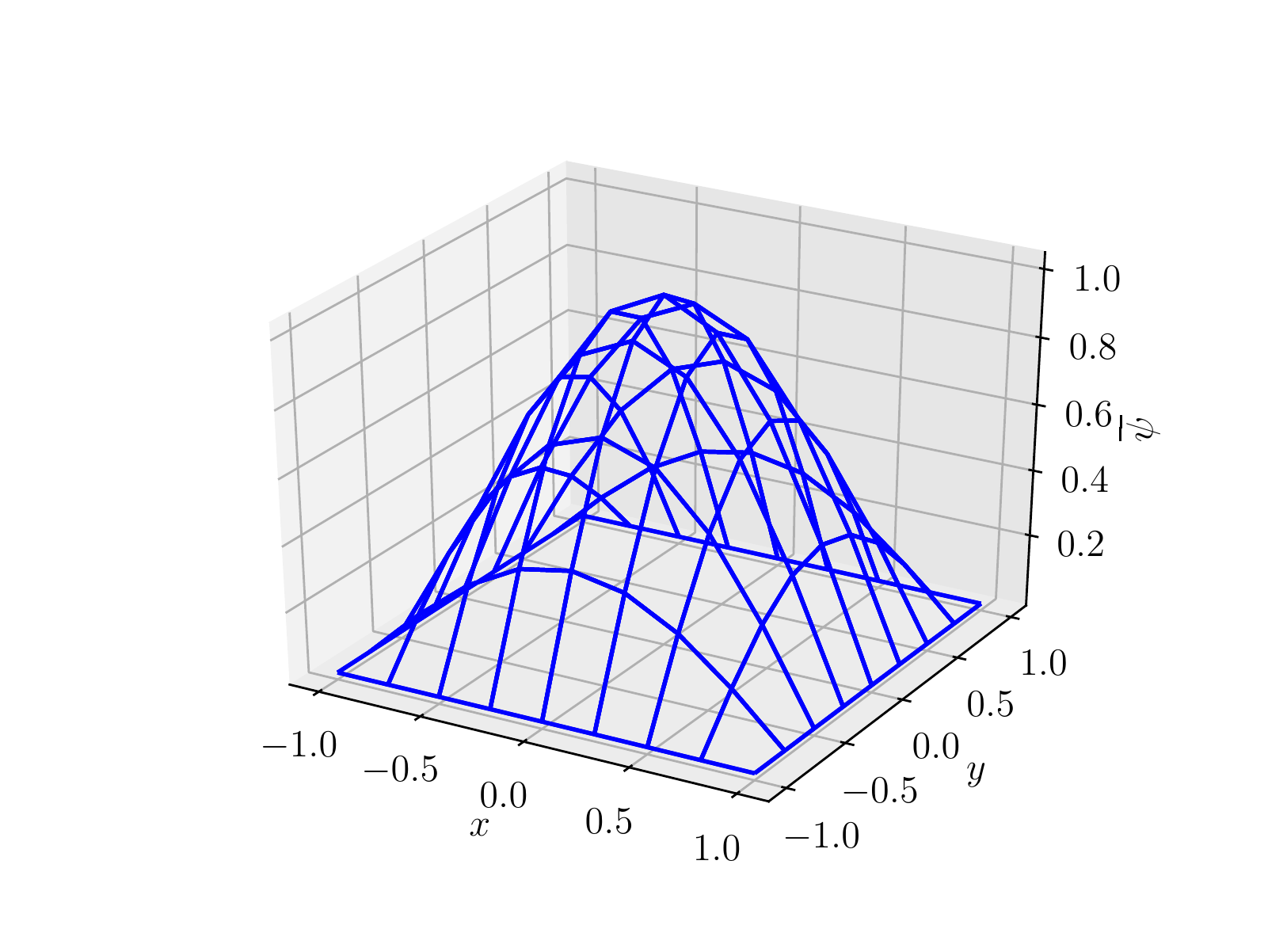}  
			\caption{LM, $\veps = 2^{-10}$.}
		\end{subfigure}
		~
		\begin{subfigure}[b]{0.23\textwidth}
			\includegraphics[width=\textwidth]{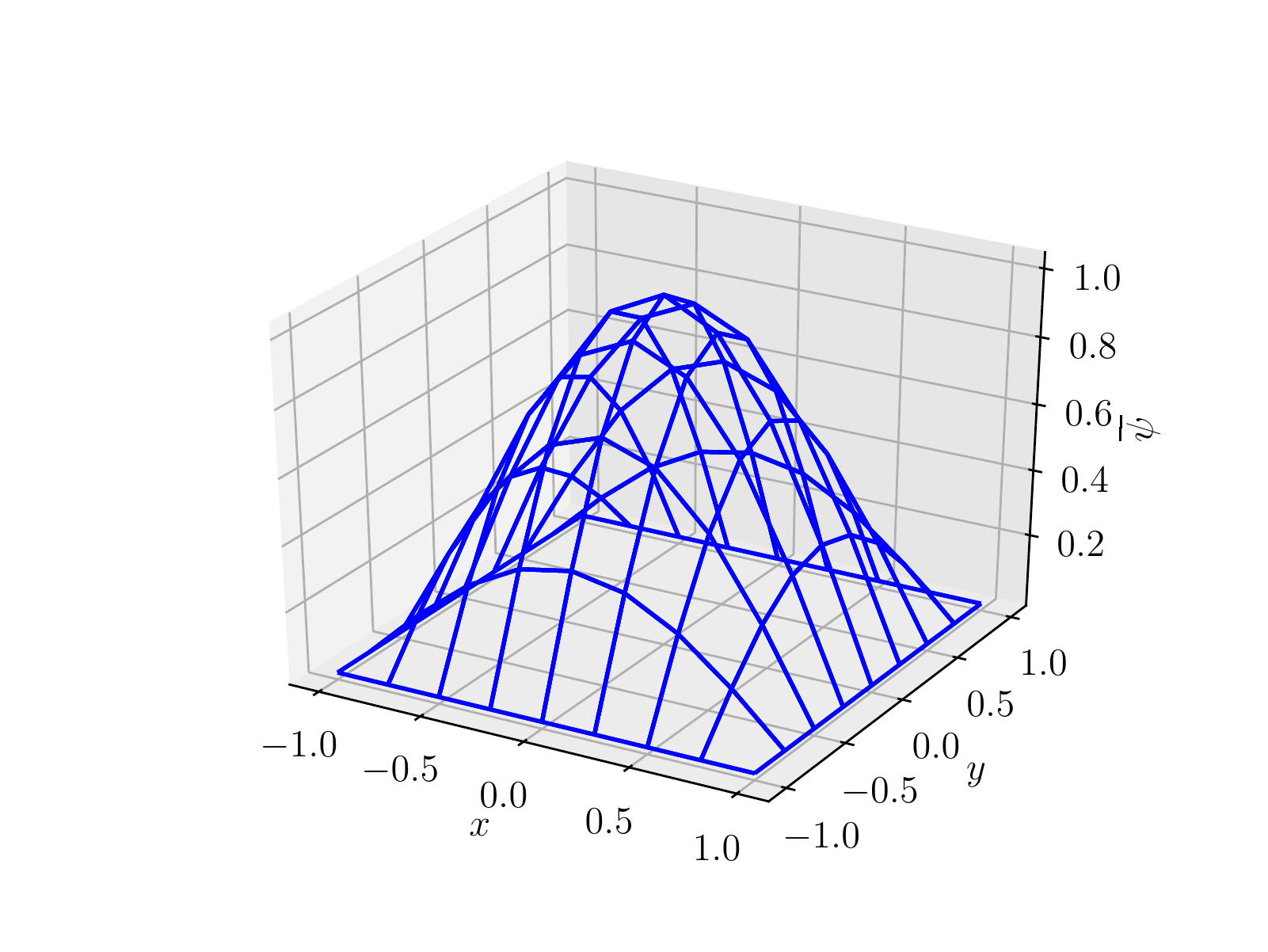}  
			\caption{LM, $\veps = 2^{-14}$.}
		\end{subfigure}\\
		\begin{subfigure}[b]{0.23\textwidth}
			\includegraphics[width=\textwidth]{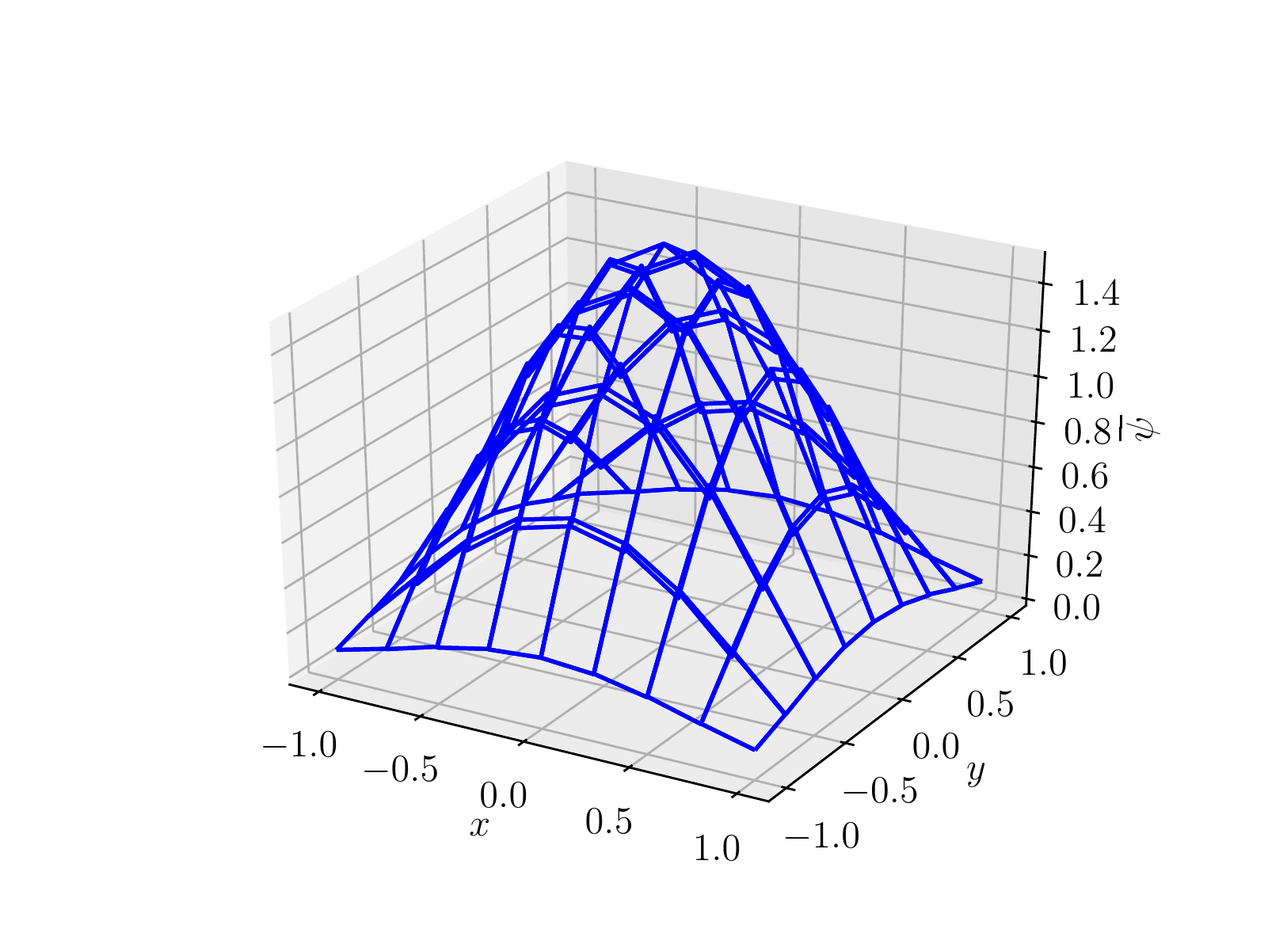}  
			\caption{RLM, $\veps = 1$.}
		\end{subfigure}
		~
		\begin{subfigure}[b]{0.23\textwidth}
			\includegraphics[width=\textwidth]{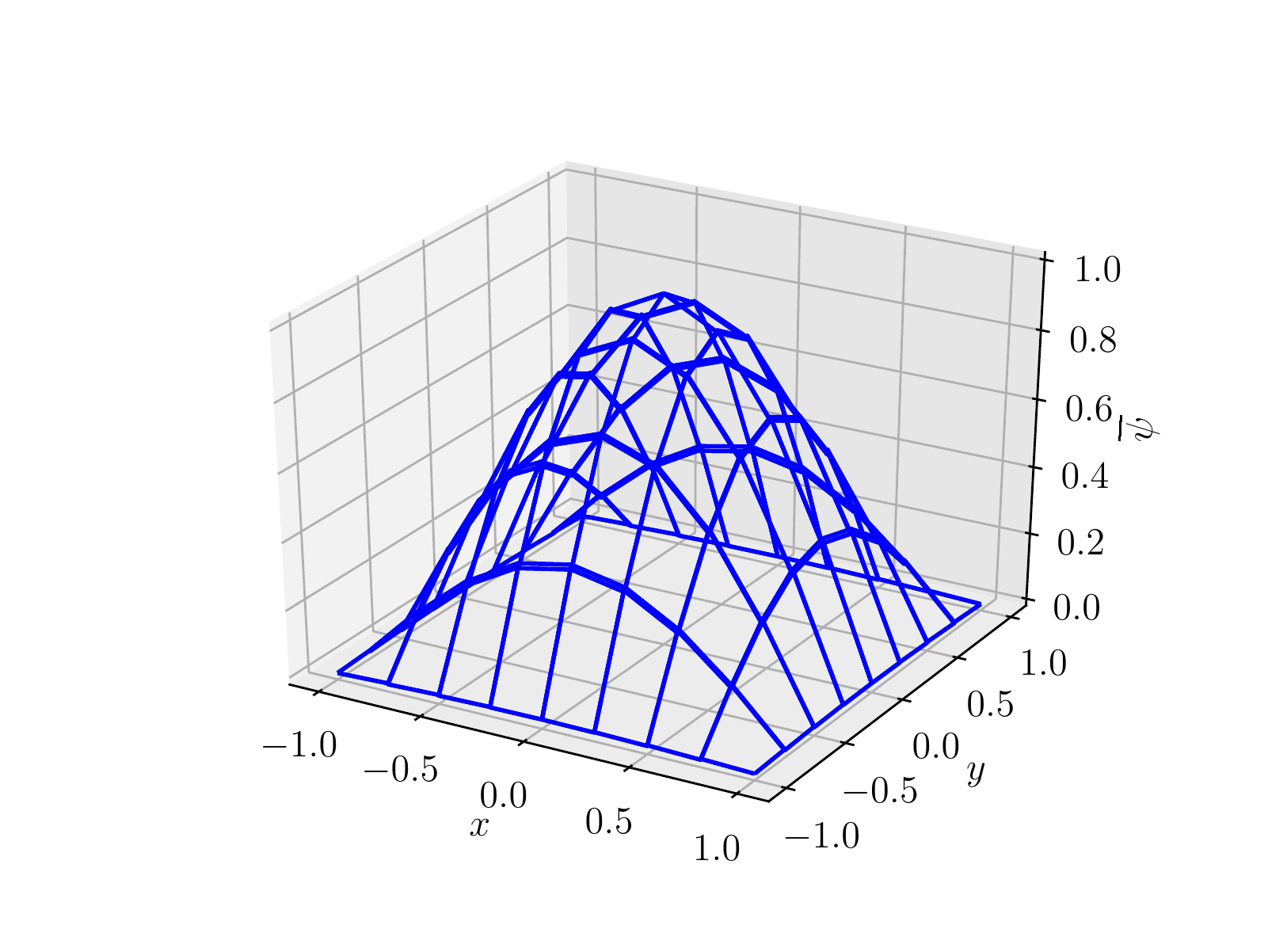}  
			\caption{RLM, $\veps = 2^{-6}$.}
		\end{subfigure}
		~
		\begin{subfigure}[b]{0.23\textwidth}
			\includegraphics[width=\textwidth]{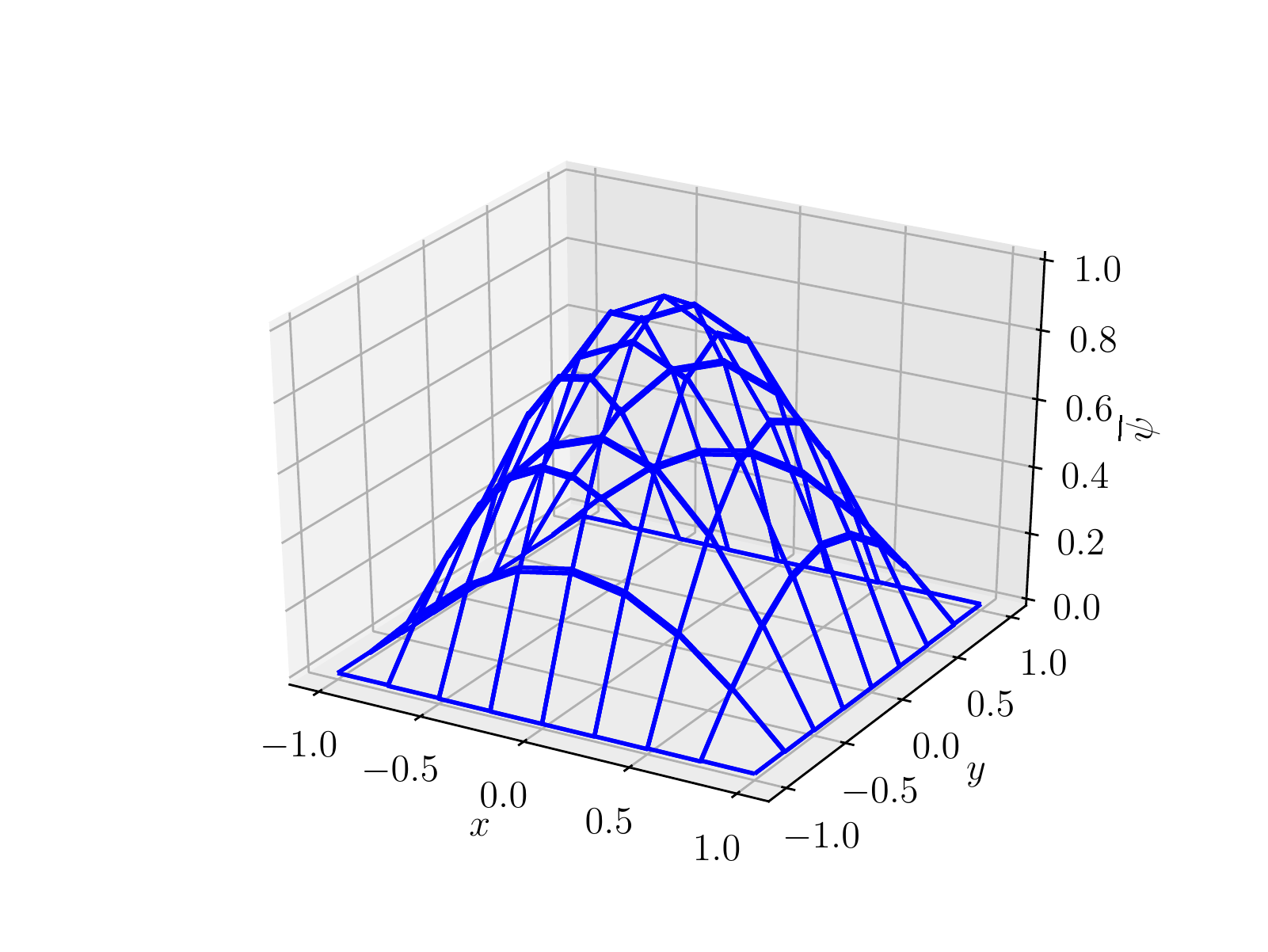}  
			\caption{RLM, $\veps = 2^{-10}$.}
		\end{subfigure}
		~
		\begin{subfigure}[b]{0.23\textwidth}
			\includegraphics[width=\textwidth]{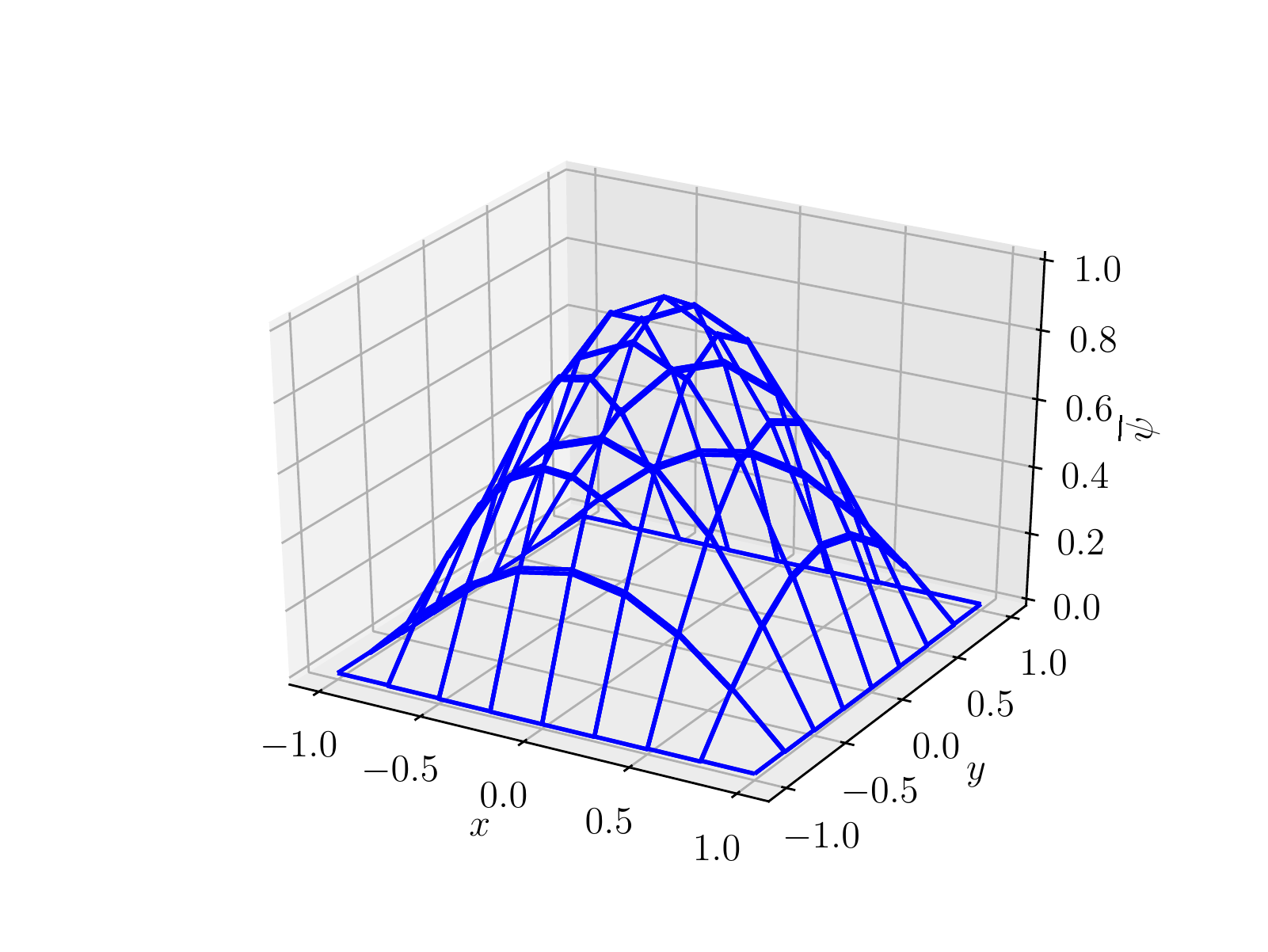}  
			\caption{RLM, $\veps = 2^{-14}$.}
		\end{subfigure}
		\caption{Profiles of numerical scalar fluxes in 
			\Cref{examp-2ddifflim}.}\label{fig-2d}
	\end{figure}
\end{example}

\section{Conclusions and future work}\label{sc-conclude}
\setcounter{equation}{0}
\setcounter{figure}{0}
\setcounter{table}{0}

In this paper, we study a class of low-memory $S_N$-DG methods for the radiative transport equation. In our first method, we use the variational form of the original $S_N$-DG scheme with a smaller finite element space, in which functions have isotropic slopes. This method preserves the asymptotic diffusion limit and can still be solved with sweeps. It is first-order accurate and exhibits second-order convergence rate near the diffusion limit. The second method is a correction of the first method with reconstructed slopes, which also preserves the diffusion limit and is second-order accurate in general settings (numerically). A summary of different methods and their properties can be found in \Cref{tab-compare}.

Future work will focus on the efficiency boost of the low-memory methods. Possible directions include: (i) further reducing degrees of freedom by enriching piecewise constant space only with continuous linear elements; (ii) developing preconditioners for linear systems; (iii) comparing numerical efficiency of the methods with different reconstruction approaches, including adaptivity.
\begin{table}[h!]
	\footnotesize
	\centering
	\begin{tabular}{c|c|c|c|c|c|c}
		\hline
		\multicolumn{2}{c|}{}&$P^0$-DG&$P^1$-DG&$Q^1$-DG&LMDG&RLMDG\\
		\hline
		\multicolumn{2}{c|}{Unisolvency when $\sig{a}\geq \delta_{\mathrm{a}} >0$}
		&\multicolumn{4}{c|}{Yes}&\multirow{3}{*}{\thead{Unknown.\\ Numeri-\\cally:\\ Yes}}\\
		\cline{1-6}
		\multirow{2}{*}{\thead{Preserves \\interior \\diffusion 
				limit}}&1D&\multirow{2}{*}{No}&\multicolumn{3}{c|}{Yes}&\\\cline{2-2}\cline{4-6}
		&2D& &\thead{Triangular: Yes \\Rectangular: No}&\multicolumn{2}{c|}{Yes}&\\
		\hline
		\multirow{2}{*}{\thead{Order of \\accuracy
		}}&isotropic&\multirow{2}{*}{1}&\multicolumn{2}{c|}{\multirow{2}{*}{2}}&2&\multirow{2}{*}{2}\\\cline{2-2}\cline{6-6}
		&anisotropic&&\multicolumn{2}{c|}{} &1&\\
		\hline
		\multirow{3}{*}{\thead{System \\dimension
		}}&1D&\multirow{3}{*}{$n_x$}&\multicolumn{4}{c}{$2n_x$}\\\cline{2-2}\cline{4-7}
		&2D&
		&{$3n_x$}&\multicolumn{3}{c}{$4n_x$}\\\cline{2-2}\cline{4-7}
		&3D&                      &{$4n_x$}&\multicolumn{3}{c}{$8n_x$}\\\cline{2-2}\cline{4-7}
		\hline
		\multirow{3}{*}{\thead{Solution \\dimension
		}}&1D&\multirow{3}{*}{$n_\Omega \cdot
			n_x$}&\multicolumn{2}{c|}{$2n_\Omega\cdot
			n_x$}&\multicolumn{2}{c}{$n_\Omega\cdot n_x+n_x$}\\\cline{2-2}\cline{4-7}
		&2D&
		&{$3n_\Omega\cdot n_x$}&$4n_\Omega\cdot
		n_x$&\multicolumn{2}{c}{$n_\Omega\cdot n_x+3n_x$}\\\cline{2-2}\cline{4-7}
		&3D&&{$4n_\Omega\cdot n_x$}&$8n_\Omega\cdot
		n_x$&\multicolumn{2}{c}{$n_\Omega\cdot n_x+7n_x$}\\
		\hline
	\end{tabular}
	\caption{Comparison of different methods.}\label{tab-compare}
\end{table}
\vspace{-0.5cm}
\section*{Acknowledgment}
ZS would like to thank Oak Ridge National Laboratory for hosting his NSF
internship and to thank the staff, post-docs, interns and other visitors at ORNL for their
warm hospitality.

\appendix
\section{Assembly of the matrices}\label{ap-sndg-matrix}
From the variational form \eqref{eq-vari-LSQ}, we can derive a matrix system
$\bL \bPsi = \bS\bPsi + \bQ$.
The matrices are defined as
$\bL = [L^{(l,p,r),(l',p',r')}]_{(n_\Omega\cdot n_x \cdot n_P)\times (n_\Omega\cdot n_x \cdot n_P)}$, 
$\bS = [S^{(l,p,r),(l',p',r')}]_{(n_\Omega\cdot n_x \cdot n_P)\times (n_\Omega\cdot n_x \cdot n_P)}$ and  \\
$\bQ = [Q^{(l,p,r)}]_{(n_\Omega\cdot n_x \cdot n_P)}$, where
\begin{align}
&    L^{(l,p,r),(l',p',r')} = L(\xi^{l',p',r'},\xi^{l,p,r})
= \delta_{ll'} w_l \sum_{K\in \cT_h}\left(-\int_K b^{p',r'}\Omega_l\cdot \nabla
b^{p,r}dx\right.\\ 
&
\left. + \int_{\partial K} \widehat{b}^{p',r'}\Omega_l\cdot \nu_K (b^{p,r})^{\mathrm{int}} dx+\int_K \left(\frac{\sigma_s}{\veps}+\veps\sigma_a\right)b^{p',r'}b^{p,r}dx\right),\nonumber\\
& S^{(l,p,r),(l',p',r')}= S(\xi^{l',p',r'},\xi^{l,p,r})=
w_lw_{l'} \sum_{K\in \cT_h}\int_K \frac{\sigma_s}{\veps} b^{p,r}b^{p',r'} dx,\\
&Q^{(l,p,r)}= Q(\xi^{l,p,r})=w_l\left( \sum_{K\in
	\cT_h}\int_K \veps q b^{p,r} dx - \sum_{K\in\cT_h}\int_{\partial K
	\cap F_h^\partial}\alpha 
\Omega_l \cdot \nu_K b^{p,r} dx\right).
\end{align}
Note that $\bS$ can be decomposed as the product of two matrices $\bS = \bM\bP$, where
$\bM = [M^{(l,p,r),(p'',r'')}]_{(n_\Omega\cdot n_x\cdot n_P)\times( n_x\cdot
	n_P)}$ with
$M^{(l,p,r),(p'',r'')} = w_l \sum_{K\in \cT_h}\int_K 
\frac{\sigma_s}{\veps} b^{p,r}b^{p'',r''} dx$,
and $\bP = [P^{(p'',r''),(l',p',r')}]_{(n_x\cdot n_P)\times (n_\Omega\cdot n_x\cdot n_P)}
$ with 
$P^{(p'',r''),(l',p',r')} = w_{l'}\delta_{p''p'}\delta_{r''r'}$.
Hence the matrix equation becomes
$\bL\bPsi = \bM \bP \bPsi + \bQ$.

\section{Proof of \Cref{thm-sndg-mat-psi}}\label{ap-sndg-mat-psi}
\begin{proof}
	Since the variational problem of the $S_N$-DG method is unisolvent, $\bL - \bM \bP$ is invertible. 
	To show $\bI_{n_x\cdot n_P}-\bP\bL^{-1}\bM$ is invertible, one only needs to check
	\begin{equation}(\bI_{n_x\cdot n_P}-\bP\bL^{-1}\bM)\bX = \mathbf{0} \Rightarrow \bX = \mathbf{0}.\end{equation}
	Indeed, with $(\bI_{n_x\cdot n_P}-\bP\bL^{-1}\bM)\bX = \mathbf{0}$, we have
	\begin{equation}
	\begin{aligned}
	(\bL -\bM \bP) (\bL^{-1}\bM\bX) = \mathbf{0}
	&\Rightarrow \bL^{-1}\bM\bX = \mathbf{0} \Rightarrow \bM\bX = \mathbf{0} \Rightarrow \bX = \mathbf{0}.
	\end{aligned}
	\end{equation}
	Hence $\bI_{n_x\cdot n_P}-\bP\bL^{-1}\bM$ is invertible.
\end{proof}

\section{Proof of \Cref{lem-B11}}\label{ap-B11}

\begin{proof}
	Note $\bB_{11}$ corresponds to the variational problem \eqref{eq-sndglm-scheme} with $\cW_h = \cV_{h,1}$. Since the variational problem is unisolvent (even when $\sig{a} = 0$), $\bB_{11}
	$ is invertible. 
	
	$\forall u,v\in \cD_{h,1}$, since $u = \overline{u}$, we have
	\begin{align}
	\sum_{j=1}^{n_\Omega} w_j \int_{ K} u_j \Omega_j\cdot\nabla v_j dx = \int_{K}
	\bar{u}
	\left(\sum_{j=1}^{n_\Omega}w_j\Omega_j \right)\cdot\nabla \bar{v} dx
	&=0,\\
	\sum_{j=1}^{n_\Omega}\sum_{K\in \cT_h} \int_{ K} \frac{\sig{s}}{\veps}
	(u_j-\overline{u}) v_j dx&= 0.\end{align}
	Therefore,
	\begin{equation}B_{11}(u,v) = 
	\sum_{K\in \cT_h} \int_{K} \veps\sig{a} \overline{u}\ 
	\overline{v} dx + 
	\sum_{j=1}^{n_\Omega} w_j \sum_{K\in \cT_h} \int_{\partial K} \Omega_j\cdot
	\nu_K \widehat{u}_j v_j^{\mathrm{int}} dx.\end{equation}
	We would like to write the last term as a summation with respect to edges.
	$\nu_{F}^+$ is defined as the unit normal of an edge $F$ such that $e_1 \cdot
	\nu_{F}^+>0$. $e_1$ is the vector in $\mathbb{R}^d$, whose first component is $1$ and others are $0$. Suppose $\nu_{F}^+$ is pointing from $K^+$ to $K^-$, we denote by $[v]_j =
	v_j|_{K^+}-v_j|_{K^-}$. Then 
	\begin{equation}
	\begin{aligned}\sum_{j=1}^{n_\Omega} w_j \sum_{K\in \cT_h} \int_{\partial K} \Omega_j\cdot
	\nu \widehat{u}_j v_j dx &= 
	\sum_{j=1}^{n_\Omega} w_j \sum_{F\in \cF_h} \int_{F} \Omega_j\cdot
	\nu_{F}^+ \widehat{u}_j [v]_j dx\\
	&=\sum_{\substack{j=1,\dots,n_\Omega,\\  \Omega_j\cdot e_1>0}} w_j \sum_{F\in \cF_h} \int_{F}| \Omega_j\cdot
	\nu_{F}^+| [u]_j [v]_j dx.
	\end{aligned}
	\end{equation}
	The last equality uses the central symmetry of the angular quadrature. 
	Hence
	\begin{equation}
	B_{11}(u,v) = 
	\sum_{K\in \cT_h} \int_{K} \veps\sig{a} \overline{u}\ 
	\overline{v} d\Omega_N dx + 
	\sum_{\substack{j=1,\dots,n_\Omega\\ \Omega_j\cdot e_1>0}} w_j \sum_{F\in \cF_h} \int_{F}| \Omega_j\cdot
	\nu_{F}^+| [u]_j [v]_j dx.
	\end{equation}
	Here $\Omega_j\cdot e_1$ gives the first component of $\Omega_j$. Since $B_{11}$ is a symmetric and positive semi-definite bilinear form,
	$\bB_{11}$ is then a symmetric and positive semi-definite matrix. The positive definiteness is implied by the fact  $\bB_{11}$ is invertible. 
\end{proof}
\section{Proof of \Cref{thm-K}}\label{ap-K}
We first prove the following lemma. 
\begin{LEM}\label{lem-key}
	Suppose $\bI_m - \bC \bA - \bD \bB $ is invertible, where $\bA$ is an $n\times
	m$ matrix, $\bB$ is an $n'\times m$ matrix, $\bC$ is an
	$m\times n$ matrix, and $\bD$ is an $m \times n'$ matrix. Then 
	$\bI_{n+n'} - \left[\begin{matrix}
	\bA\\\bB\end{matrix}\right]\left[\begin{matrix}\bC &\bD\end{matrix}\right]$ is invertible.
\end{LEM}
\begin{proof}
	It suffices to show that 
	$\left(\bI_{n+n'} - \left[\begin{matrix}
	\bA\\\bB\end{matrix}\right]\left[\begin{matrix}\bC
	&\bD\end{matrix}\right]\right)
	\left[\begin{matrix}
	\bX_0\\\bX_1\end{matrix}\right] =
	\mathbf{0}$ 
	implies $\bX_0$ and $\bX_1$ are $\mathbf{0}$. Indeed, with $\bX = \bC \bX_0 + \bD
	\bX_1$, the equality gives
	\begin{equation}\label{eq-keylem}
	\bA \bX = \bA\bC \bX_0+\bA\bD \bX_1 = \bX_0\quand \bB \bX = \bB\bC
	\bX_0+\bB\bD \bX_1 = \bX_1,
	\end{equation}
	which implies
	$(\bI_m-\bC\bA-\bD\bB)\bX = \bX-\bC \bX_0-\bD \bX_1 = \mathbf{0}$.
	Since $\bI_m-\bC \bA-\bD\bB$ is invertible, $\bX = \mathbf{0}$. Using \eqref{eq-keylem} we
	have $\bX_0 = \bA \bX = \mathbf{0}$ and $\bX_1 = \bB \bX = \mathbf{0}$.
\end{proof}

Then we show $\bK$ is invertible.
\begin{proof}
	One can see from \Cref{thm-gene-uni-solvency} that the variational problem \eqref{eq-gene-scheme} is uniquely solvable. Therefore the associated system 
	\eqref{eq-lmdg-psi0} is also unisolvent and $\bI_{n_\Omega\cdot n_x}- \bL_{00}^{-1} \bM_0 \bP_0 
	-\bL_{00}^{-1}  \bL_{01}\bSigma^T\bB_{11}^{-1} \bSigma\bL_{10}$ is invertible. 
	Taking $\bA = \bP_0$, $\bB =\bSigma \bL_{10}$, 
	$\bC = \bL_{00}^{-1}\bM_0$ and $\bD =
	\bL_{00}^{-1}\bL_{01}\bSigma^T\bB_{11}^{-1}$ in \Cref{lem-key}, one
	can show that the matrix 
	$\bI_{n_x \cdot n_P} -  
	\left[\begin{matrix}\bP_0\\ \bSigma \bL_{10}\end{matrix}\right]
	\left[\begin{matrix}\bL_{00}^{-1}\bM_0&
	\bL_{00}^{-1}\bL_{01}\bSigma^T\bB_{11}^{-1}
	\end{matrix}\right]
	$
	is invertible. Hence 
	$\bK = \left(\bI_{n_x\cdot n_P}-
	\left[\begin{matrix}\bP_0\\ \bSigma \bL_{10}\end{matrix}\right]
	\left[\begin{matrix}\bL_{00}^{-1}\bM_0&
	\bL_{00}^{-1}\bL_{01}\bSigma^T\bB_{11}^{-1}
	\end{matrix}\right]\right)
	\left[\begin{matrix}
	\bI_{n_x} &\\ & \bB_{11} \\
	\end{matrix}\right]$
	is also invertible.
\end{proof}

\bibliographystyle{siamplain}

\end{document}

%% file: ex_shared.tex
\usepackage{lipsum}
\usepackage{amsfonts}
\usepackage{graphicx}
\usepackage{epstopdf}
\usepackage{algorithmic}
\ifpdf
  \DeclareGraphicsExtensions{.eps,.pdf,.png,.jpg}
\else
  \DeclareGraphicsExtensions{.eps}
\fi

\newsiamremark{remark}{Remark}
\newsiamremark{hypothesis}{Hypothesis}
\crefname{hypothesis}{Hypothesis}{Hypotheses}
\newsiamthm{claim}{Claim}

\headers{Low-memory $S_N$-DG methods for
	radiative transport}{Z. Sun and C. D. Hauck}

\title{Low-memory, discrete ordinates, discontinuous Galerkin methods for
	radiative transport%
	\thanks{This material was based, in part, upon work supported by the DOE Office of Advanced Scientific Computing Research.  ORNL is operated by UT-Battelle, LLC., for the U.S. Department of Energy under Contract DE-AC05-00OR22725.
		This research is supported in part by an appointment with the NSF Mathematical Sciences
		Summer Internship Program sponsored by the National Science Foundation, Division of Mathematical
		Sciences (DMS). This program is administered by the Oak Ridge Institute for Science and Education (ORISE)
		through an interagency agreement between the U.S. Department of Energy (DOE) and NSF. ORISE is
		managed by ORAU under DOE contract number DE-SC0014664.
		The United States Government retains and the publisher, by accepting the article for publication, acknowledges that the United States Government retains a non-exclusive, paid-up, irrevocable, world-wide license to publish or reproduce the published form of this manuscript, or allow others to do so, for the United States Government purposes. The Department of Energy will provide public access to these results of federally sponsored research in accordance with the DOE Public Access Plan (\texttt{http://energy.gov/downloads/doe-public-access-plan}).}
}

\author{Zheng Sun\thanks{Department of Mathematics, The Ohio State University, Columbus, OH  43210, USA. (\email{ sun.2516@osu.edu}}) \and Cory D.
	Hauck\thanks{Computer Science and Mathematics Division, Oak Ridge National
		Laboratory, Oak Ridge, TN 37831, USA. (\email{hauckc@ornl.gov.}) }}

\usepackage{amsopn}
